\documentclass[a4paper,11pt]{amsart}

\usepackage[T1]{fontenc}
\usepackage{amsmath, amssymb, amsthm, mathtools}
\usepackage{xcolor}
\usepackage{graphicx}
\theoremstyle{definition}
\newtheorem{theorem}{Theorem}[section]
\newtheorem{lemma}[theorem]{Lemma}
\newtheorem{proposition}[theorem]{Proposition}
\newtheorem{proposition-definition}[theorem]{Proposition-Definition}

\newtheorem{corollary}[theorem]{Corollary}
\newtheorem{definition}[theorem]{Definition}

\newtheorem{example}[theorem]{Example}

\theoremstyle{definition}
\newtheorem{remark}[theorem]{Remark}

\newtheorem{claim}[theorem]{Claim}

\newtheorem{innercustomtheorem}{Theorem}
\newenvironment{customtheorem}[1]
  {\renewcommand\theinnercustomtheorem{#1}\innercustomtheorem}
  {\endinnercustomtheorem}

\newtheorem{innercustomcorollary}{Corollary}
\newenvironment{customcorollary}[1]
  {\renewcommand\theinnercustomcorollary{#1}\innercustomcorollary}
  {\endinnercustomcorollary}

\usepackage[pagebackref=true, colorlinks, citecolor=blue]{hyperref}

\usepackage[numbers,sort]{natbib}
\numberwithin{equation}{section}
\begin{document}

\title[Transcendence and measures]{Transcendence and measures via the refined Diophantine exponent}

\author{Quang-Khai Nguyen}
\address{Institut Camille Jordan, Universit\'e Claude Bernard Lyon 1 \newline \indent 21 avenue Claude Bernard, 69100 Villeurbanne, France}
\date{\today}
\makeatletter
\@namedef{subjclassname@2020}{\textup{2020} Mathematics Subject Classification}
\makeatother
\email{nguyen@math.univ-lyon1.fr}

\begin{abstract}
In 2007, Adamczewski and Bugeaud introduced the notion of the Diophantine exponent of an infinite word as a quantitative measure of repetition, leading to new transcendence results for real numbers whose expansions in an integer base are sufficiently simple. In the present article, we introduce the \emph{refined Diophantine exponent}, which detects weaker forms of repetition while preserving the full strength of the classical approach. This new exponent applies in situations where repetition is partially obscured by some form of \emph{noise}. Related ideas already appear in the work of Corvaja and Zannier in 2002 and, more recently, in the works of Kebis, Luca, Ouaknine, Scoones, and Worrell. Our approach provides a unified framework that recovers and extends these results, as well as those of Adamczewski and Bugeaud. We also develop quantitative refinements of this method, leading to results about transcendence measures. The recent breakthrough of Bell, Diller, and Jonsson in the context of algebraic dynamics is partly based on a similar idea, which also served as a motivation for the present work.

\end{abstract}
\maketitle
\tableofcontents
%riga
\section{Introduction}\label{section: Introduction}

The study of the Diophantine properties of values of analytic functions
\[
f_{\mathbf a}(z)=\sum_{i\ge 0} a_i z^i \in \overline{\mathbb Q}[[z]]
\]
at algebraic points is a central theme in transcendental number theory. A long-standing approach consists in exploiting combinatorial properties of the coefficient sequence $\mathbf a=a_0a_1\cdots$ to investigate the transcendence of the values of $f_{\mathbf a}$, as well as more quantitative questions such as irrationality and transcendence measures. The underlying philosophy is that when the sequence $\mathbf a$ is \emph{sufficiently simple}, in the sense that it has a simple combinatorial or dynamical structure, it can often be well approximated by eventually periodic sequences. Consequently, the function $f_{\mathbf a}$ can be well approximated by rational functions, which in turn implies that its values at algebraic points admit good algebraic approximations. From the perspective of Diophantine approximation, this typically leads to a \emph{dichotomy}: such values are either transcendental, or algebraic of a very restricted form.

This circle of ideas goes back to the pioneering work of Liouville, who proved the transcendence of numbers such as
\[\sum_{i\ge 1}\frac{1}{10^{i!}}\cdot \]
Later, deeper tools from Diophantine approximation led to further classes of transcendental numbers obtained in this way. Particularly significant are the works of Maillet~\cite{Maillet-1906} and Baker~\cite{Baker-1964} on continued fractions, and, a few decades later, the work of Ferenczi and Mauduit~\cite{Ferenczi-Mauduit-1997-LowComplexity} on Sturmian numbers. This approach acquired a new dimension in the works of Adamczewski and Bugeaud~\cite{Adamczewski-Bugeaud-2007-dynamics,Adamczewski-Bugeaud-2007}, and of Adamczewski, Bugeaud, and Luca~\cite{Adamczewski-Bugeaud-Luca-2004}, where the full strength of the $p$-adic Subspace Theorem~\cite{Schlickewei-1976-PadicThueSiegelRothSchmidt,Schlickewei-1976-LinearForms}, which is a powerful extension of Schmidt's Subspace Theorem~\cite{Schmidt-1972}, was brought into the picture.

The general strategy was formalized by Adamczewski and Bugeaud in 
\cite{Adamczewski-Bugeaud-2007-dynamics} through the introduction of the notion of the \emph{Diophantine exponent}, which measures the periodicity of an infinite word. 
The Diophantine exponent of a word $\mathbf{a}$, denoted $\mathbf{Dio}(\mathbf{a})$, is defined as the supremum of the real numbers $\rho$ for which there exist three sequences of integers $(r_n)_{n\geq 0}$, $(s_n)_{n\geq 0}$, and $(t_n)_{n\geq 0}$ satisfying, for all $n\geq 0$,
\begin{itemize}
    \item[(i)] $-1\leq r_n<s_n<t_n$;
    \item[(ii)]
    the words $\mathbf{a}[r_n+1, r_n+t_n-s_n]$ and $\mathbf{a}[s_n+1, t_n]$ coincide;
   \item[(iii)]  $t_n \geq \rho s_n$;
   \item[(iv)] 
     $s_n - r_n \to \infty$ as $n \to \infty$.
\end{itemize}
Here, $\mathbf{a}[r,s]$ denotes the finite word $a_r a_{r+1} \cdots a_s$. 
The main result of~\cite{Adamczewski-Bugeaud-Luca-2004,Adamczewski-Bugeaud-2007-dynamics,Adamczewski-Bugeaud-2007} asserts that if $\mathbf{a}$ takes values in a finite set of integers and if $\beta$ is an algebraic number with $|\beta|>1$ such that 
\[
\mathbf{Dio}(\mathbf{a}) > \frac{\log M(\beta)}{\log |\beta|},
\]
then the following dichotomy holds:
\begin{itemize}
    \item either $f_{\mathbf{a}}(1/\beta)$ belongs to the number field $\mathbb{Q}(\beta)$, or it is transcendental.
\end{itemize}
Here $M(\beta)$ denotes the Mahler measure of $\beta$. The original definition of the Diophantine exponent is slightly different, but easily seen to be equivalent to the one given above. 
Setting $U_n=\mathbf{a}[0,r_n]$ and $V_n=\mathbf{a}[r_n+1,s_n]$, Condition~(ii) implies that $\mathbf{a}$ coincides with the eventually periodic sequence $U_n V_n^{\infty}$ up to index $t_n$. 
When $\beta$ is an integer, or more generally a Pisot or a Salem number, the above result only requires that $\mathbf{Dio}(\mathbf{a})>1$. However, stronger repetitivity is needed to treat values of $f_{\mathbf{a}}$ at more general algebraic points. In particular, in order to treat values at \emph{all} algebraic numbers, one needs to assume that $\mathbf{Dio}(\mathbf{a})=\infty$, which is a very strong condition.

In retrospect, a way to overcome this difficulty takes its source in a previous work by Corvaja and Zannier~\cite{Corvaja-Zannier-2002} on lacunary numbers, by applying the Subspace Theorem in higher dimensions. 
More recently, similar ideas have appeared, apparently independently, in several works 
\cite{Luca-Ouaknine-Worrell-2025,Luca-Ouaknine-Worrell-2023,Kebis-Luca-Ouaknine-Scoones-Worrell-2024,Kebis-Luca-Ouaknine-Scoones-Worrell-2025}, 
where new combinatorial transcendence criteria are established for some class of coefficient sequences called \emph{echoing} and \emph{stuttering} over arbitrary algebraic bases. 
Such ideas also arise in the breakthrough work of Bell, Diller, and Jonsson~\cite{BDJ-2020} in the context of algebraic dynamics, where the authors construct the first example of a map with a transcendental dynamical degree. 

The present paper is motivated partly by these recent developments. 
We introduce a new exponent, called the \emph{refined Diophantine exponent} and denoted by $\mathbf{Rdio}$, which is defined in the same spirit as the classical Diophantine exponent, but with Condition~(ii) replaced by a weaker requirement (ii'): instead of asking the two words $\mathbf{a}[r_n+1,r_n+t_n-s_n]$ and $\mathbf{a}[s_n+1,t_n]$ to coincide, we only require them to be sufficiently close in a precise sense related to the Hamming distance as follows.

To define the {refined Diophantine exponent}, we need to introduce the notion of \emph{$(\epsilon,\delta)$-closeness}.

\begin{definition} \label{Definition: epsilon delta close}
    For two finite words $U=u_1\cdots u_L$ and $V=v_1\cdots v_L$ of the same length $L>0$ and for $\epsilon>0$ and integer $\delta\geq0$, we say that $U$ and $V$ are \emph{$(\epsilon,\delta)$-close} if there exist $\delta$ subintervals $I_1,\ldots,I_\delta$ of $\{1,\ldots,L\}$ such that
\begin{itemize}
    \item[(i)]  $U$ and $V$ agree outside $\bigcup_{j=1}^\delta I_{j}$, more precisely, \[\left\{1\leq i\leq L:u_i\neq v_i\right\}\subseteq \bigcup_{j=1}^\delta I_{j};\]  
    \item[(ii)] $\left|\bigcup_{j=1}^\delta I_{j}\right|\leq \epsilon L$.
    
\end{itemize}
Here, we note that the intervals $I_j$ can be empty.
\end{definition}

\begin{definition}\label{Definition: refined Dio}
    Let $\mathbf{a}=a_0a_1\cdots$ be an infinite word. Let
$\rho\geq1$ be a real number. We say that $\mathbf{a}$ satisfies Condition $(*)_\rho$ if for every $\epsilon>0$, there exist three sequences of integers $(r_n)_{n\geq 0}$, $(s_n)_{n\geq 0}$, $(t_n)_{n\geq 0}$ and an integer $\delta\geq0$ satisfying, for all $n\geq 0$,
     \begin{itemize}
     \item[(i)] $-1\leq r_n<s_n<t_n$;
     \item[(ii')]   for all $n$ large enough, $\mathbf{a}[r_n+1,r_n+t_n-s_n]$ is $(\epsilon,\delta)$-close to $\mathbf{a}[s_n+1, t_n]$;
     \item [(iii)] 
     ${t_n}\geq\rho{s_n}$;
         \item[(iv)]
         $s_n - r_n \to \infty$ as $n \to \infty$.
     \end{itemize}
    We define  $\mathbf{Rdio}(\mathbf{a})$  to be the supremum of the real numbers $\rho$ for which $\mathbf{a}$ satisfies
Condition $(*)_\rho$.
\end{definition}
The mismatches between these two words may be viewed as a form of \emph{noise} that partially obscures the periodicity of $\mathbf{a}$.  It follows from this definition that $1\leq \mathbf{Dio}(\mathbf{a}) \leq \mathbf{Rdio}(\mathbf{a})\leq\infty$. 

Our first main result, Theorem~\ref{Theorem: dichotomy}, is the analogue, for this new exponent, of the main result of~\cite{Adamczewski-Bugeaud-2007-dynamics}. 
It also allows for more general coefficient sequences and for the use of arbitrary absolute values. 
In this sense, it expresses the fact that such noise can, to a certain extent, be ignored from the Diophantine point of view. 
On the one hand, in all the aforementioned works 
\cite{Corvaja-Zannier-2002,BDJ-2020,Luca-Ouaknine-Worrell-2025,Luca-Ouaknine-Worrell-2023,Kebis-Luca-Ouaknine-Scoones-Worrell-2024,Kebis-Luca-Ouaknine-Scoones-Worrell-2025}, we show that the corresponding refined Diophantine exponent is infinite, so Theorem~\ref{Theorem: dichotomy} recovers all these results. 
On the other hand, our theorem also applies to coefficient sequences whose refined Diophantine exponent may be finite, thereby recovering the results of 
\cite{Adamczewski-Bugeaud-Luca-2004,Adamczewski-Bugeaud-2007-dynamics,Adamczewski-Bugeaud-2007}.
\begin{customtheorem}{A}\label{Theorem: dichotomy}
Let $K$ be a number field, and $ \mathsf{v}$ be a place on $K$. Let $\mathcal{S}$ be a finite set of places containing all archimedean places. Let $\beta\in K$ such that $|\beta|_\mathsf{v}>1$. Let $\mathbf{a}=a_0a_1\cdots$ be an infinite word  over the ring of $\mathcal{S}$-integers $\mathcal{O}_{K,\mathcal{S}}$ such that $\mathrm{h}(a_i)=o(i)$ and $$\mathbf{Rdio}(\mathbf{a})>\frac{\log \mathrm{H}(\beta)}{\log |\beta|_\mathsf{v}}\cdot$$ Then the number $\xi=\sum_{i\geq0}{a_i}{\beta^{-i}}\in K_\mathsf{v}$ either  lies in $K$ or is transcendental.
\end{customtheorem}
In this statement, we denote by $|\cdot|_\mathsf{v}$ the {normalized absolute value} and by $K_\mathsf{v}$ the completion of $K$ with respect to $\mathsf{v}$.  Our notion of convergence refers to $K_\mathsf{v}$. We write $\mathrm{H}(x)$ for the absolute Weil height, and $\mathrm{h}(x)$ for the absolute logarithmic Weil height (see \S\ref{subsubsection: Diophantine prerequisites}  for definitions).  

The proof of Theorem~\ref{Theorem: dichotomy} follows from the Subspace Theorem by taking into account the mismatches as new variables. This allows us to obtain linear relations between $\xi$, its rational approximations, and the mismatches. An \emph{ad hoc} argument then yields that the corresponding coefficients of the mismatches must be zero, and the desired conclusion follows. The proof in fact shows that if $\xi\in K$, then $\xi$ must be of a \emph{very restricted} form.

Our second goal is to develop a quantitative version of our approach in order to obtain transcendence measures. 
Adamczewski and Bugeaud observed that whenever the Subspace Theorem can be used to prove the transcendence of a number, it is in principle possible to derive a transcendence measure by means of the Quantitative Subspace Theorem due to Evertse~\cite{Evertse-1996}. 
In particular, one may hope to show that such numbers are either $S$- or $T$-numbers, in the sense of Mahler's classification, provided that the underlying approximations are \emph{sufficiently dense}. 
In the setting of the Diophantine exponent, this typically requires both that the sequence $(t_{n+1}/t_n)$ remains bounded and that the Diophantine exponent is finite. 
This approach has been systematically developed in the seminal works 
\cite{Adamczewski-Bugeaud-2010-LMS,Adamczewski-Bugeaud-2010-JEMS,Adamczewski-Bugeaud-2011}; 
see also~\cite{Kekec-2013,Bugeaud-Kekec-2018,Bugeaud-Kekec-2020} for $p$-adic analogues. We implement this strategy in our setting, which leads to Theorem~\ref{Theorem: transcendence measure of refined Dio}. We need the following assumption. 
\begin{definition}\label{Definition: (**)rho}
    Let $\mathbf{a}=a_0a_1\cdots$ be an infinite word, and
$\rho\geq1$ be a real number. We say that $\mathbf{a}$ satisfies Condition $(**)_\rho$ if for every $\epsilon>0$, there exist sequences of integers $(r_n)_{n\geq 0}$, $(s_n)_{n\geq 0}$, $(t_n)_{n\geq 0}$ and an integer $\delta\geq0$ satisfying the conditions (i), (ii'), (iii), (iv) of $(*)_\rho$ and \[{r_n}\ll{s_n-r_n},\limsup\frac{t_{n}}{s_n}<\infty,\limsup\frac{t_{n+1}}{t_n}<\infty\]
where the implied constants do not depend on $n$.
\end{definition}
With this condition, our second result is stated as follows.
\begin{customtheorem}{B}\label{Theorem: transcendence measure of refined Dio}

Let $K$ be a number field, and $ \mathsf{v}$ be a place on $K$. Let $\mathcal{S}$  be a finite set of places containing all archimedean places. Let $\beta\in K$ such that $|\beta|_\mathsf{v}>1$. Let $\mathbf{a}=a_0a_1\cdots$ be an infinite word  over $\mathcal{O}_{K,\mathcal{S}}$ such that $\mathrm{h}(a_i)=o(i)$ and $\mathbf{a}$ satisfies $(**)_\rho$ for some $\rho>\frac{\log \mathrm{H}(\beta)}{\log |\beta|_\mathsf{v}}$. Then there exist constants $c,H>0$ independent of $d$ such that for every $d\geq1$, the following holds: For all algebraic numbers $\alpha\not\in K$ of degree at most $d$ and of height $\mathrm{H}(\alpha)$ at least $H$, we have 
\[|\xi-\alpha|_\mathsf{v}\geq\mathrm{H}(\alpha)^{-(2d)^{c(\log4d)(\log\log4d)}}.\]  
\end{customtheorem}
Theorem~\ref{Theorem: transcendence measure of refined Dio} is proved by quantifying the proof of Theorem~\ref{Theorem: dichotomy}, namely by replacing $\xi$ with $\alpha$ and applying the Quantitative Subspace Theorem. 
The condition $\alpha\not\in K$ cannot be ignored due to the existence of \emph{exceptionally good approximations} of $\xi$ by elements in $K$.

Theorem~\ref{Theorem: transcendence measure of refined Dio} yields a new trichotomy with respect to Mahler's classification. Let us first recall this classification for complex and $p$-adic numbers,  following~\cite{Bugeaud-ApproximationbyAlgebraicNumbers}. In the following, we denote by $\mathrm{H}_\mathrm{naive}(P)$ the naive height of the polynomial $P(X)\in\mathbb Z[X]$, that is, the maximum of the absolute values of its coefficients. For $x\in\overline{\mathbb Q}$, we denote by $\mathrm{H}_\mathrm{naive}(x)$ the naive height of $x$, which is the naive height of its minimal polynomial over $\mathbb Z$.

Let $\xi\in\mathbb C$ (resp. $\xi\in K_\mathsf{v}$ for a non-archimedean place $\mathsf{v}$ on some number field $K$). For every positive integer $d$, we denote by $\omega_d(\xi)$ the supremum of the positive numbers $\omega$ for which the inequality
\[0<|P(\xi)|\leq \mathrm{H}_\mathrm{naive}(P)^{-\omega} (\text{resp. }0<|P(\xi)|_\mathsf{v}\leq \mathrm{H}_\mathrm{naive}(P)^{-\omega-1})\]
has infinitely many solutions in the set of polynomials $P(X)\in\mathbb Z[X]$ of degree at most $d$. In addition, we set \[\omega(\xi)=\limsup_{d\to\infty}\frac{\omega_d(\xi)}{d}\cdot\] It is known that $\omega_d(\xi)\in[0,\infty]$ for all $d>0$ and $\omega(\xi) \in[0,\infty]$.  We say that $\xi$ is a
\begin{itemize}
    \item {\it ($p$-adic) $A$-number} if $\omega(\xi)=0$;
    \item {\it ($p$-adic) $S$-number} if $0<\omega(\xi)<\infty$;
    \item {\it ($p$-adic) $T$-number} if $\omega(\xi)=\infty$ and $\omega_d(\xi)<\infty$ for all positive integers $d$;
    \item {\it ($p$-adic) $U$-number} if $\omega(\xi)=\infty$ and $\omega_d(\xi)=\infty$ for some positive integer $d$.
\end{itemize}
The class of $U$-numbers can be divided into smaller classes as follows.  Let $d$ be a positive integer. We call $\xi$ a
\begin{itemize}
    \item  {\it ($p$-adic) $U_d$-number} if $\omega_{\ell}(\xi)<\infty$ for $1\leq \ell<d$  and $\omega_d(\xi)=\infty$.
\end{itemize}

The class of $A$-numbers is exactly the class of algebraic numbers. In $\mathbb C$,  a $U_1$-number is also called a Liouville number, and the exponent $\omega_1(\xi)+1$ is also called \emph{irrationality exponent}. The key feature of Mahler's classification is that two complex numbers belonging to two different classes are automatically algebraically independent, see~\cite[Chapters~3.2 and~9.1]{Bugeaud-ApproximationbyAlgebraicNumbers}.
For its $p$-adic counterpart,  we refer to~\cite[Chapter~9.3]{Bugeaud-ApproximationbyAlgebraicNumbers}\footnote{We note that although the results in \textit{loc.~cit.} are stated only for $\mathbb Q_p$, the proofs apply \textit{verbatim} to any $K_\mathsf{v}$. }. 

Using Mahler’s classification, Theorem~\ref{Theorem: transcendence measure of refined Dio} yields the following trichotomy.
\begin{customcorollary}{B.1}\label{Corollary: transcendence measure of refined Dio}
    Under the assumptions of Theorem~\ref{Theorem: transcendence measure of refined Dio}, we have that  $\xi$ lies in $K$, or $\xi$ is a ($p$-adic) $U_{ d}$-number for some $1\leq d\leq [K:\mathbb Q]$, or there exists a constant $c>0$ independent of $d$ such that
\[\omega_d(\xi)\leq (2d)^{c(\log4d)(\log\log4d)}\]
for all $d\geq1$; in the last case, $\xi$ is a ($p$-adic) $S$- or a ($p$-adic) $T$-number. In other words, if $\xi\not\in K$, 
then $\xi$ is transcendental with \[\omega_d(\xi)\leq\max\{\omega_1(\xi),\ldots,\omega_{[K:\mathbb Q]}(\xi), (2d)^{c(\log4d)(\log\log4d)}\}\text{ for all }d\geq1.\]
\end{customcorollary}
% Recall that in the definition of the refined Diophantine exponent, we need the existence of three sequences $(r_n)_{n\geq0}, (s_n)_{n\geq0}, (t_n)_{n\geq0}$ that satisfy certain asymptotic conditions.
% To obtain transcendence measures, we need to assume extra conditions on these sequences as follows. 
% % (Theorem~\ref{Theorem: transcendence measure result})

% Using an absolute Weil height interpretation of Mahler’s classification (cf.~Theorem~\ref{compare two classifications} and Proposition~\ref{Proposition: compare two classifications}) and the Northcott property, Theorem~\ref{Theorem: transcendence measure of refined Dio} implies that either

\begin{remark}
    We provide examples (see Example~\ref{example: thue morse} and Remark~\ref{remark: fibonacci}) for which Theorem~\ref{Theorem: dichotomy}, Theorem~\ref{Theorem: transcendence measure of refined Dio} and Corollary~\ref{Corollary: transcendence measure of refined Dio} can be applied, while the results in~\cite{Adamczewski-Bugeaud-Luca-2004,Adamczewski-Bugeaud-2007-dynamics,Adamczewski-Bugeaud-2007,Adamczewski-Bugeaud-2011,Luca-Ouaknine-Worrell-2023,Kebis-Luca-Ouaknine-Scoones-Worrell-2025} cannot be used directly.
\end{remark}

At this level of generality, however, it remains difficult to obtain explicit transcendence measures, since it is notoriously hard to exclude the possibility that the number $f_{\mathbf a}(1/\beta)$ either belongs to $K$ or admits exceptionally good approximations by elements of $K$. In Theorems~\ref{Theorem: trans meas of lacunary} and~\ref{theorem: transcendence measure of certain numbers}  below, we identify situations in which this difficulty can be overcome; these results respectively concern lacunary sequences, Sturmian words, and $k$-bonacci words.

Corvaja and Zannier~\cite{Corvaja-Zannier-2002} proved that lacunary numbers (see Definition~\ref{Definition: lacunary numbers}) are always transcendental. Using the method from the proof of Theorem~\ref{Theorem: transcendence measure of refined Dio}, we provide a characterization of when these numbers are ($p$-adic) $S$- or $T$-numbers.
\begin{customtheorem}{C}\label{Theorem: trans meas of lacunary}
Let $K$ be a number field and $\mathsf{v}$ be a place on $K$. Let $(u_i)_{i\geq0}$ be an increasing sequence of non-negative integers such that $\liminf \frac{u_{i+1}}{u_i}>1.$ Let $\beta\in K$ such that $|\beta|_\mathsf{v}>1$, and non-zero elements $a_0,a_1,\ldots\in K$ such that $\mathrm{h}(a_i)=o(u_i)$. We set 
$\xi=\sum_{i\geq0}{a_i}{\beta^{-u_i}}\in K_\mathsf{v}.$ The following dichotomy holds:
\begin{itemize}
\item[(i)] Either $\limsup\frac{u_{i+1}}{u_i}$ is infinite and $\xi$ is a ($p$-adic) $U_d$-number for some $d\leq [K:\mathbb Q]$; 
\item[(ii)] Or $\limsup\frac{u_{i+1}}{u_i}$ is finite and there exists $c>0$ independent of $d$ such that $\omega_d(\xi)\leq (2d)^{c(\log\log 4d)}$ for all $d\geq1$. In this case, $\xi$ is a ($p$-adic) $S$- or $T$-number.
\end{itemize}
\end{customtheorem}
The main difficulty in Theorem~\ref{Theorem: trans meas of lacunary} lies in excluding $U$-numbers. The sparseness of non-zero coefficients controls the exceptionally good approximations from Theorem~\ref{Theorem: transcendence measure of refined Dio} and allows us to drop the $\mathcal{S}$-integer assumption and to improve the upper bound for $\omega_d(\xi)$. Note that under the stronger assumption that the coefficients are $\mathcal{S}$-integers, the conclusion would follow from the proof of Theorem~\ref{Theorem: transcendence measure of refined Dio}. However, since we are dealing with a more general case, we must modify some of its estimates.

Theorem~\ref{Theorem: trans meas of lacunary} provides a generalization of~\cite[Th\'eor\`eme~6.1]{Adamczewski-Bugeaud-2011} to arbitrary number fields.  Note, however, that there are only a few cases where one can determine whether a number is an $S$-number. For instance, the automatic number $\sum_{i\geq0}2^{-2^i}$ is an $S$-number, as established by Galochkin~\cite{Galochkin-1980} using Mahler's method. It is conjectured by Becker that all irrational automatic numbers are $S$-numbers, see~\cite[Conjecture~4.1]{Adamczewski-Bugeaud-2011}.

We conclude the article by establishing transcendence measures for numbers generated by $k$-bonacci words and Sturmian words. Recall that a Pisot (resp. Salem) number is a real algebraic integer greater than 1, whose Galois conjugates lie
inside the open unit disk (resp. inside the closed unit disk, with at least one of them
on the unit circle). Particularly, every integer larger than 1 is a Pisot number.
\begin{customtheorem}{D} \label{theorem: transcendence measure of certain numbers} 
     Let $\beta$ be a Pisot number of degree at most 2. Then in any of the following cases of $\mathbf{a}$:
    \begin{itemize}
    \item[(i)] $\mathbf{a}$ is a $k$-bonacci word over $\{0,1,\ldots,k-1\}$ with $ k-1\leq \beta$;
        \item[(ii)] $\mathbf{a}$ is a Sturmian word over $\{0,1\}$ whose slope has bounded partial quotients;
    \end{itemize}
   the real number $\xi=\sum_{i\geq0}a_i\beta^{-i}$ is an $S$- or $T$-number.
\end{customtheorem}

The proof of Theorem~\ref{theorem: transcendence measure of certain numbers} follows from the ideas of Baker~\cite{Baker-1964} and Adamczewski and Cassaigne~\cite{Adamczewski-Cassaigne-2006-Compositio} in constructing \emph{dense} rational approximations of $\xi$. In fact, this allows us to establish a stronger result: if $\beta$ is a Pisot or Salem number of arbitrary degree, then $\xi$ is neither a $U_1$- nor a $U_2$-number. This follows from the fact that $\mathbf{Dio}(\mathbf{a}) > 2$ for both choices of $\mathbf{a}$. However, it remains unclear how to show that $\xi$ is not a $U_d$-number for $3 \leq d \leq [\mathbb{Q}(\beta):\mathbb{Q}]$ in general, since $\mathbf{Dio}(\mathbf{a})$ might be less than $3$ (which is always the case for $k$-bonacci words).
\begin{remark}
We observe that if $\mathbf{a}$ is a Sturmian word whose slope has unbounded partial quotients, then its Diophantine exponent is infinite by~\cite[Proposition~11.1]{Adamczewski-Bugeaud-2011}. In this case, Proposition~\ref{Lemma: Dio infinite implies U number} implies that $\xi$ is a $U_d$-number for some $1\leq d\leq [\mathbb Q(\beta):\mathbb Q]$. Theorem~\ref{theorem: transcendence measure of certain numbers}, combined with this observation, can be seen as a quadratic extension of~\cite[Th\'eor\`eme~3.1]{Adamczewski-Bugeaud-2011}.
\end{remark}

% that $\omega_d(\xi)\leq (2d)^{c(\log 4d)(\log\log 4d)}$ 

Finally, we recall that there exists another classical general method to study the Diophantine properties of values of analytic functions at algebraic points. 
It is based on the construction of Pad\'e approximants and requires that the functions under consideration satisfy suitable \emph{functional equations}. 
This approach was first introduced in the context of transcendence in the celebrated memoir of Hermite~\cite{Hermite-1873}, where he established the transcendence of $e$.   It has since been developed in various settings, notably for Siegel $E$-functions and Mahler $M$-functions (see, for instance,~\cite{Adamczewski-Faverjon-2020}).  
One major advantage of this method is that it allows one to address questions of linear and algebraic independence, and typically yields stronger transcendence measures, as well as algebraic independence measures (see, for instance,~\cite{adamczewski-faverjon-2025-algebraicindependencemeasuresvalues,Adamczewski-Faverjon-2026}). 
However, this comes at the cost of requiring the presence of functional equations, which is a very restrictive condition: such functions form at most a countable family and are not robust under perturbations.  By contrast, the approach developed in the present article leads to weaker results, but applies to uncountable families of functions and remains effective in the presence of sufficiently moderate noise.
\begin{remark}
    For some families of $\mathbf{a}$  considered above, the corresponding generating series are related to Mahler functions. More precisely, when $\mathbf{a}$ is a Sturmian word, $f_\mathbf{a}$ satisfies a chain of linear Mahler equations~\cite{Bugeaud-Laurent-2023}. When $\mathbf{a}$ is a lacunary sequence given by $u_i=d^i+r$ where $d\geq2$, $f_\mathbf{a}$ satisfies a linear Mahler equation~\cite{Mahler-1975}. In such cases, Mahler's method~\cite{adamczewski-faverjon-2025-algebraicindependencemeasuresvalues,Adamczewski-Faverjon-2026} could be applied. Moreover, Mahler functions are known to have the unit circle as a natural boundary (hence transcendental) once they are not rational. The same property holds in our case: for any infinite word $\mathbf{a}$ satisfying asymptotic conditions as in Theorem~\ref{Theorem: dichotomy}, the generating series is either rational or has the unit circle as a natural boundary, as follows from a generalization of the P\'olya-Carlson dichotomy~\cite[Theorem~1.2]{Bell-Gunn-Nguyen-Saunders-2023}. Furthermore, the generating series of an infinite word $\mathbf{a}$ with $\mathbf{Rdio}(\mathbf{a})>1$ is algebraic if and only if $\mathbf{a}$ is eventually periodic. Similarly, the generating series of a lacunary sequence always has the unit circle as a natural boundary.
\end{remark}

The article is organized as follows. In Section~\ref{refined Diophantine exponent}, we introduce some basic properties of the refined Diophantine exponent and compute it for several families of words. In Section~\ref{section: Transcendence result}, after a few preparatory results, we prove Theorem~\ref{Theorem: dichotomy}. Section~\ref{section: Transcendence measure} is devoted to the proofs of Theorem~\ref{Theorem: transcendence measure of refined Dio} and Corollary~\ref{Corollary: transcendence measure of refined Dio}. We then adapt the strategy from the proof of Theorem~\ref{Theorem: transcendence measure of refined Dio} to prove Theorem~\ref{Theorem: trans meas of lacunary} in Section~\ref{section: Transcendence measures of lacunary numbers}. Finally, we give an application of our theorems in Section~\ref{section: applications}, providing a proof of Theorem~\ref{theorem: transcendence measure of certain numbers}.

\subsection*{Acknowledgement}
I am grateful to Boris Adamczewski for his suggestions, helpful discussions, careful reading, and guidance in writing this article. I sincerely thank Colin Faverjon for discussions on transcendence measures, as well as for his careful reading, suggestions, and corrections that helped improve the presentation. I would also like to thank Charles Favre for his constant encouragement.

This project has received funding from the European Union’s MSCA-Horizon Europe, grant agreement No. 101126554. 
 \subsection*{Disclaimer}
Co-funded by the European Union. Views and opinions expressed are however those of the
author only and do not necessarily reflect those of the European Union. Neither the European Union nor the granting authority can be held responsible for them.

\noindent \includegraphics[height = 1cm]{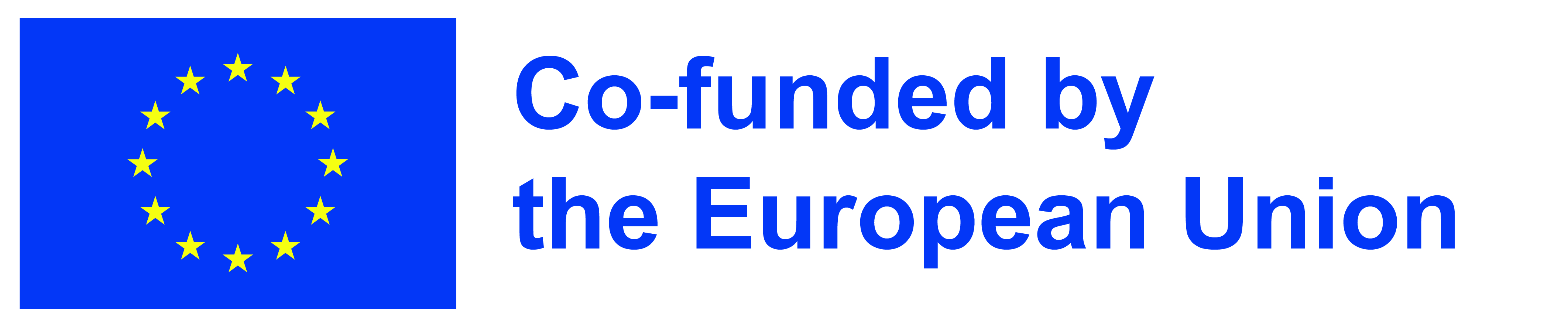}

\section{Refined Diophantine exponent}\label{refined Diophantine exponent}
In this section, we study some basic combinatorial properties of the refined Diophantine exponent. We then compute it for several recently introduced words.

\subsection{Refined Diophantine exponent}

Recall that $\mathbf{Rdio}(\mathbf{a})\geq\mathbf{Dio}(\mathbf{a})$.  In particular, if $\mathbf{Dio}(\mathbf{a})=\infty$, then $\mathbf{Rdio}(\mathbf{a})=\infty$. In contrast, there exist examples of infinite words—namely, lacunary sequences—whose Diophantine exponent is as small as possible, while their refined Diophantine exponents are infinite (see Proposition~\ref{Proposition: refined Dio of lacunary} below). However, we still have the following.

\begin{proposition}\label{Proposition: Rdio>1 yields Dio>1}
     Let $\mathbf a$ be an infinite word. Then $\mathbf{Rdio}(\mathbf{a})=1$ if and only if $\mathbf{Dio}(\mathbf{a})=1$.
\end{proposition}

To prove this proposition, we introduce some notation. For two finite subsets $I$ and $J$ of $\mathbb N=\{0,1,2,\ldots\}$, we write $I<J$ if $\max\{i:i\in I\}<\min\{j:j\in J\}$ and in this case, we denote $d(I,J)=\min\{j:j\in J\}-\max\{i:i\in I\}$ for their distance. For a finite set, we use $|\cdot|$ to denote its cardinality. We need the following refinement of the condition $(*)_\rho$.
\begin{definition}
      Let $\mathbf{a}=a_0a_1\cdots$ be an infinite word.  Let
$\rho\geq1$ and $\epsilon>0$. We say that $\mathbf{a}$ satisfies Condition $(*)_{\rho,\epsilon}$ if there exist  $(r_n)_{n\geq 0}$, $(s_n)_{n\geq 0}$, $(t_n)_{n\geq 0}$ and $\delta$ satisfying conditions (i), (ii'), (iii), and (iv) in Definition~\ref{Definition: refined Dio}.
\end{definition}

The condition $(*)_{\rho,\epsilon}$ ensures that for all sufficiently large $n$, the set of \emph{mismatches} \[ \left\{1\leq i\leq t_n-s_n:a_{i+r_n}\neq a_{i+s_n}\right\} \]is contained in $\delta$ intervals of total length at most $\epsilon(t_n-s_n)$. These intervals, denoted by $I_{n,j}=[l_{n,j}, r_{n,j}]$ for $j=1,\ldots,\delta$, are ordered such that $I_{n,1}<\cdots<I_{n,\delta}$. Here, $l_{n,j}$ (resp. $r_{n,j}$) denotes the left (resp. right) endpoint of the interval $I_{n,j}$.
   % and set $l_{n,j}=r_{n,j}=0$
\begin{remark}
    If  $\mathbf{a}$ is an infinite word and $\rho\geq1$, then  $\mathbf{a}$ satisfies $(*)_\rho$ if and only if $\mathbf{a}$ satisfies $(*)_{\rho,\epsilon}$ for all $\epsilon>0$. Further, if $\mathbf{a}$ satisfies $(*)_{\rho,\epsilon}$ for some $\epsilon>0$, then $\mathbf{a}$ satisfies $(*)_{\rho',\epsilon'}$ for all $\rho\geq\rho'\geq1$ and all $\epsilon'\geq \epsilon$.
\end{remark}
 The following lemma shows that, using a suitable modification of the intervals $I_{n,j}$, we can always assume that there is a large gap between any two consecutive intervals $I_{n,j}$ and $I_{n,j+1}$.
\begin{lemma}\label{Lemma: gap between boxes}
  Let $\rho>1$ be a real number. In the definition of $(*)_\rho$, we can modify the boxes $I_{n,j}$ so that for all $n$ large enough, the distances between two consecutive intervals are at least $\frac{\epsilon(t_n-s_n)}{2+\delta}$.  
\end{lemma}
\begin{proof}
  For any $\epsilon>0$, we set $\epsilon'=\frac{\epsilon}{2+\delta}>0$ and let  $(r_n)_{n\geq 0}$, $(s_n)_{n\geq 0}$, $(t_n)_{n\geq 0}$, $\delta$ be a data associated with $\epsilon'$ in the definition of $(*)_{\rho,\epsilon'}$. Let $I_{n,1}',\ldots,I_{n,\delta}'$ be the corresponding intervals. If $\delta=0$, we are done. If $\delta\geq1$, we do the following modifications:\begin{enumerate}
        \item If $I_{n,j}'$ and $I_{n,j+1}'$ have distance $<{\epsilon'}(t_n-s_n)$ for some $j$, then we merge them into a new box $I_{n,j}=[l_{n,j},r_{n,j+1}]$;
        \item If $I_{n,j}'$ and $I_{n,j+1}'$ have distance $\geq{\epsilon'}(t_n-s_n) $ for some $j$, then we do nothing.
    \end{enumerate}
    For each $n$ sufficiently large, we observe that after completing all the possible modifications, we obtain at most $\delta$ boxes whose total length is $\leq(\delta-1) \epsilon'(t_n-s_n)+\epsilon'(t_n-s_n)\leq\epsilon(t_n-s_n)$. Therefore, the new boxes satisfy our desired conditions.
\end{proof}

We can also obtain a large gap between $r_n$ and the first mismatch as follows. Recall that $I_{n,1}=[l_{n,1},r_{n,1}]$.
\begin{lemma}\label{Lemma: gap between sn and ln1}
Let $\rho>1$ be a real number. Let $\mathbf{a}$ be an infinite word satisfying the condition $(*)_\rho$. Let $\epsilon>0$ be sufficiently small, namely $\epsilon < \min \left\{ \frac{1}{4}, \frac{1}{2(\rho-1)} \right\} $. Then there exist \[\epsilon'\in (\epsilon,2\epsilon) \text{ and } \rho' \in \left(\frac{\rho}{1 + 2\epsilon(\rho - 1)},\rho \right)\] such that $\mathbf{a}$ satisfies the condition $(*)_{\rho',\epsilon'}$ with a data $(r_n')_{n\geq 0}$, $(s_n')_{n\geq 0}$, $(t_n')_{n\geq 0}$, $\delta'$ satisfying the following gap conditions with respect to $\rho',\epsilon'$ and $\delta'$: when $n$ is large enough, the distances between two consecutive intervals are at least $\frac{\epsilon'(t_n'-s_n')}{2(2+\delta')}$, and the distance between $r_n'$ and the first mismatch is at least $\frac{\epsilon'(t_n'-s_n')}{2(2+\delta')}$, i.e, $l_{n,1}'\geq\frac{\epsilon'(t_n'-s_n')}{2(2+\delta')}$.
\end{lemma}
\begin{proof}
By Lemma~\ref{Lemma: gap between boxes}, we can find a data for $(*)_{\rho,\epsilon}$ such that for all $n$ large enough, the distances between two consecutive intervals are at least $\frac{\epsilon(t_n-s_n)}{2+\delta}$. In particular, we have $d(I_{n,1},I_{n,2})\geq \frac{\epsilon(t_n-s_n)}{2+\delta}$. Since the total length of $\delta$ intervals is at most  $\epsilon(t_n-s_n)$, we have $r_{n,1}-l_{n,1}\leq\epsilon(t_n-s_n)$. If $l_{n,1}\geq\frac{\epsilon(t_n-s_n)}{2+\delta}$, then we are done. Now, if $l_{n,1}<\frac{\epsilon(t_n-s_n)}{2+\delta}$, then we do the following modification. We set $r_n'=r_n+r_{n,1}$, $s_n'=s_n+r_{n,1}$ and $t_n'=t_n$. Then, since $\epsilon<\frac{1}{4}$, we have
  \[t_n-s_n>t_n'-s_n'=t_n-s_n-r_{n,1}>(t_n-s_n)\left(1-\epsilon-\frac{\epsilon}{2+\delta}\right)>0\] since $r_{n,1}<(t_n-s_n)\left(\epsilon+\frac{\epsilon}{2+\delta}\right)$.  
  Since all the mismatches in the first interval now belong to $\mathbf{a}[r_n']$, we deduce that the total length of the remaining $\delta-1$ intervals is at most $\epsilon(t_n-s_n)\leq\epsilon'(t_n-s_n)$ where $\epsilon'=\frac{\epsilon}{1-\epsilon-\frac{\epsilon}{2+\delta}}\in(\epsilon,2\epsilon]$ since $\epsilon<\frac{1}{4}$. Thus the distances between two new consecutive intervals are $\geq \frac{\epsilon(t_n-s_n)}{2+\delta}> \frac{\epsilon'(t_n'-s_n')}{2(2+\delta)}$. In addition,  the distance between $r_n'$ and the first new mismatch (between $\mathbf{a}[r_n'+1,r_n'+t_n'-s_n']$ and $\mathbf{a}[s_n'+1,t_n']$) is $\geq \frac{\epsilon(t_n-s_n)}{2+\delta}> \frac{\epsilon'(t_n'-s_n')}{2(2+\delta)}$.
  Moreover,  $s_n'-r_n'=s_n-r_n$ still tends to $\infty$. Further, there is $\rho'<\rho$ such that $\rho'\frac{1-\epsilon-\frac{\epsilon}{2+\delta}}{1-\rho'\epsilon-\rho'\frac{\epsilon}{2+\delta}}=\rho$. Since $t_n\geq\rho s_n$, we have $t_n\left(1-\rho'\left(\epsilon+\frac{\epsilon}{2+\delta}\right)\right)\geq \rho' s_n\left(1-\left(\epsilon+\frac{\epsilon}{2+\delta}\right)\right)$, i.e. , 
  \[t_n\geq \rho's_n+\rho'(t_n-s_n)\left(\epsilon+\frac{\epsilon}{2+\delta}\right),\]
  which yields that $t_n'\geq \rho' s_n+\rho'r_{n,1}=\rho' s_n'$, hence $t_n'\geq\max\{1,\rho'\}s_n'$.  We note also that since $\epsilon<\frac{1}{2(\rho-1)}$, we have  $\rho'>\frac{\rho}{1 + 2\epsilon(\rho - 1)}$. Therefore, our new data satisfies the desired conditions.
  % $\rho'=\frac{\rho}{1-2\epsilon+2\epsilon\rho}\in(1,\rho)$, we have $\rho=\frac{\rho'(1-2\epsilon)}{1-2\epsilon\rho'}$, hence
\end{proof}
\begin{remark}
    When $\epsilon$ tends to $0$, $\rho'$ tends to $\rho$.
\end{remark}
\begin{proof}[Proof of Proposition~\ref{Proposition: Rdio>1 yields Dio>1}]

We only need to show that if $\mathbf{Rdio}(\mathbf{a})>1$, then $\mathbf{Dio}(\mathbf{a})>1$. Let $\rho$ be a real number such that $1<\rho<\mathbf{Rdio}(\mathbf{a})$ and let $\epsilon>0$ be sufficiently small. In virtue of Lemma~\ref{Lemma: gap between sn and ln1}, we can find a data $(r_n)_{n\geq0}$, $(s_n)_{n\geq0}$, $(t_n)_{n\geq0}$, $\delta$ for $(*)_{\rho,\epsilon}$ such that the distance between ${r_n}$ and the first mismatch ${r_n+l_{n,1}}$ is at least $\frac{\epsilon(t_n-s_n)}{2(2+\delta)}$. We set $U_n=a[0,r_n]$ and  $V_n=a[r_n+1,s_n]$, then $\mathbf{a}$ coincides with the eventually periodic sequence $U_n V_n^{\infty}$ up to the index $s_n+l_{n,1}$ with \[\frac{s_n+l_{n,1}}{s_n}\geq 1+\frac{\epsilon(\rho-1)}{2(2+\delta)}\cdot \]
    Therefore $\mathbf{Dio}(\mathbf{a})>1$.
\end{proof} 

Using Theorem~\ref{Theorem: dichotomy}, we see that there are many examples of infinite words over a finite alphabet whose refined Diophantine exponent is equal to 1.
\begin{corollary}\label{corollary: Rdio=1}
    Let $\beta$ be a Pisot or Salem number. Let $x\in[0,1)\setminus\mathbb{Q}(\beta)$ be an algebraic number. Let $0.a_0a_1\ldots$ be its $\beta$-expansion. Then $\mathbf{Dio}(a_0a_1\cdots)=\mathbf{Rdio}(a_0a_1\cdots)=1$.
\end{corollary}
\subsection{Connections to other families of words}
We show that for several families of infinite words, ranging from classical examples to more recently introduced ones, their refined Diophantine exponents are infinite.
\subsubsection{Lacunary sequences} We begin with lacunary sequences, which are known to be a quintessential case study in transcendental number theory, dating back to the work of Liouville~\cite{Liouville-1844}.

\begin{definition}\label{Definition: simple lacunary number}
  An infinite word $\mathbf{a}$ over an alphabet $\mathcal{A}$ is called a \emph{lacunary} sequence if there exist $b\in \mathcal{A}$ and a sequence of integers $0\leq u_0<u_1<\cdots$ such that $\liminf\frac{u_{n+1}}{u_n}>1$ and $\mathbf a$ is of the form $a_i\neq b$ when $i=u_n$ for some $n$, and $a_i=b$ otherwise. So we have \[\mathbf{a}=b\cdots ba_{u_0}b\cdots ba_{u_1}b\cdots ba_{u_j}b\cdots \]
\end{definition}
\begin{proposition}\label{Proposition: refined Dio of lacunary}
Let $\mathbf{a}$ be a lacunary sequence. Then $\mathbf{Rdio}(\mathbf{a})=\infty$. 
\end{proposition}
\begin{proof}
   We fix any $\rho\geq1$. Let $\epsilon>0$, then there is some $n_0>0$ such that $u_{n+n_0}\geq \rho u_n$ for all $n\geq0$. We set $r_n=u_{n}$, $s_n=u_{n+1}$, $t_n=u_{n+n_0}$. So we have  $\mathbf{a}={b\cdots ba_{r_n}}b\cdots ba_{s_n}b\cdots ba_{t_n}b\cdots$
   Then the number of mismatches between \[\mathbf{a}[r_n+1,r_n+t_n-s_n]=a_{u_n+1}\cdots a_{u_{n+1}}\cdots a_{u_n+u_{n+n_0}-u_{n+1}}\] and \[\mathbf{a}[s_n+1,t_n]=a_{u_{n+1}+1}\cdots a_{u_{n+n_0}}\] are at most $\delta=2(n_0-1)$, and $t_n/s_n\geq\rho$. Therefore, $\mathbf{a}$ satisfies the condition $(*)_\rho$ for all $\rho$ as desired.
\end{proof}

In fact, the proof of Proposition~\ref{Proposition: refined Dio of lacunary} yields that when $\limsup \frac{u_{i+1}}{u_i}$ is finite, the lacunary sequence $\mathbf{a}$ satisfies the condition $(**)_\rho$ for all $\rho\geq1$, so we can apply Theorem~\ref{Theorem: transcendence measure of refined Dio} to obtain a trichotomy for transcendence measures of lacunary numbers. Nevertheless, we even have a dichotomy (see Theorem~\ref{Theorem: trans meas of lacunary}) based on the Diophantine exponent: for  a lacunary sequence $\mathbf{a}$, we have $\mathbf{Dio}(\mathbf{a})=\limsup\frac{u_{n+1}}{u_n}$, see~\cite[Th\'eor\`eme~6.1]{Adamczewski-Bugeaud-2011}.  In particular, $\mathbf{Dio}(\mathbf{a})$ is finite if and only if $\limsup\frac{u_{n+1}}{u_n}$ is finite.

\begin{remark}
   In view of lacunary sequences, one can similarly construct other examples of words with infinite refined Diophantine exponents. For instance, for any integers $0\leq u_0<u_1<\cdots$ such that $\liminf\frac{u_{n+1}}{u_n}>1$, we set $\mathbf{a}=0^{u_1}1^{u_1}0^{u_2}1^{u_2}\cdots\in\{0,1\}^{\mathbb N}$ where $0^{u_i}$ (resp. $1^{u_i}$) denotes the finite word defined by $u_i$ occurrences of 0 (resp. $1$). One can verify that $\mathbf{Rdio}(\mathbf{a})=\infty$. 
\end{remark}

\subsubsection{Stuttering words}
Next, we compute refined Diophantine exponents of {stuttering words} introduced in~\cite{Luca-Ouaknine-Worrell-2023}. The motivation behind these words is the transcendence and linear independence of numbers arising from {Sturmian words} over any algebraic base. Over integer bases, their transcendence (resp. linear independence) was first proved by Ferenczi--Mauduit~\cite{Ferenczi-Mauduit-1997-LowComplexity} (resp.~\cite{Bugeaud-Kim-Laurent-Nogueira2021}). However, going to any algebraic bases requires a more detailed analysis of the combinatorial nature of Sturmian words. 

\begin{definition}\label{Definition: Stuttering words}(\cite[page~3]{Luca-Ouaknine-Worrell-2023})
    An infinite word over a finite alphabet $\mathbf{a}$ is called a \emph{stuttering} word if for all $\rho>0$, there exist sequences of positive integers $(u_n)_{n\geq0},(v_n)_{n\geq0}$ and an integer $d\geq2$ such that
    \begin{itemize}
        \item[(i)] $u_n$ is unbounded and $v_n\geq\rho u_n$ for all $n$;
        \item[(ii)] for every $n$, there exist integers $i_0(n)=0<i_1(n)<\cdots<i_d(n)\leq i_{d+1}(n)=v_n$ such that the words $a_0\cdots a_{v_n}$ and $a_{u_n}\cdots a_{u_n+v_n}$ differ at positions  $\bigcup_{j=1}^d \{i_j(n),i_j(n)+1\}$;
        \item[(iii)] we have
$i_d(n)-i_1(n)=\omega(\log u_n)$
and $i_{j+1}(n)-i_j(n)=\omega(1)$ for all  $j\in\{0,1,\dots,d\}$;
\item[(iv)] we have $u_{i_j(n)}+ u_{i_j(n)+1}= u_{i_j(n)+u_n}+ u_{i_j(n)+u_n+1}$ for all $n$ and all $j\in\{1,\ldots,d\}$.
    \end{itemize}
\end{definition}

Examples of stuttering words include \emph{Sturmian words}, which can be defined as follows. The \emph{complexity function} of an infinite word $\mathbf{a}=a_0a_1\cdots$ is the function $n\mapsto p(n,\mathbf{a})$ defined by\[p(n,\mathbf{a})=|\{a_i\cdots a_{i+n-1}:i\geq0\}|.\]An infinite word $\mathbf{a}$ is then called Sturmian if $p(n,\mathbf{a})=n+1$ for all $n$. A classic example of a Sturmian word is the Fibonacci word.

The proof of the following proposition will also be used in Theorem~\ref{theorem: transcendence measure of certain numbers} to verify that Sturmian words satisfy condition $(**)_\rho$ for all $\rho\geq1$.
\begin{proposition}\label{Proposition: refined Dio of stuttering}
    Let $\mathbf a$ be a stuttering word, then $\mathbf{Rdio}(\mathbf{a})=\infty$.  
\end{proposition}
\begin{proof} Let $\rho\geq1$ and $\epsilon>0$. With the data given by Definition~\ref{Definition: Stuttering words}, we set $r_n=0,s_n=u_n,t_n=u_n+v_n$ for all $n$, and set $\delta=2d$. The number of mismatches between $a[r_n+1,r_n+t_n-s_n]=a_1\cdots a_{v_n}$ and $a[s_n+1,t_n]=a_{u_n+1}\cdots a_{u_n+v_n}$ is at most $\delta$, which is $\leq \epsilon(t_n-s_n)$ for $n$ sufficiently large. Therefore $\mathbf{Rdio}(\mathbf{a})\geq\rho$ for all $\rho\geq1$, hence $\mathbf{Rdio}(\mathbf{a})=\infty$.
\end{proof}
\begin{remark}\label{Remark: recover stuttering}
  Combining with Theorem~\ref{Theorem: dichotomy}, we recover the dichotomy part of~\cite[Theorem~5]{Luca-Ouaknine-Worrell-2023}. 
\end{remark}
\subsubsection{Echoing words}
We now study the refined Diophantine exponent of echoing words, which were recently introduced in~\cite{Kebis-Luca-Ouaknine-Scoones-Worrell-2025} to analyze certain Arnoux-Rauzy words, namely $k$-bonacci words.
\begin{definition}\label{Definition: Echoing words}
Let $\mathbf{a} = a_0a_1\cdots$ be an infinite word over a finite alphabet. Then $\mathbf{a}$ is said to be \emph{echoing} if for every $\epsilon>0$,
    there exist integers $0 \leq r_n < s_n$ and non-empty intervals $
\{0\} = I_{n,0} < I_{n,1} < I_{n,2} < \cdots$
for every $n \in \mathbb{N}$ such that
\begin{itemize}
\item[(i)] $\{ i\geq1 :  a_{i+s_n} \ne  a_{i+r_n} \}
\subseteq \bigcup_{j=1}^{\infty} I_{n,j}$;
\item[(ii)] $\mathrm{den}\left(\bigcup_{j=1}^{\delta} I_{n,j}\right) \le \epsilon$
for all sufficiently large  $\delta$  and  $n$;
\item[(i)]  $s_n < s_{n+1}$ for all $n$;
\item[(i)] we have $s_n - r_n \gg s_n$,
 and $d(I_{n,j}, I_{n,j+1}) \gg s_n,$
where the implied constants are independent of $n,j \in \mathbb{N}$.
\end{itemize}
Here, for a finite subset $I\neq\{0\}$ of $\mathbb N$, we define its \emph{density} as \[\mathrm{den}(I)=\frac{|I|}{\max\{i:i\in I\}}\cdot\] 
\end{definition}
Typical examples of echoing words are the \emph{$k$-bonacci words}~\cite[Theorem~12]{Kebis-Luca-Ouaknine-Scoones-Worrell-2025}, which can be defined as follows. For $k\geq2$, the $k$-bonacci word over the finite alphabet $\Sigma=\{b_0,b_1,\ldots,b_{k-1}\}$ is the infinite word given by the limit $\lim\varphi^n(0)$, where $\varphi$ is the $k$-bonacci self-morphism on the free semigroup generated by $\Sigma$, defined by\[\varphi(b_0)=b_0b_1,\varphi(b_1)=b_0b_2, \ldots,\varphi(b_{k-2})=b_0b_{k-1}, \varphi(b_{k-1})=b_0.\]

The proof of the following result will also be used in establishing Theorem~\ref{theorem: transcendence measure of certain numbers} to check that $k$-bonacci words satisfy condition $(**)_\rho$ for all $\rho\geq1$.
\begin{proposition}\label{Proposition: echoing word has inf refined Dio}
    Let $\mathbf a$ be an echoing word over a finite alphabet, then $\mathbf{Rdio}(\mathbf{a})=\infty$.
\end{proposition}
\begin{proof}
 Let  $\rho\geq 1$ and $0<\epsilon<1$. We use the data of Definition~\ref{Definition: Echoing words} with respect to $\epsilon$.
It follows that there exists $\delta\in\mathbb N$ such that for all $n$, the right endpoint of $I_{n,\delta}$, say $t_n$, is $\geq\rho s_n$. We note that  $\mathbf{a}[r_n+1, t_n-s_n]=a_{r_n+1}\cdots a_{r_n+t_n-s_n}$ and $\mathbf{a}[s_{n}+1,t_n-s_n]=a_{s_n+1}\cdots a_{t_n}$. Thus the number of mismatches between $\mathbf{a}[r_n+1, t_n-s_n]$ and $\mathbf{a}[s_{n}+1,t_n-s_n]$  is equal to the cardinality of the set  \[\left\{1\leq i\leq t_n-s_n:a_{i+s_n}\neq a_{i+r_n}\right\},\] which is $\leq\epsilon t_{n}-s_n\leq\epsilon(t_n-s_n)$. Thus $\mathbf{a}[r_n+1, t_n-s_n]$ is $(\epsilon,\delta)$-close to $\mathbf{a}[s_{n}+1,t_n-s_n]$. Therefore $\mathbf{Rdio}({\mathbf a})\geq\rho$ for all $\rho\geq1$ as wanted.
\end{proof}
\begin{remark}\label{Remark: recover echoing}
    Combining with Theorem~\ref{Theorem: dichotomy}, we recover~\cite[Theorem~9]{Kebis-Luca-Ouaknine-Scoones-Worrell-2025}. In addition, Proposition~\ref{Proposition: echoing word has inf refined Dio} also applies to \emph{strongly echoing} words introduced in~\cite[Definition~10]{Kebis-Luca-Ouaknine-Scoones-Worrell-2025}.
\end{remark}

\subsubsection{Algebraic dynamics}
Next, we focus on words arising from algebraic dynamics--more precisely, words coming from coding of rotations by rational intervals. We will show that the refined Diophantine exponents of these words are always infinite. Before stating the result, let us provide some motivation from algebraic dynamics. 

Algebraic dynamics is the study of self-maps of algebraic varieties. One of the most important birational invariants of such dynamical systems is the \emph{dynamical degree}, an algebraic analogue of entropy in topological dynamics. In general, dynamical degrees are very difficult to compute. The only accessible cases include monomial maps, algebraically stable maps, and certain polynomial endomorphisms. In all these cases, the dynamical degrees are algebraic integers, and it had been conjectured that all dynamical degrees are algebraic numbers. It was therefore a major surprise when Bell, Diller, and Jonsson~\cite{BDJ-2020} constructed the first example of a rational map whose dynamical degree is transcendental. Their construction relies on toric geometry, and their transcendence proof relies on Diophantine approximation, motivated by the works~\cite{Corvaja-Zannier-2002, Adamczewski-Bugeaud-Luca-2004,Adamczewski-Bugeaud-2007-dynamics, Adamczewski-Bugeaud-2007}. For the geometric construction, we refer the reader to~\cite[\S2]{BDJ-2020}, which can be summarized as follows.

Let $\zeta\in\mathbb Z[\mathrm{i}]$ be a Gaussian integer whose argument is incommensurable with $2\pi\mathbb Q$. Let $\theta=\frac{\arg(\zeta)}{2\pi}\in (0,1)\setminus\mathbb Q$. 
We set $\Gamma=\{0,-2,\pm2\mathrm{i},1\pm2\mathrm{i}\} $. Let $\mathbf{\gamma}=\gamma_1\gamma_2\cdots$ be the word defined by 
\begin{equation*}\label{eq: defn of gamma}
\gamma_i=\begin{cases*}
   1-2\mathrm{i}& when $\{i\theta\}\in\left(0,\frac{2}{8}\right)$,\\
     -2\mathrm{i}& when $\{i\theta\}\in\left(\frac{2}{8},\frac{3}{8}\right)$,\\
     -2& when $\{i\theta\}\in\left(\frac{3}{8},\frac{5}{8}\right)$,\\
      2\mathrm{i}& when $\{i\theta\}\in\left(\frac{5}{8},\frac{6}{8}\right)$,\\
     1+2\mathrm{i}& when $\{i\theta\}\in\left(\frac{6}{8},1\right)$.\\
\end{cases*}    
\end{equation*}
Here, $\{\cdot\}$ denotes the fractional part. We set $\Phi(z)=\sum_{i\geq1}\gamma_iz^i$. Using geometric considerations, the authors of~\cite{BDJ-2020} show that one can associate
with such a $\zeta$ a dominant rational map $f\colon\mathbb P^2\dashrightarrow \mathbb P^2$ whose dynamical degree  ${\lambda(f)}$ is a solution of the following equation 
\[\mathrm{Re}(\Phi({\lambda(f)}^{-1}{\zeta}))=1.\]
The transcendence of $\lambda(f)$  is then obtained by proving that $\mathrm{Re}(\Phi(\alpha))$ is transcendental for all algebraic numbers $\alpha$, $0<|\alpha|<1$. To this end, the authors measure the periodicity of the word $\gamma$. This allowed them to construct good rational approximations for $\Phi(z)$, and consequently for $\mathrm{Re}(\Phi(z))$, to which they applied tools from Diophantine approximation to deduce the final result. It turns out that such periodicity also implies that $\mathbf{Rdio}(\gamma)=\infty$, as we will prove next. 

In fact, we will consider a more general class of words called \emph{coding of rotations by rational
intervals}, which are defined as follows. Let $\theta\in(0,1)$ be irrational. Let $s\geq2$ be an integer, and let $\{\alpha_1,\ldots,\alpha_s\}$ be a finite alphabet with at least two distinct letters ($\alpha_i$ might equal $\alpha_j$ for $i\neq j$). Then the word $\mathbf{a}=a_1a_2\cdots$ obtained by coding of a rotation $\theta$ by rational intervals $\left(\frac{j-1}{s},\frac{j}{s}\right)$ for $1\leq j\leq s$ is defined by $a_i=\alpha_j$ if $\{i\theta\}\in \left(\frac{j-1}{s},\frac{j}{s}\right)$.

\begin{proposition}\label{Proposition: Rdio coding rotation = infty}
 For any word $\mathbf{a}$ given by coding of rotations by rational intervals, we have $\mathbf{Rdio}(\mathbf{a})=\infty$.
\end{proposition}

We will rely on the following well-known fact in the theory of continued fractions.
\begin{lemma}\label{eq: equivalences of badly approximable} (see e.g.~\cite[Proposition~3.3]{BDJ-2020})
    Let $\theta$ be an irrational number with continued fraction approximants $\frac{p_n}{q_n}$. The following assertions are equivalent.
    \begin{itemize}
        \item[(i)] There exists $\kappa>0$ such that $|q_n\theta-p_n|\geq\frac{\kappa}{q_n}$ for all $n$.
         \item[(ii)] There exists $\kappa>0$ such that $|q\theta-p|\geq\frac{\kappa}{q}$ for all $p,q$ with $q>0$.
         \item[(iii)] There exists $A$ such that $q_{n+1}\leq Aq_n$ for all $n$.
         \item[(iv)] $\theta$ has bounded partial quotients.
    \end{itemize}
\end{lemma}
 Here, the \emph{partial quotients} of $\theta$ are the integers $w_0,w_1,\ldots$ determined via the continued fraction $\theta=[w_0:w_1:w_2:\ldots]$.  The \emph{continued fraction approximants} of $\theta$ are the sequences $\frac{p_n}{q_n}=[w_0:\ldots:w_n]$ where $q_n$ are strictly increasing and $p_n$
is coprime to $q_n$.

Now, to prove Proposition~\ref{Proposition: Rdio coding rotation = infty}, we let $\theta$ be the rotation defining $\mathbf{a}$ and consider two cases of $\theta$ depending on the boundedness of the partial quotients of $\theta$.
\begin{lemma}\label{proposition: theta well approximable}
Assume that $\theta$ has unbounded partial quotients.   Then  for every $\rho>1$, there are infinitely many $q$ such that $a_{i}=a_{i-q}$ for all $i\in (q,\rho q]$. In particular, $\mathbf{Dio}(\mathbf{a})=\mathbf{Rdio}(\mathbf{a})=\infty$.
\end{lemma}
\begin{proof} We adapt the proof of~\cite[Proposition~3.9]{BDJ-2020}.    We take $\kappa=\frac{1}{2s(\rho+1)}$. By Lemma~\ref{eq: equivalences of badly approximable}, for $p=p_n$ and $q=q_n$ with $n$ sufficiently large, we have $|q\theta-p|<\frac{\kappa}{q}$ and $p$ is coprime to $q$. We consider any $i\in (q,\rho q]$ and assume that $a_i=\alpha_j$, i.e., $\{i\theta\}\in \left(\frac{j-1}{s},\frac{j}{s}\right)$ or $\left(\frac{j}{s},\frac{j+1}{s}\right)$ where   $j$ is the integer closest to $si\theta$. We have two cases.
    \begin{enumerate}
        \item If $\left|{i}\theta-\frac{j}{s}\right|\geq\frac{\kappa}{q}$, then either both $\{i\theta\}$ and $\{(i-q)\theta\}$ lie in $\left(\frac{j-1}{s},\frac{j}{s}\right)$ or both lie in $\left(\frac{j}{s},\frac{j+1}{s}\right)$.
        \item If $\left|{i}\theta-\frac{j}{s}\right|<\frac{\kappa}{q}$, then
        \begin{align*}
            |sip-j q|=s\left|i(p-q\theta)+q\left(i\theta-\frac{j}{s}\right) \right|&<s\left(\frac{i\kappa}{q}+\kappa\right)\\&\leq s\kappa(i+1)\leq \frac{1}{2}\cdot
        \end{align*}
         Thus $sip=j q$, whence there exists an integer $j'$ such that $i=\frac{j'q}{s}$, hence $j=j'p$. Therefore $j'\in(s,\rho s]$. Further, we have
        $i\theta-\frac{j}{s}=\frac{j'}{s}(q\theta-p)$ and $(i-q)\theta-\frac{j-sp}{s}=\frac{j'-s}{s}(q\theta-p).$ Since the absolute values of the right sides of these equalities are bounded above by $\frac{1}{2s}$, we conclude either both $\{i\theta\}$ and $\{(i-q)\theta\}$ lie in $\left(\frac{j-1}{s},\frac{j}{s}\right)$ or both lie in $\left(\frac{j}{s},\frac{j+1}{s}\right)$ as desired.
    \end{enumerate}
In either case, we always have $a_i=a_{i-q}$ as desired.
    The second assertion then follows. 
\end{proof}

\begin{lemma}\label{proposition: theta badly approximable}
Assume that $\theta$ has bounded partial quotients. Then, there is a constant $\nu>0$ such that for infinitely many $q$, the following holds: for any $\rho>1$, there are at most $\frac{\rho}{\nu}$  indices $i\in (q,\rho q]$ such that $a_{i}\neq a_{i-q}$.    In particular, $\mathbf{Rdio}(\mathbf{a})=\infty$.
\end{lemma}
\begin{proof} We mimic the proof of~\cite[Proposition~3.15]{BDJ-2020}.   Since $\theta$ has bounded partial quotients, there is $\kappa>0$ such that $|q\theta-p|\geq \frac{\kappa}{q}$ for any $p,q $ with $q>0$. Now we take $p=p_n,q=q_n$ where $\frac{p_n}{q_n}$ are continued fraction approximants of $\theta$, whence $|q\theta-p|<\frac{1}{q}$. 

We set $\nu=\frac{\kappa}{2s^2}$. It suffices to show that  $|i-i'|\geq \nu q$ for any $i>i'>q$ such that $a_{i}\neq a_{i-q}$ and $a_{i'}\neq a_{i'-q}$. Let $j$ be the integer closest to $si\theta$. If $\left|i\theta-\frac{j}{s}\right|\geq\frac{1}{q}$, then either both $\{i\theta\}$ and $\{(i-q)\theta\}$ lie in $\left(\frac{j-1}{s},\frac{j}{s}\right)$ or both lie in $\left(\frac{j}{s},\frac{j+1}{s}\right)$, which is absurd. Thus $\left|i\theta-\frac{j}{s}\right|<\frac{1}{q}$. Similarly, we write  $j'$ for the integer closest to $si'\theta$, and we have $\left|i'\theta-\frac{j'}{s}\right|<\frac{1}{q}$. It follows that 
    \[\frac{\kappa}{s|i-i'|}\leq|s(i-i')\theta-(j-j')|\leq s\left|i\theta-\frac{j}{s}\right|+s\left|i'\theta-\frac{j'}{s}\right|<\frac{2s}{q}\cdot\]
    Hence $|i-i'|>\frac{\kappa q}{2s^2}=\nu q$ and the first assertion follows.

   For the second assertion, we set $r_n=-1,s_n=q_n-1,t_n=\lfloor \rho q_n\rfloor-1$ and $\delta=\rho/\nu$ (noting that there is a shift by 1 since in the definition of the refined Diophantine exponent, we start with $a_0$). We deduce that $\mathbf{Rdio}(\mathbf{a})\geq \rho$ for all $\rho\geq1$, hence  $\mathbf{Rdio}(\mathbf{a})=\infty$.

\end{proof}

\begin{proof}[Proof of Proposition~\ref{Proposition: Rdio coding rotation = infty}]
    This follows from Lemmas~\ref{proposition: theta well approximable} and~\ref{proposition: theta badly approximable}.
\end{proof}

\begin{remark}\label{remark: dynamics very restricted forms 2}
   Using Theorem~\ref{Theorem: dichotomy} and Proposition~\ref{Proposition: Rdio coding rotation = infty}, we obtain that for the word $\mathbf{\gamma}$ mentioned above, and for any algebraic number $\beta$ with $|\beta|>1$, the real number $\mathrm{Re}(\Phi(\beta^{-1}))$ is either rational or is transcendental. Moreover, the proof of Theorem~\ref{Theorem: dichotomy} shows that if this number is rational, it must be of a very restricted form, see Remark~\ref{remark: dynamics very restricted forms 2} below. The authors of~\cite{BDJ-2020} then use specific tools from Diophantine approximation (namely, results regarding $S$-unit equations) to rule out this case.
\end{remark}
\begin{remark}
     The word $\mathbf{a}$ given by coding of rotations by rational
intervals is not eventually periodic since $\{i\theta\}$ is dense in $(0,1)$. Further, $\mathbf{a}$ has sub-linear complexity. Indeed, for any $n\geq 0$, we consider the fractional parts  $\left\{\frac{j}{s}-i\theta\right\}\in (0,1)$ for $j\in\{0,1,\ldots,s-1\}$ and $i\in\{0,1,\ldots,n-1\}$. Then these $sn$ points split $(0,1)$ into at most $sn+1$ intervals. It suffices to show that the finite word $a_u a_{u+1} \cdots a_{u+n-1}$ depends only on the interval in which $\{u\theta\}$ lies. Assume the contrary that there are $u\neq v$ on some interval such that $a_{u}a_{u+1}\cdots a_{u+n-1}\neq a_{v}a_{v+1}\cdots a_{v+n-1}$, i.e., there is some $0\leq i\leq n-1$ such that
    $a_{u+i}\neq a_{v+i}$. Thus there is some $j$ such that  $\{u\theta+i\theta\}\in \left(\frac{{j-1}}{s},\frac{j}{s}\right) $ and  $\{v\theta+i\theta\}\not\in \left(\frac{{j-1}}{s},\frac{j}{s}\right)$. It follows that either $\left\{\frac{j}{s}-i\theta\right\}$ or $\left\{\frac{j-1}{s}-i\theta\right\}$ lies between $\{u\theta\}$ and $\{v\theta\}$, a contradiction. 
\end{remark}

\subsubsection{Further examples}
Finally, we construct some infinite words whose Diophantine properties cannot be obtained directly by using known results in~\cite{Adamczewski-Bugeaud-Luca-2004,Adamczewski-Bugeaud-2007-dynamics,Adamczewski-Bugeaud-2007,Adamczewski-Bugeaud-2011,Luca-Ouaknine-Worrell-2023,Kebis-Luca-Ouaknine-Scoones-Worrell-2025}. We begin with an observation that, given any infinite word, we can construct a new one by modifying it with a suitable lacunary sequence such that the Diophantine exponent of the new word is as small as possible, while its refined Diophantine exponent is at least that of the original word.

We consider the following condition which is stronger than $(*)_\rho$.
\begin{definition}\label{Definition: (***)rho}
    Let $\mathbf{a}=a_0a_1\cdots$ be an infinite word, and
$\rho\geq1$ a real number. We say that $\mathbf{a}$ satisfies Condition $(***)_\rho$ if for every $\epsilon>0$, there exist sequences of integers $(r_n)_{n\geq 0}$, $(s_n)_{n\geq 0}$, $(t_n)_{n\geq 0}$ and an integer $\delta\geq0$ satisfying the conditions of $(*)_\rho$ and $\limsup\frac{t_{n}}{r_n}<\infty$.
\end{definition}

\begin{proposition}\label{Proposition: merge}
Let $\mathbf{a}=a_0a_1\cdots$ be an infinite word over an alphabet $\mathcal{A}$, and let $b$ be a letter not in $\mathcal{A}$. Let $0\leq u_0<u_1<\cdots$ be an increasing sequence of integers such that $\liminf\frac{u_{n+1}}{u_n}>1$. Let $\mathbf{a}'$ be the new word defined by $a_i'=a_i$ if $i\neq u_n$, and $a_i'=b$ if $i=u_n$ for some $n$. 
\begin{itemize}
    \item Then $\mathbf{Dio}(\mathbf{a}')\leq \limsup\frac{u_{n+1}}{u_n}$, and $\mathbf{a}'$ is not eventually periodic.
    \item For any $\rho\geq1$, if $ \mathbf{a}$ satisfies $(***)_\rho$, then $ \mathbf{a}'$ also satisfies $(***)_\rho$, in particular, $\mathbf{Rdio}(\mathbf{a}')\geq \rho$.
    % \item If $ \mathbf{a}$ satisfies $(***)_\rho$ for all $\rho<\mathbf{Dio}(\mathbf{a})$, then $\mathbf{Rdio}(\mathbf{a}')\geq \mathbf{Rdio}(\mathbf{a})$.
 % this condition holds for all $\rho\leq \mathbf{Rdio}(\mathbf{a})$, then $\mathbf{Rdio}(\mathbf{a}')\geq \mathbf{Rdio}(\mathbf{a})$.
\end{itemize}
 
% In addition, if $\mathbf{a}$ has sub-linear complexity, then so is $\mathbf{a}$.
\end{proposition}
\begin{proof}
    The first assertion is clear from the definition of the Diophantine exponent and from the fact that lacunary sequences are not eventually periodic.

For the second assertion, since $\liminf\frac{u_{n+1}}{u_n}>1$, there exists $c_0>1$ such that $u_{n+m}\geq c_0^mu_n$ for all $n,m\geq0$. The hypotheses yield that $t_n\ll r_n$, thus there are $c_1,c_2>0$ such that $c_1t_n\leq r_n$ and $r_n+t_n-s_n\leq c_2t_n$.  The number of occurrences of $b$ in $\mathbf{a}'[r_n+1,r_n+t_n-s_n]$ is at most the numbers of $j$ such that
\begin{equation}\label{eq: the numbres of modifications}
c_1t_n\leq u_j\leq c_2t_n.    
\end{equation}
Assume that there are exactly $k$ indices $m,m+1,\ldots,m+k-1$ satisfying~\eqref{eq: the numbres of modifications}, i.e., \[u_{m-1}<c_1t_n\leq u_m<u_{m+1}<\cdots<u_{m+k-1}\leq c_2t_n<u_{m+k}.\]
It follows that \[c_0^{k}\leq\frac{u_{m+k-1}}{u_{m-1}}<\frac{c_2}{c_1},\]
hence $k$ is bounded. Thus, the number of occurrences of $b$ in $\mathbf{a}'[r_n+1,r_n+t_n-s_n]$ is always bounded by some constant $k$. Similarly, the number of occurrences of $b$ in $\mathbf{a}'[s_n+1,t_n]$ is also bounded by some constant $k$. Therefore, the number of mismatches between  $\mathbf{a}'[r_n+1,r_n+t_n-s_n]$ and $\mathbf{a}'[s_n+1,t_n]$ is contained in a union of the set of such $2k$ positions of $b$ and $\delta$ intervals of mismatches between  $\mathbf{a}[r_n+1,r_n+t_n-s_n]$ and $\mathbf{a}[s_n+1,t_n]$. By setting $r_n'=r_n,s_n'=s_n,t_n'=t_n$ and $\delta'=\delta+2k$, we deduce that $ \mathbf{a}'$ satisfies $(***)_\rho$ as desired.
% The third assertion is clear by the definition of $\mathbf{Rdio}$.
\end{proof}
\begin{remark}
    The new word $\mathbf{a}'$ might not have sub-linear complexity even if $\mathbf{a}$ has sub-linear complexity. Therefore, we can not apply directly theorems in~\cite{Adamczewski-Bugeaud-2011} to $\mathbf{a}'$.
\end{remark}

To illustrate this construction, we work with the Thue-Morse word.
\begin{example}\label{example: thue morse}
    Recall that the Thue-Morse word $\mathbf{t}$ on $\{0,1\}$ given by the morphism $\tau:\{0,1\}\to \{0,1\}, \tau(0)=01,\tau(1)=10$. More precisely, $\mathbf{t}=\lim\tau^n(0)= 01101001100\cdots$. We have $\mathbf{Dio}(\mathbf{t})\leq2$ since $\mathbf{t}$ is overlap-free. Further, observe that for every $n$, $U_nV_n^2$ is a prefix of $\mathbf{t}$ where $U_n=\tau^{n}(0),V_n=\tau^{n}(1),$ and $|U_n|=|V_n|=2^n$. Thus $\mathbf{Rdio}(\mathbf{t})\geq\mathbf{Dio}(\mathbf{t})\geq\frac{3}{2}$, and  $\mathbf t$ satisfies $(***)_{{3}/{2}}$.

    Now, we consider any real number $1<\lambda<\frac{3}{2}$. The lacunary sequence $\mathbf{l}_\lambda$ on $\{0,2\}$ given by the sequence $u_n=\lfloor \lambda^n\rfloor$ has Diophantine exponent $\limsup\frac{u_{n+1}}{u_n}=\lambda$, where $\lfloor\cdot\rfloor$ denotes the floor function.  We modify $\mathbf{t}$ by $\mathbf{l}_\lambda$ as in Proposition~\ref{Proposition: merge} to obtain the new word $\mathbf{t}'$ with $\mathbf{Dio}(\mathbf{t}')\leq\lambda$. It follows that  $\mathbf{Rdio}(\mathbf{t}')\geq\frac{3}{2}$. Thus, if $ \mathsf{v}$ is any place on $\mathbb Q(\beta)$ and $\beta$ is any algebraic number such that $\frac{3}{2}>\frac{\log \mathrm{H}(\beta)}{\log |\beta|_{\mathsf{v}}}$, then either $\xi=\sum_{i\geq0}{t_i'}{\beta^{-i}}\in \mathbb Q(\beta)_\mathsf{v}$ belongs to $\mathbb Q(\beta)$ or $\xi$ is transcendental. 
    % In the latter case, either $\xi$ is a ($p$-adic) $U_d$ for some $1\leq d \leq [\mathbb Q(\beta):\mathbb Q]$, or $\xi$ is a ($p$-adic) $S$- or $T$-number. 
% This construction also yields a linear independence result. Namely, for such $\beta$, the 
    % In particular, when $\beta$ is a positive integer with $\beta>2$, $\xi$ is always transcendental.
\end{example}
   \begin{remark}\label{remark: fibonacci}
Similarly, the Thue-Morse word can be replaced by any Sturmian or $k$-bonacci word, and the coefficients $0$, $1$, and $2$ can be replaced by any algebraic numbers.
\end{remark}

\section{Transcendence result}\label{section: Transcendence result}
In this section, we prove Theorem~\ref{Theorem: dichotomy}. Before going to the proof, we recall the Subspace Theorem and some auxiliary results.
\subsection{Diophantine prerequisites}\label{subsubsection: Diophantine prerequisites}
We start with some preparations.

Let $K$ be a number field of degree $d_0$ and $M_K$ (resp. $M_K^\infty$ and $M_K^\mathrm{fin}$) be its sets of all places (resp. all archimedean places and all non-archimedean places). We will use the following {normalized absolute values}. For $x\in K$ and $ \mathsf{v}\in M_K$, we set
\begin{itemize}
    \item[(i)] $|x|_\mathsf{v}=|\sigma(x)|^{1/d_0}$ when $\mathsf{v}$ corresponds to a real embedding $\sigma\colon K\xhookrightarrow{}\mathbb R$;
    \item[(ii)] $|x|_\mathsf{v}=|\sigma(x)|^{2/d_0}$ when $\mathsf{v}$ corresponds to a complex (non-real) embedding $\sigma\colon K\xhookrightarrow{}\mathbb C$;
    \item[(iii)] $|x|_\mathsf{v}=(N\mathfrak p)^{-\mathrm{ord}_\mathfrak p(x)/d_0}$ when $\mathsf{v}$ corresponds to a prime ideal $\mathfrak p$ of the ring of integers $\mathcal{O}_K$. 
\end{itemize}

 We note that with these normalizations, the triangle inequality holds for $v$ archimedean. The {product formula} reads as 
 $\prod_{\mathsf{v}\in M_K}|x|_\mathsf{v}=1 \text{ for }x\in K^\times.$ For each $v$, we denote the local degree by $d_\mathsf{v}=[K_\mathsf{v}:\mathbb Q_\mathsf{v}].$ We will  write $|x|$ for the usual absolute value on $\mathbb C$.

For $n\geq2$, we use the following height functions associated with a point $\mathbf{x}=(x_1,\cdots,x_n)\in K^n$. For each place $\mathsf{v}\in M_K$, we set 
\[|\mathbf{x}|_\mathsf{v}=\max\{|x_1|_\mathsf{v},\cdots,|x_n|_\mathsf{v}\}.\] The {absolute Weil height} $\mathrm{H}(x_1,\cdots,x_n)$ is then defined as
$\mathrm{H}(\mathbf{x})=\prod_{\mathsf{v}\in M_K}|\mathbf{x}|_\mathsf{v}$. For $x\in K^\times,$ we write $\mathrm{H}(x)=\mathrm{H}(x,1)$. This notation will not cause any confusion. More generally, for $\mathbf{x}\in\overline{\mathbb Q}^n$, we can find a number field $K$ such that $\mathbf{x}\in K^n$, and one observes that these heights do not depend on the choice of $K$. Thanks to the product formula, we always have $\mathrm{H}(\mathbf{x})=\mathrm{H}(\alpha \mathbf{x})$ for all $\alpha\in\overline{\mathbb Q}^\times$. We write $\mathrm{h}(x)=\log\mathrm{H}(x)$ for the absolute logarithmic Weil height. For a polynomial $P$ with algebraic coefficients, we define its height $\mathrm{H}(P)$ to be the absolute Weil height of its coefficient vector.   

\begin{lemma}(The Northcott property)
    For any positive real numbers $d$ and $H$, the set \[\left\{\mathbf{x}\in\overline{\mathbb Q}^n:\deg (\mathbf{x})\leq d,\mathrm{H}(\mathbf{x})\leq H\right\}\] is finite. Here, $\deg (\mathbf{x})$ denotes the degree of the smallest number field $K$ such that $\mathbf{x}\in K^{n}$.
\end{lemma}

In the following, we fix an extension $|\cdot|_\mathsf{v}$ to $\overline{\mathbb Q}$ for each $\mathsf{v}$.
\begin{customtheorem}{S}(The Subspace Theorem~\cite[Theorem~7.2.2]{Bombieri-Gubler-book})\label{p-adic Subspace Theorem}
 Let $n\geq 1$, $K$ a number field, and $\mathcal{S}$ a finite subset of $M_K$. For each $\mathsf{v}\in \mathcal{S}$, let $L_{1,\mathsf{v}},\dots,L_{n,\mathsf{v}}$ be $n$ linearly independent linear forms in $n$ variables $X_1,\dots,X_n$ with algebraic coefficients. Then for any $\epsilon>0$, the set of non-zero  solutions $\mathbf{x}=(x_1,\cdots,x_n)\in K^{n}$ of the inequality
 \[\prod_{\mathsf{v}\in \mathcal{S}}\prod_{i=1}^n\frac{|L_{i,\mathsf{v}}(\mathbf{x})|_\mathsf{v}}{|\mathbf{x}|_\mathsf{v}}\leq \mathrm{H}(\mathbf{x})^{-n-\epsilon}\]
 is contained in finitely many proper subspaces of $K^{n}$.
 \end{customtheorem}

 The following height estimates will be useful for us.

\begin{lemma}\label{upper bound for height of polynomial}(\cite[Lemma~3.7]{Waldschmidt-book}) For $P\in \mathbb Z[X_1,\cdots,X_n]$, we set $L(P)$ to be the sum of the absolute values of the coefficients of $P$.  Then for any non-zero $(x_1,\cdots,x_n)\in K^n$, we have
    \[\mathrm{H}(P(x_1,\ldots,x_n))\leq L(P)\prod_{i=1}^n \mathrm{H}(x_i)^{\deg_{X_i}P}.\]
    In particular, we have \[\mathrm{H}(x_1x_2)\leq \mathrm{H}(x_1)\mathrm{H}(x_2)\text{ and }\mathrm{H}(x_1+\cdots +x_n)\leq n\prod_{i=1}^n\mathrm{H}(x_i).\]
\end{lemma}

 \begin{lemma}(\cite[Proposition~2.3]{Lenstra-1999})\label{Lemma: a lemma of Lenstra}
    Let $K$ be a number field. Let $P \in K[x]$ be a polynomial with at most $k+1$ terms. Assume that $P$ can be written as the sum of two polynomials $g$ and $h$ with  $\deg g\leq d_1$ and every monomial in $h$ has degree at least $d_2$. Let $\beta$ be a root of $P$ that is not a root of unity. If
\[d_2 - d_1 > \frac{\log\left(k\mathrm{H}(P)\right)}{\log \mathrm{H}(\beta)},\]
then $\beta$ is a common root of $g$ and $h$.
 \end{lemma}

 Throughout the rest of this article, we use Vinogradov's notation $\ll$, $\gg$, and $\asymp$, where the implied constants are positive and independent of $n$.

\subsection{Proof of Theorem~\ref{Theorem: dichotomy}}

We now have the necessary tools to establish Theorem~\ref{Theorem: dichotomy}.
\begin{proof}[Proof of Theorem~\ref{Theorem: dichotomy}]
We assume that $\xi=\sum_{i\geq0}a_i\beta^{-i}\in K_\mathsf{v}$ is algebraic, and we need to prove that $\xi\in K$. We assume the contrary that $\xi\not\in K$.

 By enlarging $\mathcal S$, we assume that $\mathcal{S}$ contains $M_{K}^\infty$ and $\{\mathsf{w}\in M_{K}^\mathrm{fin}:|\beta|_\mathsf{w}\neq1\}$. In particular, $\mathsf{v}\in \mathcal{S}$. We fix an extension of $|\cdot|_\mathsf{v}$ to $\overline{\mathbb Q}$.

Let $\epsilon, \tau,$ and $\epsilon'$ be small positive real numbers that will be specified later. The parameter $\epsilon'$ will be required subsequently in~\eqref{eq: first inequality for Subspace Theorem}. For clarity, we emphasize that the order of selection is first $\epsilon$, then $\tau$, and finally $\epsilon'$.

Since $\mathbf{Rdio}(\mathbf{a})>\frac{\log \mathrm{H}(\beta)}{\log |\beta|_{\mathsf{v}}}$, there exists $\rho>1$ such that $\mathbf{Rdio}(\mathbf{a})>\rho$ and $\mathrm{H}(\beta)<|\beta|_{\mathsf{v}}^\rho$. In view of Lemma~\ref{Lemma: gap between sn and ln1}, by choosing a smaller $\epsilon$ if necessary, 
there are sequences $(r_n)_{n\geq0}$, $(s_n)_{n\geq0}$, $(t_n)_{n\geq0}$ and some $\delta\geq0$ satisfying the gap conditions in Lemma~\ref{Lemma: gap between sn and ln1} with respect to $\rho$, $\epsilon$, and $\delta$.

We first require that $e^{\tau}|\beta|_\mathsf{v}<1$ and  $\frac{\epsilon(1-1/\rho)}{2(2+\delta)}>\frac{\tau|\mathcal{S}|}{\mathrm{H}(\beta)}$. Our computation will be based on the following simple observation. 
\begin{lemma}
    We have $|a_i|_\mathsf{w}\ll e^{\tau i}$ for all $\mathsf{w}\in \mathcal{S}$. In particular, the sum defining $\xi$ converges in $K_\mathsf{v}$.
\end{lemma}
\begin{proof}[Proof of the lemma]
This follows from $\mathrm{h}(a_i)=o(i)$ and $|a_i|_\mathsf{w}\leq \mathrm{H}(a_i)$.
\end{proof}

Recall that $\{1\leq i\leq t_n-s_n:a_{i+r_n}\neq a_{i+s_n}\}\subseteq\bigcup_{j=1}^\delta I_{n,j}$. Thus we have 
\begin{align*}
    \xi\beta^{s_n}-\xi\beta^{r_n}&=\sum_{i=0}^{s_n}a_i\beta^{s_n-i}-\sum_{i=0}^{r_n}a_i\beta^{r_n-i}\\&+\sum_{j=1}^\delta\sum_{i\in I_{n,j}}(a_{i+s_n}-a_{i+r_n})\beta^{-i}+\sum_{i=t_n-s_n+1}^{\infty}(a_{i+s_n}-a_{i+r_n})\beta^{-i}.
\end{align*}
It follows that
\begin{align}\label{eq: compare alpha sn with alpha rn}
\Bigg|\xi\beta^{s_n}-\xi\beta^{r_n}-\sum_{i=0}^{s_n}a_i\beta&^{s_n-i}+\sum_{i=0}^{r_n}a_i\beta^{r_n-i}-\sum_{j=1}^\delta\sum_{i\in I_{n,j}}(a_{i+s_n}-a_{i+r_n})\beta^{-i}\Bigg|_\mathsf{v}\\&=\Bigg|\sum_{i=t_n-s_n+1}^{\infty}(a_{i+s_n}-a_{i+r_n})\beta^{-i}\Bigg|_\mathsf{v}\ll e^{\tau t_n}|\beta|_\mathsf{v}^{s_n-t_n}.    \notag
\end{align}
   We set \[x_{n,1}=\beta^{s_n},x_{n,2}=\beta^{r_n},x_{n,3}=\sum_{i=0}^{s_n}a_i\beta^{s_n-i}-\sum_{i=0}^{r_n}a_i\beta^{r_n-i}\]
   and 
   for $j=1,\ldots,\delta$, we set \[x_{n,3+j}=\sum_{i\in I_{n,j}}(a_{i+s_n}-a_{i+r_n})\beta^{-i}.\] If $\delta = 0$, we set $x_{n,3+j} = 0$. We note for later use that \[{x_{n,3}}{\beta^{-s_n}}=\sum_{i=0}^{s_n}{a_i}{\beta^{-i}}-\beta^{r_n-s_n}\sum_{i=0}^{r_n}{a_i}{\beta^{-i}},\] hence $\lim{x_{n,3}}{\beta^{-s_n}}=\xi$  since $ s_n-r_n$ tends to $\infty$.

Replacing $\mathbb N$ by a suitable infinite subset, we may and do assume that there exists a subset $\mathcal{J}\subseteq\{1,\ldots,\delta\}$, which might be empty,  such that for every $n\in\mathbb N$, $x_{n,3+j}\neq0$ if and only if $j\in\mathcal{J}$.  We consider the following vectors \[\mathbf{x}_n=(x_{n,1},x_{n,2},x_{n,3+j}:j\in \mathcal{J})\in\mathcal{O}_{K,\mathcal{S}}^{3+|\mathcal{J}|}\]
 and the following linear forms
   \begin{itemize}
       \item $L_{3,\mathsf{w}}=\xi X_1-\xi X_2-X_3-\sum_{j\in \mathcal J}X_j$;
       \item $L_{i,\mathsf{w}}=X_i$ for ${i\in\{1,2,3,3+j:j\in\mathcal{J}\}}$ and $\mathsf{w}\in \mathcal{S}$ such that $(i,\mathsf{w})\neq(3,\mathsf{v})$.
   \end{itemize}   
   
   We observe that $\prod_{\mathsf{w}\in \mathcal{S}}|x_{n,1}|_\mathsf{w}=\prod_{\mathsf{w}\in \mathcal{S}}|x_{n,2}|_\mathsf{w}=1$ by the product formula. For any place $\mathsf{w}\in \mathcal{S}$, using the triangle inequality we have \[|x_{n,3}|_\mathsf{w}\ll s_ne^{\tau s_n}\max\{1, |\beta|_\mathsf{w}\}^{s_n}.\]
It follows that \begin{align}\label{eq: prod xn3}
    \prod_{\mathsf{w}\in \mathcal{S}}|x_{n,3}|_\mathsf{w}\ll s_n^{|\mathcal{S}|}e^{\tau|\mathcal{S}|s_n}\mathrm{H}(\beta)^{ s_n}.
\end{align}
Similarly, for $j\in\mathcal J$, we have
\begin{align*}
   \prod_{\mathsf{w}\in \mathcal{S}}|x_{n,3+j}|_\mathsf{w}&\ll\prod_{\mathsf{w}\in \mathcal{S}}|I_{n,j}|e^{\tau |I_{n,j}|}\max\{1,|\beta|_\mathsf{w}\}^{|I_{n,j}|}\\&\ll |I_{n,j}|^{|\mathcal{S}|}e^{\tau |I_{n,j}||\mathcal{S}|}\mathrm{H}(\beta)^{|I_{n,j}|}. 
\end{align*}
 Since $\sum_{j\in\mathcal{J}}|I_{n,j}|\leq\epsilon(t_n-s_n)<\epsilon t_n$ for $n$ large enough, we have \[\prod_{j\in\mathcal{J}}|I_{n,j}|\leq( {\epsilon t_n})^{|\mathcal{J}|}\leq (\epsilon t_n)^\delta.\] Therefore
 \begin{align}\label{eq: prod xn3+j}
     \prod_{j\in\mathcal{J}}\prod_{\mathsf{w}\in \mathcal{S}}|x_{n,3+j}|_\mathsf{v}&\ll \prod_{j\in\mathcal{J}}|I_{n,j}|^{|\mathcal{S}|}e^{\tau |I_{n,j}||\mathcal{S}|}\mathrm{H}(\beta)^{\sum_{j\in\mathcal{ J}}|I_{n,j}|}\\&\ll( {\epsilon t_n})^{|\mathcal{S}|\delta}e^{\tau\epsilon  |\mathcal{S}|t_n }\mathrm{H}(\beta)^{\epsilon t_n}. \notag
 \end{align}
     Combining~\eqref{eq: prod xn3} and~\eqref{eq: prod xn3+j}, we deduce that, for all $n$ sufficiently large,
   \begin{align}\label{eq: bound above product linear forms}
   &\prod_{\mathsf{w}\in \mathcal{S}}\prod_{i\in {\{1,2,3,3+j:j\in\mathcal{J}\}}}\frac{|L_{i,\mathsf{w}}(\mathbf{x}_n)|_\mathsf{w}}{|\mathbf{x}_n|_\mathsf{w}}\\& \ll s_n^{|\mathcal{S}|}e^{\tau|\mathcal{S}|s_n}  ({\epsilon t_n})^{|\mathcal{S}|\delta}e^{(\tau\epsilon  |\mathcal{S}|+\tau)t_n }\mathrm{H}(\beta)^{s_n+\epsilon t_n}|\beta|_{\mathsf{v}}^{-t_n}\prod_{\mathsf{w}\in \mathcal{S}}|\mathbf{x}_n|_\mathsf{w}^{-3-|\mathcal{J}|}\notag\\
   &\ll s_n^{|\mathcal{S}|}e^{\tau|\mathcal{S}|s_n} ({\epsilon t_n})^{|\mathcal{S}|\delta}e^{(\tau\epsilon  |\mathcal{S}|+\tau)t_n }\mathrm{H}(\beta)^{s_n+\epsilon t_n}|\beta|_{\mathsf{v}}^{-t_n}\mathrm{H}(\mathbf{x}_n)^{-3-|\mathcal{J}|}\prod_{\mathsf{w}\not\in \mathcal{S}}|\mathbf{x}_n|_\mathsf{w}^{3+|\mathcal{J}|}\notag\\
   &\ll s_n^{|\mathcal{S}|}e^{\tau|\mathcal{S}|s_n}  ({\epsilon t_n})^{|\mathcal{S}|\delta}e^{(\tau\epsilon  |\mathcal{S}|+\tau)t_n }\mathrm{H}(\beta)^{s_n+\epsilon t_n}|\beta|_{\mathsf{v}}^{-t_n}\mathrm{H}(\mathbf{x}_n)^{-3-|\mathcal{J}|}\notag
   \end{align}
   since $|\mathbf{x}_n|_\mathsf{w}\leq1$ for all $\mathsf{w}\not\in \mathcal{S}$.
   On the other hand, we note the following estimate.
   \begin{lemma}
       For all $n$ sufficiently large
   \[\mathrm{H}(\mathbf{x}_n)\ll  t_n^{|\mathcal{S}|}e^{\tau|\mathcal{S}|t_n}\mathrm{H}(\beta)^{t_n}.\]
   \end{lemma}
   \begin{proof}[Proof of the lemma]
       This follows from the triangle inequality  and the bound $|a_i|_\mathsf{w}\ll e^{\tau i}$ at every place $\mathsf{w}\in \mathcal{S}$.
   \end{proof}
   
    We want to apply the Subspace Theorem for vectors $\mathbf{x}_n$. Therefore, in view of~\eqref{eq: bound above product linear forms}, we need to find some $\epsilon'>0$ such that
     \begin{align}\label{eq: first inequality for Subspace Theorem}
      s_n^{|\mathcal{S}|}e^{\tau|\mathcal{S}|s_n} ({\epsilon t_n})^{|\mathcal{S}|\delta}e^{(\tau\epsilon  |\mathcal{S}|+\tau)t_n }\mathrm{H}(\beta)^{s_n+\epsilon t_n}&|\beta|_{\mathsf{v}}^{-t_n}\\&\ll t_n^{-\epsilon'|\mathcal{S}|}e^{-\epsilon'\tau|\mathcal{S}|t_n} \mathrm{H}(\beta)^{-\epsilon't_n}\notag
       \end{align}  
   for all $n$ sufficiently large. Recall that $s_n\leq t_n/\rho$ and $\mathrm{H}(\beta)^{1/\rho}|\beta|_{\mathsf{v}}^{-1}<1$ since $H(\beta) < |\beta|_\mathsf{v}^\rho$. Therefore, at the beginning of the proof, along with the previously stated conditions, we choose $\epsilon>0$ such that\[\mathrm{H}(\beta)^{1/\rho}|\beta|_{\mathsf{v}}^{-1}<\mathrm{H}(\beta)^{-\epsilon}.\]Next, we choose $\tau>0$ to satisfy the extra condition\[e^{\tau|\mathcal{S}|/\rho}e^{\tau\epsilon |\mathcal{S}|+\tau }\mathrm{H}(\beta)^{1/\rho}|\beta|_{\mathsf{v}}^{-1}<\mathrm{H}(\beta)^{-\epsilon}.\]Finally, we find some $\epsilon'>0$ such that
\[e^{\tau|\mathcal{S}|/\rho}e^{\tau\epsilon  |\mathcal{S}|+\tau}\mathrm{H}(\beta)^{1/\rho+\epsilon}|\beta|_{\mathsf{v}}^{-1}< e^{-\epsilon'\tau|\mathcal{S}|}\mathrm{H}(\beta)^{-\epsilon'}.\]
It follows that the desired inequality~\eqref{eq: first inequality for Subspace Theorem} holds for all $n$ sufficiently large.

Next, we apply the Subspace Theorem and obtain finitely many hyperplanes in $K^{3+|\mathcal{J}|}$ containing all $\mathbf{x}_n$ when $n$ is sufficiently large. In particular, there exists $b_1,b_2,b_3,b_{3+j}\in K$ with $j\in\mathcal{J}$, not all are zero, such that
\[b_1x_{n,1}+b_2x_{n,2}+b_3x_{n,3}+\sum_{j\in\mathcal{J}}b_{3+j}x_{n,3+j}=0\]for all $n$ lying in an infinite subset $\mathcal{N}$ of $\mathbb N$. 
\begin{claim}\label{Claim: b_i=0 for i>4}
  We have $b_{3+j}=0$ for all $j\in\mathcal{J}$.
\end{claim}
\begin{proof}[Proof of the claim]
We set $P_{n,1}(x)=1$, $P_{n,2}(x)=x^{s_n-r_n}$, \[P_{n,3}(x)=x^{s_n}\left(\sum_{i=0}^{s_n}\frac{a_i}{x^{s_n-i}}-\sum_{i=0}^{r_n}\frac{a_i}{x^{r_n-i}}\right)\]
and for each $j\in\{1,\ldots,\delta\}$, we set \[P_{n,3+j}(x)=x^{s_n}\sum_{i\in I_{n,j}}(a_{i+s_n}-a_{i+r_n})x^i.\]
We set $P_n(x)=b_1P_{n,1}+b_2P_{n,2}+b_3P_{n,3}+\sum_{j\in\mathcal{J}}b_{3+j}P_{n,3+j}$, then $P_n(x)$ is a  polynomial with coefficient in $K$ such that $P_n(\beta)=\beta^{-s_n}(b_1x_{n,1}+b_2x_{n,2}+b_3x_{n,3}+\sum_{j\in\mathcal{J}}b_{3+j}x_{n,3+j})=0$ for  $n\in\mathcal{N}$. We observe that $P_n$ has at most \[1+1+s_n+1+\sum_{j\in\mathcal J}|I_{n,j}|\leq 3+s_n+\epsilon(t_n-s_n)\leq 3+t_n(1/\rho+\epsilon) \]terms. Further, we have the following bound.

\begin{lemma}\label{eq: upper bound for HPn}
    We have $\mathrm{H}(P_n)\ll t_n^{|\mathcal{S}|}e^{\tau|\mathcal{S}|{t_n}}\mathrm{H}(b_1,b_2,b_3,b_{3+j}:j\in\mathcal{J}).$
\end{lemma}
\begin{proof}{Proof of the lemma}
    Again, this follows from the triangle inequality and the bound $|a_i|_\mathsf{w}\ll e^{\tau i}$ for every $\mathsf{w}\in \mathcal{S}$.
\end{proof}
Therefore, we have
\[ \mathrm{H}(P_n)\ll t_n^{|\mathcal{S}|}e^{\tau|\mathcal{S}|{t_n}}.\]
On the other hand, we note that the difference between the maximal degree in $P_{n,3}$ and the minimal degree in $P_{n,4}$ is at least $\frac{\epsilon(t_n-s_n)}{2(2+\delta)}\geq\frac{\epsilon t_n(1-1/\rho)}{2(2+\delta)}$. In addition, for every $j\in\mathcal{J}$, the difference between the maximal degree in $P_{n,3+j}$ and the minimal degree in $P_{n,4+j}$ is at least $\frac{\epsilon(t_n-s_n)}{2(2+\delta)}\geq\frac{\epsilon t_n(1-1/\rho)}{2(2+\delta)}$.

It follows from the input condition $\frac{\epsilon(1-1/\rho)}{2(2+\delta)}>\frac{\tau|\mathcal{S}|}{\mathrm{H}(\beta)}$ and Lemma~\ref{Lemma: a lemma of Lenstra} that for $n\in\mathcal{N}$ sufficiently large, $P_{n,3+j}(\beta)=0$ for all $j\in\mathcal{J}$. Since for every $n\in\mathbb N$, $x_{n,3+j}\neq0$ if and only if $j\in\mathcal{J}$, we deduce that $b_{3+j}=0$ for all $j\in\mathcal{J}$.
\end{proof}

 It follows from the claim that there are $b_1,b_2,b_3\in K$, not all are zero, such that $b_1x_{n,1}+b_2x_{n,2}+b_3x_{n,3}=0$ for infinitely many $n$. In the latter case, by dividing by $\beta^{s_n}$ and letting $n$ go to infinity, we obtain $b_1+b_3\xi=0$. Now, if $b_3=0$, then $b_1=0$, hence $b_2=0$, which is absurd. Thus $b_3\neq0$ and $\xi=\frac{-b_1}{b_3}\in K$, a contradiction.

\end{proof}
\begin{remark}\label{remark: very restricted forms}
The proof shows that if $\xi$ is algebraic, then there exists an infinite subset $\mathcal{N}\subseteq\mathbb N$ and elements $b_1,b_2,b_3$ in $K$, not all zero, such that $b_1 x_{n,1} + b_2 x_{n,2} + b_3 x_{n,3} = 0$ for all $n\in\mathcal{N}$, and that $\xi = \frac{-b_1}{b_3}$. In this case, we have $-\xi x_{n,1} + \frac{b_2}{b_3} x_{n,2} + x_{n,3} = 0$ for $n\in\mathcal{N}$, or equivalently,
\begin{equation}\label{eq: when xi is algebraic}
    \xi = \frac{b_2}{b_3}\beta^{r_n-s_n} + \sum_{i=0}^{s_n}a_i\beta^{-i} - \beta^{r_n-s_n}\sum_{i=0}^{r_n}a_i\beta^{-i}. 
\end{equation} 
We consider two cases.
\begin{itemize}
    \item If $r_n$ tends to infinity as $n\in\mathcal{N}$ tends to infinity, then dividing both sides of~\eqref{eq: when xi is algebraic} by $\beta^{r_n-s_n}$ and letting $n\in\mathcal{N}$ tend to infinity, we deduce that $\frac{b_2}{b_3}=\xi$. Thus,
    \begin{align*}
        \xi&=\frac{\sum_{i=0}^{s_n}a_i\beta^{-i} - \beta^{r_n-s_n}\sum_{i=0}^{r_n}a_i\beta^{-i}}{1-\beta^{r_n-s_n}}\\
        &= \sum_{i=0}^{r_n}a_i\beta^{-i} + \sum_{k=0}^{\infty} \sum_{i=r_n+1}^{s_n} a_i \beta^{-(i + k(s_n-r_n))}.
    \end{align*} 
    Using this observation and Proposition~\ref{Proposition: refined Dio of lacunary}, one can recover the transcendence of lacunary numbers~\cite[Corollary~5]{Corvaja-Zannier-2002}.
    
    \item If the sequence $(r_n)_{n\in\mathcal{N}}$ is bounded, then we may assume that there is some $r$ such that $r_n=r$ for all $n\in\mathcal{N}$. Dividing both sides of~\eqref{eq: when xi is algebraic} by $\beta^{r-s_n}$ and letting $n\in\mathcal{N}$ tend to infinity, we obtain $\frac{b_2}{b_3}=\sum_{i=0}^{r}a_i\beta^{-i}$. Thus,
    \[
        \xi=\sum_{i=0}^{s_n}a_i\beta^{-i}
    \]
    for all $n\in\mathcal{N}$. Combining this observation with Propositions~\ref{Proposition: refined Dio of stuttering} and~\ref{Proposition: echoing word has inf refined Dio}, one obtains the transcendence of the real numbers associated with Sturmian words and $k$-bonacci words (written over $\mathbb{N}$) when $\beta$ is a positive real algebraic number with $|\beta|>1$. This follows because $r_n=0$ in these cases (see Section~\ref{section: applications}).
\end{itemize}
To summarize, if $\xi$ is algebraic, then $\xi$ must take a very restricted form. In practice, if one wants to prove that $\xi$ is transcendental, ad hoc methods must be employed to rule out this specific form.
\end{remark}

\section{Transcendence measure results}\label{section: Transcendence measure}
The goal of this section is to prove Theorem~\ref{Theorem: transcendence measure of refined Dio}. To do so, we require several tools from Diophantine approximation: the Quantitative Subspace Theorem, Siegel’s lemma, and Liouville’s inequality.
\subsection{Diophantine prerequisites}
Throughout, let $K$ be a number field.
\subsubsection{The Quantitative Subspace Theorem}
 To state the Quantitative Subspace Theorem, we need some invariants attached to linear forms. Let $L(\mathbf{X})=\sum_{i=1}^n{\xi_iX_i}$ be a linear form in $n$ variables $\mathbf{X}=(X_1,\cdots,X_n)$ with algebraic coefficients. We set $\mathrm{H}(L)=\mathrm{H}(\xi_1,\cdots,\xi_n)$ (which remains the same when scaling $L$) and $K(L)=K\left(\frac{\xi_1}{\xi_i},\cdots,\frac{\xi_n}{\xi_i}\right)$ where $i$ is any index with $\xi_i\neq0$.

In the following, we fix an extension $|\cdot|_\mathsf{v}$ to $\overline{\mathbb Q}$ for each $\mathsf v$.
\begin{customtheorem}{ES}(The Quantitative Subspace Theorem~\cite{Evertse-1996})\label{The quantify Subspace Theorem}
    Let $n\geq1$, $K$ a number field. Let $\mathcal{S}$ be a finite subset of $M_K$ of cardinality $s>0$. Let $d>0$, $H>0$, and $1>\epsilon>0$. For each $\mathsf{v}\in \mathcal{S}$, let $L_{1,\mathsf{v}},\dots,L_{n,\mathsf{v}}$ be $n$ linearly independent linear forms in  $\mathbf{X}=(X_1,\dots,X_n)$ with algebraic coefficients such that for every $i=1,\cdots,n$ and every $\mathsf{v}\in \mathcal{S}$, we have
    \[[K(L_{i,\mathsf{v}}):K]\leq d \text{ and }\mathrm{H}(L_{i,\mathsf{v}})\leq H.\] 
    Then the set of non-zero solutions $\mathbf{x}=(x_1,\cdots,x_n)\in K^n$ of the inequality
    \[\prod_{\mathsf{v}\in \mathcal{S}}\prod_{i=1}^n\frac{|L_{i,\mathsf{v}}(\mathbf{x})|_\mathsf{v}}{|\mathbf{x}|_\mathsf{v}}\leq\prod_{\mathsf{v}\in \mathcal{S}}|\det(L_{1,\mathsf{v}},\cdots,L_{n,\mathsf{v}})|_\mathsf{v}\mathrm{H}(\mathbf{x})^{-n-\epsilon}\]
    with \[\mathrm{H}(\mathbf{x})>\max\{n^{4n/\epsilon},\sqrt{n}H\}\]
is contained in at most $ c_{n,s,\epsilon} (\log 4d)(\log\log 4d)$ proper subspaces  of $K^n$ where $c_{n,s,\epsilon}$ is a constant depending only on $n,s$ and $\epsilon$.
    
\end{customtheorem}

\begin{remark}\label{remark: H'}
    Here, we note that our choice of height $\mathrm{H}$ is different from the one in~\cite{Evertse-1996}. Let $d_0$ be the degree $[K:\mathbb Q]$. In \textit{loc.~cit.}, the authors define $|\mathbf{x}|_\mathsf{v}=
(|x_1|_\mathsf{v}^{2d_0}+\cdots+|x_n|_\mathsf{v}^{2d_0})^{1/2d_0}$ if $\mathsf{v}$ is a real archimedean place, and  $|\mathbf{x}|_\mathsf{v}=
(|x_1|_\mathsf{v}^{d_0}+\cdots+|x_n|_\mathsf{v}^{d_0})^{1/d_0}$ if $\mathsf{v}$ is a complex archimedean place. So if we denote their height by $\mathrm{H}'$, then $\mathrm{H}(\mathbf{x})\leq \mathrm{H}'(\mathbf{x})\leq \sqrt{n}\mathrm{H}(\mathbf{x})$ where $\mathbf{x}\in K^n$. Using a standard $\epsilon$-argument, the result in \textit{loc.~cit.} and the Northcott property yield Theorem~\ref{The quantify Subspace Theorem}, which is our version of the Quantitative Subspace Theorem for the height $\mathrm{H}$.
\end{remark}

\subsubsection{Siegel's lemma}
We recall the following version of Siegel's lemma (see e.g.~\cite{Bombieri-Vaaler-1983,Roy-Thunder-1995}). Here, we use the absolute Weil height $\mathrm{H}$ and Hadamard's inequality to translate the result in \emph{loc.~cit.}, which uses the height defined by the Euclidean norm at archimedean places, into our desired statement.

\begin{lemma}\label{Lemma: Siegle's lemma}
Let $r$ and $n$ be positive
integers with $r<n$. Consider $r$ homogeneous equations
\[\sum_{j=1}^n\xi_{ij}x_j=0,\text{ for }1\leq i\leq r,\]
in $n$ unknowns $x_1,\ldots,x_n$
with coefficients in a number field $K$. Assume that the rank of the system of equations is $r$. Then
there exists a basis of the solution space $\mathbf{x}_1,\ldots,\mathbf{x}_{n-r}$ in $K^n$ such that
\[\prod_{i=1}^{n-r}\mathrm{H}(\mathbf{x}_i) \leq C^{n-r} n^{r/2} \prod_{i=1}^r\mathrm{H}(\xi_{i1},\ldots,\xi_{in})\]
where $C$ is a constant depending only on $K$.
\end{lemma}

Since we will work over a general algebraic base, we will need the following versions of~\cite[Lemmas~7.4 and~7.5]{Adamczewski-Bugeaud-2010-LMS}. The proofs presented here are similar to those in \textit{loc.~cit.}
\begin{lemma}\label{Lemma: existence of hyperplane 1}
    Let $r<n$ be two positive integers. Suppose that there exist linearly independent vectors $\mathbf{p}_1,\ldots,\mathbf{p}_r$ in $K^n$ such that $\mathrm{H}(\mathbf{p}_1)\leq\cdots\leq\mathrm{H}(\mathbf{p}_r)$. Then there exist a constant $C$ depending only on $K$ and  a hyperplane $\mathcal{H}$ in $K^n$ containing all $\mathbf{p}_i$ such that \[\mathrm{H}(\mathcal{H})\leq C^{n-r} n^{r/2} \mathrm{H}(\mathbf{p}_{r})^r.\]
\end{lemma}
\begin{proof}
  This follows from Siegel's lemma.
\end{proof}
\begin{lemma}\label{Lemma: existence of hyperplane 2}
    Let $n$ and $N$ be two integers such that $N>2^n$. Let $\mathbf{p}_1,\ldots,\mathbf{p}_N$ be non-zero vectors in $K^n$ such that 
\[\mathrm{H}(\mathbf{p}_1)\leq\mathrm{H}(\mathbf{p}_2)\leq\cdots\leq\mathrm{H}(\mathbf{p}_N)\]
    and \[\mathrm{rank}_K(\mathbf{p}_1,\ldots,\mathbf{p}_N)<n.\] Then there exist integers $1\leq j_1<j_2<\cdots<j_l$ with $l\geq N/2^n$, and points $\mathbf{p}_{j_1},\ldots,\mathbf{p}_{j_l}$ belong to the same hyperplane $\mathcal{H}$ in $K^n$ such that
    for the previous constant $C$, we have $\mathrm{H}(\mathcal{H})\leq C^{n}n^n\mathrm{H}(\mathbf p_{j_1})^n$. 
\end{lemma}
 \begin{proof}
     If we denote by $f:\{1,\ldots,n\}\to\{1,\ldots,n-1\}$ the function which sends $k$ to $\mathrm{rank}_K(\mathbf{p}_1,\ldots,\mathbf{p}_{\lfloor N/2^k \rfloor })$, then there exists $2\leq k\leq n$ such that $f(k-1)=f(k)$, i.e., 
     \[\mathrm{rank}_K(\mathbf{p}_1,\ldots,\mathbf{p}_{\lfloor N/2^k\rfloor})=\mathrm{rank}_K(\mathbf{p}_1,\ldots,\mathbf{p}_{\lfloor N/2^{k-1}\rfloor})= r<n.\]
     It follows that there exist $1\leq i_1<i_2<\cdots<i_r\leq \lfloor N/2^k\rfloor$ such that $\mathbf{p}_{i_1},\ldots,\mathbf{p}_{i_r}$ generates $\{\mathbf{p}_1,\ldots,\mathbf{p}_{\lfloor N/2^{k-1}\rfloor}\}$.  Therefore $\mathbf{p}_{i_1},\ldots,\mathbf{p}_{i_r}$ belong to a hyperplane $\mathcal{H}$ in $K^n$ with $\mathrm{H}(\mathcal{H})\leq C^{n-r} n^{r/2} \mathrm{H}(\mathbf{p}_{i_r})^r$.

     Now we set $l=\lfloor N/2^{k-1}\rfloor-\lfloor N/2^{k}\rfloor +1$. For $1\leq m\leq l$, we set $j_m=\lfloor N/2^{k}\rfloor+m-1$. We note that $l\geq N/2^n$ and $j_1\geq i_r$. Thus $\mathbf{p}_{j_1},\ldots,\mathbf{p}_{j_l}$ belong to $\mathcal{H}$ where \[\mathrm{H}(\mathcal{H})\leq C^{n-r} n^{r/2} \mathrm{H}(\mathbf{p}_{i_r})^r\leq C^{n}n^n\mathrm{H}(\mathbf p_{j_1})^n.\]
      as desired.
 \end{proof}

\subsubsection{Liouville's inequality}

The following naive estimates will be useful in proving Theorem~\ref{Theorem: transcendence measure of refined Dio} and in constructing dense approximating sequences in Theorem~\ref{theorem: transcendence measure of certain numbers}.
\begin{theorem}({Liouville's inequality}, see~\cite[\S3.5.1]{Waldschmidt-book})\label{Liouville's inequality}
   For every $x\in K^\times$ and $\mathsf{v}\in M_K$, we have
    $|x|_\mathsf{v}\geq \frac{1}{\mathrm{H}(x)^{d_\mathsf{v}}}.$ 
\end{theorem}

\begin{corollary}(A gap principle)\label{A gap principle}
    For any $\mathsf{v}\in M_K$ and any non-zero $x,y\in K$ such that $x\neq y$, we have
    \[|x-y|_\mathsf{v}\geq \frac{1}{2^{d_\mathsf{v}}\mathrm{H}(x)^{d_\mathsf{v}}\mathrm{H}(y)^{d_\mathsf{v}}}\cdot\]
\end{corollary}
\begin{proof}
   It follows from Theorem~\ref{Liouville's inequality} that
    \[|x-y|_\mathsf{v}\geq \frac{1}{\mathrm{H}(x-y)^{d_\mathsf{v}}}\geq\frac{1}{2^{d_\mathsf{v}}\mathrm{H}(x)^{d_\mathsf{v}}\mathrm{H}(y)^{d_\mathsf{v}}}\cdot\]
\end{proof}
\subsubsection{Absolute Weil height interpretation of Mahler's classification
}\label{section: Multiplicative height interpretation of Mahler's classification}
In 1939, Koksma defined a classification that shares the same spirit as the classification of Mahler. Following~\cite[Chapters~3 and 9]{Bugeaud-ApproximationbyAlgebraicNumbers}, for every $\xi\in \mathbb C$ (resp. $\xi\in K_\mathsf{v}$ for a non-archimedean place $\mathsf{v}$ on a  number field $K$),  one denotes by $\omega_d^*$, for every positive integer $d$, the supremum of $\omega\in \mathbb R$ for which the inequality
\begin{equation}\label{eq: Koksma clasification}
 0<|\xi-\alpha|\leq \mathrm{H}_\mathrm{naive}(\alpha)^{-\omega-1}  \text{ }(\text{resp. }  0<|\xi-\alpha|_\mathsf{v}\leq \mathrm{H}_\mathrm{naive}(\alpha)^{-\omega-1})
\end{equation}
has infinitely many solutions $\alpha\in\overline{\mathbb Q}$ (resp.  algebraic $\alpha\in K_\mathsf{v}$) of degree at most $d.$ Again, one sets \[\omega^*(\xi)=\limsup_{d\to\infty}\frac{\omega_d^*(\xi)}{d}\cdot\]

We have $\omega_d^*(\xi)\in[0,\infty]$ for all $d>0$ and $w(\xi)\in[0,\infty]$. Similar to Mahler's classification, one has the notions of ($p$-adic) \emph{$A^*$-, $S^*$-, $T^*$-, $U^*$-}numbers, and two complex numbers (resp. two numbers in $K_\mathsf{v}$) belonging to two different classes are automatically algebraically independent. The fact that the classifications of Mahler and Koksma coincide can be found in~\cite[Chapters 3 and 9]{Bugeaud-ApproximationbyAlgebraicNumbers}. We need the following comparison between two exponents $w_d$ and $w_d^*$.
\begin{theorem}(\cite[Proposition~3.2, Theorems~3.4,~9.1 and~9.3]{Bugeaud-ApproximationbyAlgebraicNumbers})\label{compare two classifications}
    For any $d>0$ and any $\xi$, we have $\omega_d(\xi)\geq \omega_d^*(\xi).$ Further, if $\xi$ is transcendental, we have \[ \frac{\omega_d(\xi)}{2}\leq \omega_d^*(\xi)\leq \omega_d(\xi).\]
\end{theorem}

Nevertheless, since we will make use of the Quantitative Subspace Theorem, it is more convenient for us to transfer these inequalities into statements involving the absolute Weil height instead of the naive height. We note that for every $\alpha\in\overline{\mathbb Q}$ of degree $d $, one has
\[\frac{\mathrm{H}_\mathrm{naive}(\alpha)}{2^d }\leq \mathrm{H}(\alpha)^d \leq \frac{\mathrm{H}_\mathrm{naive}(\alpha)}{(d +1)^d }\]
by~\cite[Lemma~3.11]{Waldschmidt-book}. It motivates us to define $\omega_d^{**}$ to be the supremum of $\omega\in\mathbb R$ for which the inequality
\begin{equation}\label{eq: classification involving multiplicative heights}
0<|\xi-\alpha|\leq \mathrm{H}(\alpha)^{-d\omega}   \text{ } (\text{resp. }  0<|\xi-\alpha|_\mathsf{v}\leq \mathrm{H}(\alpha)^{-d\omega})
\end{equation}
has infinitely many solutions in $\alpha\in\overline{\mathbb Q}$ (resp. algebraic $\alpha\in K_\mathsf{v}$) of degree at most $d$. Again, we define $\omega^{**}(\xi)=\limsup_{d\to\infty}\frac{\omega_d^{**}(\xi)}{d}$. Correspondingly, we have the notions of ($p$-adic) \emph{$A^{**}$-, $S^{**}$-, $T^{**}$-, and $U^{**}$-numbers}.

\begin{proposition}\label{Proposition: compare two classifications}
For every $d>0$, we have 
\[\frac{\omega^*_d(\xi)+1}{d}\leq \omega_d^{**}(\xi)\leq \omega_d^*(\xi)+1.\]
In particular, when $\xi$ is transcendental, $\omega_d(\xi)$ is finite if and only if $\omega_d(\xi)^*$ is finite, if and only if $\omega_d^{**}(\xi)$ is.
\end{proposition}
\begin{proof}
We prove for $|\cdot|_\mathsf{v}$ where $\mathsf{v}$ is a non-archimedean place on a number field $K$, the proof for $|\cdot|$ on $\mathbb C$ is similar. For $\omega>-1$ satisfying that the inequality~\eqref{eq: Koksma clasification} has infinitely many solutions $\alpha$ of degree $\ell \leq d$, we have 
    \[0<|\xi-\alpha|_\mathsf{v}\leq \mathrm{H}_\mathrm{naive}(\alpha)^{-\omega-1}\leq (\mathrm{H}(\alpha)^\ell (\ell +1)^{\ell})^{-(\omega+1)}\leq \mathrm{H}(\alpha)^{-(\omega+1)}.\]
     It follows that $\frac{\omega^*_d(\xi)+1}{d}\leq \omega^{**}_d(\xi)$ for all $d>0$. In particular, $\omega_d^{**}(\xi)>0$. 

     Now, for $\omega>0$ satisfying that the inequality~\eqref{eq: classification involving multiplicative heights} has infinitely many solutions $\alpha$ of degree $\ell \leq d$, we have 
     \[0<|\xi-\alpha|_\mathsf{v}\leq \mathrm{H}(\alpha)^{-d\omega}\leq \mathrm{H}_\mathrm{naive}(\alpha)^{-d\omega/\ell }2^{d\omega}\leq \mathrm{H}_\mathrm{naive}(\alpha)^{-\omega}2^{d\omega}.\]
     For any $\epsilon>0$, there are only finitely many $\alpha$ of degree at most $d$ such that $2^{d\omega}\leq \mathrm{H}_\mathrm{naive}(\alpha)^\epsilon$. Thus the inequality
     \[0<|\xi-\alpha|_\mathsf{v}\leq \mathrm{H}_\mathrm{naive}(\alpha)^{-(\omega-\epsilon)}\]
     has infinitely many solutions $\alpha$ of degree at most $d$. We deduce that $\omega_d^{**}(\xi)\leq \omega_d^*(\xi)+1$ as wanted. The last assertion follows from Theorem~\ref{compare two classifications}.
\end{proof}
\begin{remark}
 Recall that for any archimedean place $v$ of $K$ corresponding to an embedding $\sigma\colon K\xhookrightarrow{}\mathbb C$, we have $|x|_{\mathsf{v}}=|\sigma(x)|^{d_{\mathsf{v}}/d}$. Thus, when estimating the exponents $\omega_d$, $\omega_d^*$, $\omega_d^{**}$ for any archimedean place $\mathsf{v}$, we can use $|\cdot|_{\mathsf{v}}$ instead of $|\cdot|$, as the resulting exponents differ only by multiplicative constants.
\end{remark}

\subsection{Proof of Theorem~\ref{Theorem: transcendence measure of refined Dio}}

With these results in hand, we are now ready to prove Theorem~\ref{Theorem: transcendence measure of refined Dio}.
\begin{proof}[Proof of Theorem~\ref{Theorem: transcendence measure of refined Dio}]

As in the proof of Theorem~\ref{Theorem: dichotomy}, we may and do assume that  $\mathcal{S}$ contains   $M_{K}^\infty$ and $\{\mathsf{w}\in M_{K}^\mathrm{fin}:|\beta|_\mathsf{w}\neq1\}$. We fix an extension of $|\cdot|_\mathsf{v}$ to $\overline{\mathbb Q}$.

Let $\rho$ be a positive real number in the assumption of Theorem~\ref{Theorem: transcendence measure of refined Dio}, then  $\mathrm{H}(\beta)<|\beta|_{\mathsf{v}}^\rho$. Let $\epsilon$ be a positive number such that \[e^\epsilon\mathrm{H}(\beta)^{1/\rho+\epsilon}|\beta|_{\mathsf{v}}^{-1}<1.\]
Since  $\mathbf{a}$ satisfies $(**)_{\rho}$, there exist sequences $(r_n)_{n\geq0}$, $(s_n)_{n\geq0}$, $(t_n)_{n\geq0}$ and some $\delta\geq0$ such that $\mathbf{a}$ satisfies $(*)_{\rho,\epsilon}$ with respect to such a data and ${r_n}\ll{s_n-r_n}$, $\limsup\frac{t_{n}}{s_n}<\infty$, $\limsup\frac{t_{n+1}}{t_n}<\infty$. It is straightforward to verify that these asymptotic conditions still hold after the modification in the proof of Lemma~\ref{Lemma: gap between sn and ln1}. Thus, by choosing a smaller $\epsilon$ if necessary, we may and do further assume that the sequences $(r_n)_{n\geq0}$, $(s_n)_{n\geq0}$, $(t_n)_{n\geq0}$ satisfy the gap properties in Lemma~\ref{Lemma: gap between sn and ln1} with respect to $\rho,\epsilon$ and $\delta$. 

There exist positive numbers $\tau$ and $\epsilon'$ satisfying the following conditions:
\begin{align*}
    e^{\tau}|\beta|_\mathsf{v}^{-1}<1,
\end{align*}
    \begin{align}\label{eq: condition on tau epsilon}
    \frac{\epsilon(1-1/\rho)}{2(2+\delta)}>\frac{\tau|\mathcal{S}|}{\mathrm{H}(\beta)},
    \end{align}
    \[e^{\tau|\mathcal{S}|/\rho}e^{\tau\epsilon  |\mathcal{S}|+\tau }\mathrm{H}(\beta)^{1/\rho+\epsilon}|\beta|_{\mathsf{v}}^{-1}<1,\]
     \begin{align}\label{eq: conditions on tau epsilon epsilon'}
        e^{\tau|\mathcal{S}|/\rho}e^{\tau\epsilon  |\mathcal{S}|+\tau}\mathrm{H}(\beta)^{1/\rho+\epsilon}|\beta|_{\mathsf{v}}^{-1}<e^{-\epsilon'\tau|\mathcal{S}|} \mathrm{H}(\beta)^{-\epsilon'}.
    \end{align}
Note that the third condition is necessary to ensure the fourth.

As in the proof of Theorem~\ref{Theorem: dichotomy}, we set \[x_{n,1}=\beta^{s_n},x_{n,2}=\beta^{r_n},x_{n,3}=\sum_{i=0}^{s_n}a_i\beta^{s_n-i}-\sum_{i=0}^{r_n}a_i\beta^{r_n-i}\]
and \[x_{n,3+j}=\sum_{i\in I_{n,j}}(a_{i+s_n}-a_{i+r_n})\beta^{-i}\]
   for $j=1,\ldots,\delta$. 
   % For convenience, we denote  $q_n=\sum_{j=1}^\delta x_{n,3+j}.$
   \begin{lemma}\label{lemma: extract subsequence of tn}
       We can extract a subsequence of $(t_{n_i})_{i\geq0}$ such that $t_{n_{i+1}}\geq 2t_{n_i}$ for all $i$ and $\limsup\frac{t_{n_{i+1}}}{t_{n_i}}<\infty$.
   \end{lemma}
\begin{proof}[Proof of the lemma]
  There exists $A>0$ such that $t_{n+1}\leq At_n$ for all $n\geq0$. We set $n_0=0$. For every $i\geq0$, let $n_{i+1}$ be the smallest $n>0$ such that $t_{n_{i+1}}\geq 2t_{n_i}$. Then $t_{n_{i+1}}\leq At_{n_{i+1}-1}\leq 2At_{n_i}$ for all $i\geq0$ as desired. 
\end{proof}
   Thanks to the lemma, we assume in addition that $t_{n+1}\geq2t_n$ for all $n\geq0$. Before going to find transcendence measures, we need some setup regarding height estimates. We list here some necessary estimates, which have already appeared in the proof of Theorem~\ref{Theorem: dichotomy} and follow easily from the triangle inequality and the assumptions that $a_i\in \mathcal{O}_{K,\mathcal{S}}$ and $\mathrm{H}(a_i)=e^{o(i)}$.
   \begin{lemma}
       There exist $c_1>1$ such that
       \begin{itemize}
       \item For all $t>s>r\geq-1$, we have 
       \begin{align}\label{eq: condition on tsr}
           \left|\frac{\sum_{i\geq t-s+1}(a_{i+s}-a_{i+r})\beta^{-i}}{\beta^s-\beta^r}\right|_\mathsf{v}\leq c_1e^{\tau t}|\beta|_\mathsf{v}^{-t};
       \end{align}

           \item For all $n\geq0$, we have
           \begin{align}\label{eq: bounded above |xn3|v}
           \prod_{\mathsf{w}\in \mathcal{S}}|x_{n,3}|_\mathsf{w}\leq c_1s_n^{|\mathcal{S}|}e^{\tau|\mathcal{S}|s_n}\mathrm{H}(\beta)^{ s_n};    
           \end{align}
           \item For all $n\geq0$ and all $1\leq j\leq \delta$, we have\begin{align}\label{eq: bounded above |xn3+j|v}
               \prod_{\mathsf{w}\in \mathcal{S}}|x_{n,3+j}|_\mathsf{w}\leq c_{1}|I_{n,j}|^{|\mathcal{S}|}e^{\tau |I_{n,j}||\mathcal{S}|}\mathrm{H}(\beta)^{|I_{n,j}|};
           \end{align} 
           \item For all $n\geq0$, we have 
\begin{align}\label{eq: bounded above H(xn1,...,xn3+delta)}\mathrm{H}(x_{n,1},\ldots,x_{n,3+\delta})\leq c_1t_n^{|\mathcal{S}|}e^{\tau|\mathcal{S}|t_n}\mathrm{H}(\beta)^{t_n};
           \end{align}
           \item For all $\mathcal{J}\subseteq\{1,\ldots,\delta\}$ and all $\{b_1,b_2,b_3,b_{3+j}:j\in\mathcal{J}\}\in K^{3+|\mathcal{J}|}$, we have 
           \begin{align}\label{eq: bounded above H(Pn)}
              \mathrm{H}(P_n)\leq c_1t_n^{|\mathcal{S}|}e^{\tau|\mathcal{S}|{t_n}}\mathrm{H}(b_1,b_2,b_3,b_{3+j}:j\in\mathcal{J}) 
           \end{align}
 for all $n$,   where\begin{align*}
     P_n(x)=b_1+b_2x^{s_n-r_n}+&b_3x^{s_n}\Bigg(\sum_{i=0}^{s_n}\frac{a_i}{x^{s_n-i}}-\sum_{i=0}^{r_n}\frac{a_i}{x^{r_n-i}}\Bigg)\\+&\sum_{j\in\mathcal{J}}b_{3+j}x^{s_n}\sum_{i\in I_{n,j}}(a_{i+s_n}-a_{i+r_n})x^i;
 \end{align*} 
 % is the polynomial considered in the proof of Claim~\ref{Claim: b_i=0 for i>4};
  \item For all $n\geq0$ we have \begin{align}\label{eq: |alpha|  leq c1}
\sum_{i=n}^{\infty}|a_i|_\mathsf{v}|\beta|_\mathsf{v}^{-i}\leq c_1e^{\tau n}|\beta|_\mathsf{v}^{-n};
        \end{align}
           \item And $c_1\geq C$ where $C$ is the constant in Lemma~\ref{Lemma: existence of hyperplane 2} with respect to the number field $K$.
       \end{itemize}
   \end{lemma}
We note also that
\begin{equation}\label{eq: |xn|w leq 1}
   \prod_{\mathsf{w}\not\in \mathcal{S}} |\mathbf{x}_n|_\mathsf{w}\leq1.
\end{equation}
\begin{remark}
    The inequalities~\eqref{eq: bounded above H(xn1,...,xn3+delta)},~\eqref{eq: bounded above H(Pn)}, and~\eqref{eq: |xn|w leq 1} are the only estimates using the assumption that the $a_i$ are $\mathcal{S}$-integers.
\end{remark}
% We have two cases.

% {\bf The first case: $\lim \frac{t_n}{s_n}=\infty$.} 

% {\bf The second case:  $\lim \frac{t_n}{s_n}<\infty$.} 

Thanks to the hypothesis, we have the following assertions.
\begin{lemma}
    There are $A_1,A_2,A_3>1$ and an integer $N_1>0$ such that
    \begin{itemize}
        \item $A_1s_n\geq t_n\geq\rho s_n$ for all  $n$; 
        \item $s_n\geq A_2r_n$ for all  $n$;
        \item $2t_n\leq t_{n+1}\leq A_3t_{n}$ for all  $n$;
        \item $\sum_{j=1}^\delta|I_{n,j}|<\epsilon t_{n}$ for all  $n\geq N_1$;
        \item We have \begin{align}\label{eq: 1st condition of 1st quantitative subspace}
             \mathrm{H}(\beta)^{2^{N_1-1}(1-1/A_2)/A_1}\geq (3+\delta)^{4(3+\delta)/\epsilon'};
        \end{align}
       
    \end{itemize}
\end{lemma}

We now return to the proof of Theorem~\ref{Theorem: transcendence measure of refined Dio}.  Let $\tau,\epsilon,\epsilon'$, $c_1$, $A_1,A_2,A_3$, $N_1$ be constants chosen as above. In addition, let $N\geq N_1$ be a sufficiently large integer.
\begin{remark}
    During the proof, we may increase $N$ if necessary. The point is that the new $N$ is still independent of $d$. 
\end{remark}
We fix an integer $d\geq1$.  Let $\alpha$ be an algebraic number of degree at most $d$. For now, we do not assume that $\alpha \notin K$; we will return to this condition later, see~\eqref{condition: alpha not in K}.  Let $\chi$ be a real number such that 
    \[|\xi-\alpha|_\mathsf{v}=\mathrm{H}(\alpha)^{-\chi}.\]
We aim to find an upper bound for $\chi$ except for finitely many $\alpha$. Thus we may and do assume that $\chi>1$ and 
    \[(e^{-\tau}|\beta|_\mathsf{v})^{t_{N_1}}\leq \mathrm{H}(\alpha).\]
    Let $\kappa>0$ be the smallest integer such that $\mathrm{H}(\alpha)<(e^{-\tau}|\beta|_\mathsf{v})^{t_{\kappa+1}}$, so we have $\kappa\geq N_1$ and \begin{align}\label{eq: input condition on iota}
    (e^{-\tau}|\beta|_\mathsf{v})^{t_{\kappa}}\leq \mathrm{H}(\alpha)<(e^{-\tau}|\beta|_\mathsf{v})^{t_{\kappa+1}}.    
    \end{align}
Let $M>0$ be the largest integer such that \begin{align}\label{eq: input condition on M}
     \mathrm{H}(\alpha)^{-\chi}<  (e^{\tau}|\beta|_\mathsf{v}^{-1})^{A_3^{M-1}t_\kappa}.
    \end{align} 

\begin{claim}\label{eq: bound chi by A3 and M}
    We have $\chi\leq A_3^{M}$.
\end{claim}
    \begin{proof}
        By~\eqref{eq: input condition on iota} and~\eqref{eq: input condition on M}, we have
\[\mathrm{H}(\alpha)^{\chi}\leq (e^{-\tau}|\beta|_\mathsf{v})^{A_3^{M}t_{\kappa}}\leq  \mathrm{H}(\alpha)^{A_3^{M}},\]
so $\chi\leq A_3^{M}$. 
    \end{proof}
    For $h=0,\ldots,M-1$, we have ${t_{\kappa+h}}\leq {A_3^{M-1}t_\kappa}$. By~\eqref{eq: condition on tsr} we deduce that
\begin{align*}
    \Bigg|\alpha-\frac{x_{\kappa+h,3}}{\beta^{s_{\kappa+h}}-\beta^{r_{\kappa+h}}}&-\frac{\sum_{j=1}^\delta x_{\kappa+h,3+j}}{\beta^{s_{\kappa+h}}-\beta^{r_{\kappa+h}}}\Bigg|_\mathsf{v}\\&\leq |\alpha-\xi|_\mathsf{v}+ \Bigg|\xi-\frac{x_{\kappa+h,3}}{\beta^{s_{\kappa+h}}-\beta^{r_{\kappa+h}}}-\frac{\sum_{j=1}^\delta x_{\kappa+h,3+j}}{\beta^{s_{\kappa+h}}-\beta^{r_{\kappa+h}}}\Bigg|_\mathsf{v}\notag\\
    &< {(e^{\tau}|\beta|_\mathsf{v}^{-1})^{ A_3^{M-1}t_\kappa}}+{c_1}(e^{\tau}|\beta|_\mathsf{v}^{-1})^{t_{\kappa+h}}\leq {2c_1(e^{\tau}|\beta|_\mathsf{v}^{-1})^{t_{\kappa+h}}}.
\end{align*}
Thus we have
\begin{equation}\label{eq: input ineq for alpha beta}
\left|\beta^{s_{\kappa+h}}\alpha-\beta^{r_{\kappa+h}}\alpha-x_{\kappa+h,3}-\sum_{j=1}^\delta x_{\kappa+h,3+j}\right|_\mathsf{v}<{4c_1(e^{\tau}|\beta|_\mathsf{v}^{-1})^{t_{\kappa+h}}}.    
\end{equation}

    We consider vectors $\mathbf{x}_{\kappa+h},0\leq h\leq M-1$, and linear forms 
   \begin{itemize}
       \item $L_{3,\mathsf{v}}=\alpha X_1-\alpha X_2-X_3-\sum_{j\in \mathcal J}X_j$;
       \item $L_{i,\mathsf{w}}=X_i$ for ${i\in\{1,2,\ldots,3+\delta\}}$ and $\mathsf{w}\in \mathcal{S}$ such that $(i,\mathsf{w})\neq(3,\mathsf{v})$.
   \end{itemize}   
    In addition, we divide the set $\{0,\ldots,M-1\}$ into disjoint subsets $\mathbb N_\mathcal{J}$ where $\mathcal{J}$ ranges over subsets of $\{1,\ldots,\delta\}$, including the empty set, such that for every $h\in \mathbb N_\mathcal{J}$, $x_{\kappa+h,3+j}\neq0$ if and only if $j\in \mathcal{J}$. The goal is to give an upper bound for  $\max\{h:h\in \mathbb N_\mathcal{J}\}$, hence for $M$ and $\chi$. We fix, for now, an arbitrary $\mathcal J$.

As in the estimate~\eqref{eq: bound above product linear forms}, we combine~\eqref{eq: bounded above |xn3|v},~\eqref{eq: bounded above |xn3+j|v},~\eqref{eq: |xn|w leq 1} and~\eqref{eq: input ineq for alpha beta} with the fact that $\sum_{j\in \mathcal{J}}|I_{\kappa+h,j}|<\epsilon t_{\kappa+h}$ (recall that $\kappa\geq N_1$) to deduce  for all $h\in \mathbb N_\mathcal{J}$ that
% \label{eq: product of linear forms}
\begin{align*}
    \prod_{\mathsf{w}\in \mathcal{S}}&\prod_{i\in {\{1,2,3,3+j:j\in\mathcal{J}\}}}\frac{|L_{i,\mathsf{w}}(\mathbf{x}_{\kappa+h})|_\mathsf{w}}{|\mathbf{x}_{\kappa+h}|_\mathsf{w}}\\\leq& 4c_1^{\delta+1} s_{\kappa+h}^{|\mathcal{S}|}e^{\tau|\mathcal{S}|s_{\kappa+h}}(\epsilon t_{\kappa+h})^{|\mathcal{S}|\delta}e^{(\tau\epsilon |\mathcal{S}|+\tau)t_{\kappa+h}}\mathrm{H}(\beta)^{s_{\kappa+h}+\epsilon t_{\kappa+h}}|\beta|_{\mathsf{v}}^{-t_{\kappa+h}}\mathrm{H}(\mathbf{x}_{\kappa+h})^{-3-|\mathcal{J}|}.\notag
\end{align*}
On the other hand, we have 
\[\mathrm{H}(\mathbf{x}_{\kappa+h})\geq \mathrm{H}(\beta)^{s_{\kappa+h}-r_{\kappa+h}}\geq \mathrm{H}(\beta)^{t_{\kappa+h}(1-1/A_2)/A_1}\geq \mathrm{H}(\beta)^{2^ht_{\kappa}(1-1/A_2)/A_1}.\]
Further, we have $[K(L_{i,\mathsf{w}}):K]\leq d$ for all $(i,\mathsf{w})$.
In addition, $\mathrm{H}(L_{i,\mathsf{w}})=1$ when $(i,\mathsf{w})\neq(3,\mathsf{v})$, $\mathrm{H}(L_{3,\mathsf{v}})=\mathrm{H}(\alpha)$ and
\[\prod_{\mathsf{w}\in \mathcal{S}}|\det(L_{i,\mathsf{w}})|_\mathsf{w}=1.\]

Thus, to apply the Quantitative Subspace Theorem, we need \begin{align}\label{eq: 1st input of 1st quantitative subspace}
    4c_1^{\delta+1} s_{\kappa+h}^{|\mathcal{S}|}e^{\tau|\mathcal{S}|s_{\kappa+h}}(\epsilon t_{\kappa+h})^{|\mathcal{S}|\delta}&e^{(\tau\epsilon |\mathcal{S}|+\tau)t_{\kappa+h}}\mathrm{H}(\beta)^{s_{\kappa+h}+\epsilon t_{\kappa+h}}|\beta|_{\mathsf{v}}^{-t_{\kappa+h}}\\&\leq c_1^{-\epsilon'}
    t_{\kappa+h}^{-\epsilon'|\mathcal{S}|}e^{-\epsilon'\tau|\mathcal{S}|t_{\kappa+h}}
    \mathrm{H}(\beta)^{-\epsilon't_{\kappa+h}}\notag
\end{align}
and \begin{align}\label{eq: 2nd input of 1st quantitative subspace}
\mathrm{H}(\mathbf{x}_{\kappa+h})>\max\{(3+|\mathcal{J}|)^{4(3+|\mathcal{J}|)/\epsilon'},\sqrt{3+|\mathcal{J}|}\mathrm{H}(\alpha) \}.    
\end{align}
Recall that $e^{\tau|\mathcal{S}|/\rho}e^{\tau\epsilon  |\mathcal{S}|+\tau}e^{\epsilon'\tau|\mathcal{S}|}\mathrm{H}(\beta)^{1/\rho+\epsilon}< \mathrm{H}(\beta)^{-\epsilon'}|\beta|_{\mathsf{v}}$ (see~\eqref{eq: conditions on tau epsilon epsilon'}) and $s_{\kappa+h}\leq t_{\kappa+h}/\rho$, thus the inequality~\eqref{eq: 1st input of 1st quantitative subspace} holds for all $h\geq N$ once we choose $N$ sufficiently large.  Recall also that $\mathrm{H}(\alpha)<(e^{-\tau}|\beta|)^{t_{\kappa+1}}$, so if $N$ is chosen to be sufficiently large from the beginning, the condition~\eqref{eq: 2nd input of 1st quantitative subspace} is also satisfied when $h\geq N$ thanks to~\eqref{eq: 1st condition of 1st quantitative subspace}.

Therefore, we can apply the Quantitative Subspace Theorem to deduce that there are at most
\begin{equation}\label{eq: upper bound for T}
    T=c_{3+|\mathcal{J}|,|\mathcal{S}|,\epsilon}(\log 4d)(\log\log 4d)
\end{equation} linear subspaces of $K^{3+|\mathcal{J}|}$ containing all the vectors $\mathbf{x}_{\kappa+h}$, $h\in \mathbb N_\mathcal{J}$ with $h\geq N_1$, where the constant $c_{3+|\mathcal{J}|,|\mathcal{S}|,\epsilon}$ depends on $3+|\mathcal{J}|,|\mathcal{S}|$ and $\epsilon$.

Considering a linear space $H$ consisting of $L$ points in the set $\{\mathbf{x}_{\kappa+h}:h\in \mathbb N_\mathcal{J},h\geq N\}$, we would like to bound above $L$ in terms of $d$. By Lemma~\ref{Lemma: existence of hyperplane 2}, we  can find $N\leq h_1<h_2<\cdots<h_{l}$ in $\mathbb N_\mathcal{J}$ with $l\geq L/2^{3+|\mathcal{J}|}$
and
a hyperplane $\mathcal{H}$ in $K^{3+|\mathcal{J}|}$ given by \[b_1X_1+b_2X_2+b_3X_3+\sum_{j\in\mathcal{J}}b_{3+j}X_{3+j}=0\] such that 
\begin{align}\label{eq: bounded above mathcal H1}
\mathrm{H}(\mathcal{H})&\leq  c_1^{3+\delta}(3+\delta)^{3+\delta}\mathrm{H}(\mathbf{x}_{\kappa+h_1})^{3+\delta}\\&\leq c_1^{6+\delta}(3+\delta)^{3+\delta} t_{\kappa+h_1}^{(3+\delta)|\mathcal{S}|}e^{\tau(3+\delta) |\mathcal{S}|t_{\kappa+h_1}}\mathrm{H}(\beta)^{(3+\delta)t_{\kappa+h_1}}.  \notag 
\end{align}
Here, we have used~\eqref{eq: bounded above H(xn1,...,xn3+delta)}. Without loss of generality, we may assume that $b_{i_0} = 1$ for some $i_0$, which ensures that $\mathrm{H}(b_i)\leq \mathrm{H}(\mathcal{H})$ for all $i$.
 
To bound above $L$ in terms of $d$, it remains to bound $l$. If $L\leq 2^{3+|\mathcal{J}|} N$, we are done. So we assume that $L>2^{3+|\mathcal{J}|} N$, hence $l>N$. 
\begin{claim}\label{Claim: bi=0 for i>4 quantitative}
When $N$ is chosen to be sufficiently large (independent of $d$), we have $b_{3+j}=0$ for all $j\in\mathcal{J}$.
\end{claim}
\begin{proof}[Proof of the claim]
We keep notation as in the proof of Claim~\ref{Claim: b_i=0 for i>4}. Recall that
$P_{\kappa+h_i}(\beta)=0$ for  $1\leq i\leq l$. Further, $P_{\kappa+h_i}$ has at most $ 3+t_{\kappa+h_i}(1/\rho+\epsilon) $ terms and $\mathrm{H}(P_{\kappa+h_i})$ is bounded above by \[\mathrm{H}(P_{\kappa+h_i})\leq c_1t_{\kappa+h_i}^{|\mathcal{S}|}e^{\tau|\mathcal{S}|t_{\kappa+h_i}}\mathrm{H}(\mathcal{H})\leq t_{\kappa+h_i}^{|\mathcal{S}|} e^{\tau|\mathcal{S}|t_{\kappa+h_i}}c_1^{4+\delta}(3+\delta)^{(3+\delta)}\mathrm{H}(\mathbf{x}_{\kappa+h_1})^{3+\delta}\] due to~\eqref{eq: bounded above H(Pn)} and~\eqref{eq: bounded above mathcal H1}. 
We observe that the right hand side is a product of $t_{\kappa+h_i}^{|\mathcal{S}|}e^{\tau|\mathcal{S}|t_{\kappa+h_i}}$ with a term bounded by a power of $t_{\kappa+h_1}$. 

On the other hand, the difference between the maximal degree in $P_{{\kappa+h_i},3}$ and the minimal degree in $P_{{\kappa+h_i},4}$ is at least $\frac{\epsilon t_{\kappa+h_i}(1-1/\rho)}{2(2+\delta)}$. In addition, for every $j\in\mathcal{J}$, the difference between the maximal degree in $P_{{\kappa+h_i},3+j}$ and the minimal degree in $P_{{\kappa+h_i},4+j}$ is at least $\frac{\epsilon t_{\kappa+h_i}(1-1/\rho)}{2(2+\delta)}$. Thus, thanks to~\eqref{eq: condition on tau epsilon}, if $N$ is chosen to be sufficiently large from the beginning (independent of $d$), then for $i\geq N$ we have

\begin{align*}
\frac{\epsilon t_{\kappa+h_i}(1-1/\rho)}{2(2+\delta)}> \frac{\log(3+t_{\kappa+h_i}(1/\rho+\epsilon))+\log(\mathrm{H}(P_{\kappa+h_i}))}{\log\mathrm{H}(\beta)}\cdot    
\end{align*}

We apply Lemma~\ref{Lemma: a lemma of Lenstra} to deduce that for all  $h_i$ with $i\geq N$, we have \[P_{\kappa+h_i,3+j}(\beta)=0\] for all $j\in\mathcal{J}$, hence  $b_{3+j}=0$ for all $j\in\mathcal{J}$ as wanted.    
\end{proof}
It follows that the hyperplane $\mathcal{H}$ is given by   
\[b_1X_1+b_2X_2+b_3X_3=0\]
for some $b_1,b_2,b_3\in K$, not all are zero, with

\begin{align}\label{eq: uppoer bound for H(mathcal H2)}
\mathrm{H}(b_1),\mathrm{H}(b_2)&,\mathrm{H}(b_3)\leq  c_1^{3+\delta}(3+\delta)^{3+\delta}\mathrm{H}(\mathbf{x}_{\kappa+h_1})^{3+\delta}\\&\leq c_1^{6+\delta}(3+\delta)^{3+\delta} t_{\kappa+h_1}^{(3+\delta)|\mathcal{S}|}e^{\tau(3+\delta) |\mathcal{S}|t_{\kappa+h_1}}\mathrm{H}(\beta)^{(3+\delta)t_{\kappa+h_1}}.  \notag 
\end{align}

From $b_1\beta^{s_{\kappa+h_i}}+b_2\beta^{r_{\kappa+h_i}}+b_3x_{\kappa+h_i,3}=0$ for all $1\leq i\leq l$, we have\begin{align}\label{eq: hyperplane with x iota+ki}
b_1+\frac{b_2}{\beta^{s_{\kappa+h_i}-r_{\kappa+h_i}}}+b_3\alpha+b_3\left(\frac{x_{\kappa+h_i,3}}{\beta^{s_{\kappa+h_i}}}-\alpha\right)=0.    
\end{align}
We have\begin{align}\label{eq: upper bound for y/beta}
    \Bigg|&\frac{b_2}{\beta^{s_{\kappa+h_i}-r_{\kappa+h_i}}}\Bigg|_\mathsf{v}\leq \frac{|b_2|_{\mathsf{v}}}{|\beta|_\mathsf{v}^{s_{\kappa+h_i}(1-1/A_2)}}\\&\leq c_1^{3+\delta}(3+\delta)^{3+\delta}\mathrm{H}(\mathbf{x}_{\kappa+h_1})^{3+\delta}{({e^{\tau}|\beta|_\mathsf{v}^{-1})^{t_{\kappa+h_i}(1-1/A_2)/A_1}}}.\notag
\end{align} 
Now we consider two cases. 

Case 1: If $b_3=0$, then by Theorem~\ref{Liouville's inequality} we have \[\left|\frac{b_2}{\beta^{s_{\kappa+h_i}-r_{\kappa+h_i}}}\right|_\mathsf{v}=|b_1|_\mathsf{v}\geq \frac{1}{\mathrm{H}(b_1)^{d_\mathsf{v}}}\geq\frac{1}{(c_1^{3+\delta}(3+\delta)^{3+\delta}\mathrm{H}(\mathbf{x}_{\kappa+h_1})^{3+\delta})^{d_\mathsf{v}}}\] for all $1\leq i\leq l$. Here, recall that $d_\mathsf{v}=[K_\mathsf{v}:\mathbb Q_\mathsf{v}].$  By plugging $i=l$ into~\eqref{eq: upper bound for y/beta} and using~\eqref{eq: bounded above H(xn1,...,xn3+delta)}, we deduce that $l$ must be bounded by a constant independent of $d$.
    
    Case 2: If $b_3\neq0$, then we will establish a lower bound and an upper bound for $|b_1+b_3\alpha|_{\mathsf{v}}$. 
    
    For the lower bound, we need to assume in addition that\begin{equation}\label{condition: alpha not in K}\tag{$\star$}
    \alpha\not\in K.
\end{equation} The reason is that we would like to have $b_1+b_3\alpha\neq0$ in order to apply Corollary~\ref{A gap principle}. Note that $[ K(\alpha)_\mathsf{v}:\mathbb Q_\mathsf{v}]\leq [K_\mathsf{v}:\mathbb Q_\mathsf{v}][ K (\alpha)_\mathsf{v}:K_\mathsf{v}]\leq d_\mathsf{v}d$ thanks to the extension formula, see e.g.~\cite[\S3.1.5]{Waldschmidt-book}. Therefore, we have
\begin{align}
    \label{eq: lower bound for |x+zgamma|}|b_1+b_3\alpha|_{\mathsf{v}}&\geq \Bigg(\frac{1}{2\mathrm{H}(b_1)\mathrm{H}(b_3)\mathrm{H}(\alpha)}\Bigg)^{d_\mathsf{v}d}\\&\geq\frac{1}{2^{d_\mathsf{v}d}(c_1^{3+\delta}(3+\delta)^{3+\delta}\mathrm{H}(\mathbf{x}_{\kappa+h_1})^{3+\delta})^{2d_\mathsf{v}d}(e^{-\tau}|\beta|_{\mathsf{v}})^{d_\mathsf{v}d t_{\kappa+1}}}\cdot\notag
\end{align}

For the upper bound, it follows from~\eqref{eq: |alpha|  leq c1} and~\eqref{eq: input condition on M}  that
\begin{align*}
\Bigg|\frac{x_{\kappa+h_i}}{\beta^{s_{\kappa+h_i}}}&-\alpha\Bigg|_\mathsf{v}=\Bigg|\sum_{n>s_{\kappa+h_i}}a_n\beta^{-n}-\sum_{n=0}^{r_{\kappa+h_i}}a_i\beta^{r_{\kappa+h_i}-s_{\kappa+h_i}-n}\Bigg|_\mathsf{v}+|\xi-\alpha|_\mathsf{v}\\
&\leq \Bigg|\sum_{n>s_{\kappa+h_i}}a_n\beta^{-n}\Bigg|_\mathsf{v}+\sum_{n=0}^{r_{\kappa+h_i}}|a_n|_\mathsf{v}|\beta|_\mathsf{v}^{s_{\kappa+h_i}/{A_2}-s_{\kappa+h_i}-n}+(e^{\tau}|\beta|_\mathsf{v}^{-1})^{t_{\kappa+h_i}}\\
&\leq c_1(e^{\tau}|\beta|_\mathsf{v}^{-1})^{s_{\kappa+h_i}} +c_1(e^{\tau}|\beta|_\mathsf{v}^{-1})^{s_{\kappa+h_i}(1-1/A_2)}+(e^{\tau}|\beta|_\mathsf{v}^{-1})^{t_{\kappa+h_i}}\\
&\leq (2c_1+1)(e^{\tau}|\beta|_\mathsf{v}^{-1})^{t_{\kappa+h_i}(1-1/A_2)/A_1}. 
\end{align*}

Therefore, from~\eqref{eq: uppoer bound for H(mathcal H2)},~\eqref{eq: hyperplane with x iota+ki} and~\eqref{eq: upper bound for y/beta} we obtain \[|b_1+b_3\alpha|_{\mathsf{v}}\leq c_1^{3+\delta}(3+\delta)^{3+\delta}\mathrm{H}(\mathbf{x}_{\kappa+h_1})^{3+\delta}(2c_1+2)(e^{\tau}|\beta|_{\mathsf{v}}^{-1})^{t_{\kappa+h_i}(1-1/A_2)/A_1}\]for all $1\leq i\leq l$. In particular, for ${i}={{l}}$ we have\begin{align}\label{eq: upper bound for|x+zgamma|}
  |b_1&+b_3\alpha|_{\mathsf{v}}\\&\leq c_1^{3+\delta}(3+\delta)^{3+\delta}\mathrm{H}(\mathbf{x}_{\kappa+h_1})^{3+\delta}(2c_1+2)(e^{\tau}|\beta|_{\mathsf{v}}^{-1})^{t_{\kappa+h_{l}}(1-1/A_2)/A_1}.\notag  
\end{align} 
Since $t_{\kappa+h_{l}}\geq 2^{{l}-1}t_{\kappa+h_1}$, the inequalities~\eqref{eq: lower bound for |x+zgamma|} and~\eqref{eq: upper bound for|x+zgamma|}  imply that $l$ must be bounded by $c_2\log d$ for some $c_2>N$ independent of $d$.

 Therefore, in any case, we have $L\leq 2^{3+\delta}c_2\log d$ for some $c_2$ independent of $d$. It follows that \[|\{h\in \mathbb N_\mathcal{J}:h\geq N\}|\leq T2^{3+\delta}c_2\log d.\]

Let $\mathcal{J}$ run over all subsets of $\{1,\ldots,\delta\}$, we deduce that 
$M\leq c_3$ where $c_3= N+ 2^\delta T2^{3+\delta}c_2\log d$.
By the upper bound~\eqref{eq: upper bound for T} and~\eqref{eq: bound chi by A3 and M}, we conclude that  there exists a constant $c>0$ independent of $d$ such that
\[|\xi-\alpha|_\mathsf{v}\geq\mathrm{H}(\alpha)^{-(2d)^{c(\log4d)(\log\log4d)}}\] for all $\alpha\not\in K$ of degree $d\geq1$ satisfying $\mathrm{H}(\alpha)\geq H$ where $H=(e^{-\tau}|\beta|_\mathsf{v})^{t_{N_1}}$ independent of $d$, as required.
%where $H=(e^{-\tau}|\beta|_\mathsf{v})^{t_{N_1}}$ It follows from Proposition~\ref{Proposition: compare two classifications} that 
\end{proof}

Next, we prove Corollary~\ref{Corollary: transcendence measure of refined Dio} by using the absolute Weil height interpretation of Mahler’s classification.

\begin{proof}[Proof of Corollary~\ref{Corollary: transcendence measure of refined Dio}] By Theorem~\ref{Theorem: dichotomy}, it suffices to assume that $\xi$ is transcendental and not a ($p$-adic) $U_\ell$-number for any $\ell\leq [K:\mathbb Q]$. We must bound $\omega_d(\xi)$ for any $d> [K:\mathbb Q]$.

% show that $\xi$ is not a ($p$-adic) $U_d$-number \footnote{Actually, we can take $H(d)=2^dH^d$ where $H$ is the constant appeared in the conclusion of Theorem~\ref{Theorem: transcendence measure of refined Dio}.} depending on $d$ 
Let $d> [K:\mathbb Q]$. By Theorem~\ref{Theorem: transcendence measure of refined Dio}, and using Theorem~\ref{compare two classifications} and Proposition~\ref{Proposition: compare two classifications}
 to translate to the naive height $\mathrm{H}_{\mathrm{naive}}$, there exist constants $c$ and $H$  independent of $d$ such that  
 the set of algebraic numbers $\alpha$ satisfying $[K:\mathbb Q]<\deg(\alpha)\leq d$ and
 \[0<|\xi-\alpha|_\mathsf{v}<\mathrm{H}_{\mathrm{naive}}(\alpha)^{-(2d)^{c(\log4d)(\log\log4d)}}\]is contained in the set $\{\alpha\in \overline{\mathbb Q} :\deg(\alpha)\leq d, \mathrm{H}(\alpha)<H\}$, which is a finite set by the Northcott property. Now let $c_0$ be any real number such that \[c_0>\max\{\omega_1(\xi),\ldots,\omega_{[K:\mathbb Q]}(\xi),(2d)^{c(\log4d)(\log\log4d)}\}.\] By definition, the set of algebraic numbers $\alpha$ of degree at most $[K:\mathbb Q]$ satisfying the inequality\[0<|\xi-\alpha|_\mathsf{v}<\mathrm{H}_{\mathrm{naive}}(\alpha)^{-c_0}\]is finite.
Therefore, the set of algebraic numbers $\alpha$ of degree at most $d$ satisfying \[0<|\xi-\alpha|_\mathsf{v}<\mathrm{H}_{\mathrm{naive}}(\alpha)^{-c_0}\]is finite. Thus \[\omega_d(\xi)\leq\max\{\omega_1(\xi),\ldots,\omega_{[K:\mathbb Q]}(\xi), (2d)^{c(\log4d)(\log\log4d)}\}.\]  In particular, $\xi$ is not a ($p$-adic) $U_d$-number for any $d> [K:\mathbb Q]$. 
% Therefore, $\omega_d^{**}(\xi)\leq (2d)^{c(\log4d)(\log\log4d)}$. The desired conclusion then follows Theorem~\ref{compare two classifications} and Proposition~\ref{Proposition: compare two classifications}.
\end{proof}

\begin{remark}
Theorem~\ref{Theorem: transcendence measure of refined Dio} and    Corollary~\ref{Corollary: transcendence measure of refined Dio} can be applied to lacunary sequences, Sturmian words, $k$-bonacci words, and words generated by coding rotations on rational intervals. 
    % For example, see the proof of Theorem~\ref{theorem: transcendence measure of certain numbers} below for the verification of Sturmian and $k$-bonacci words.
\end{remark}

To conclude this section, we note that in contrast to the refined Diophantine exponent, transcendental numbers associated with words of infinite Diophantine exponent are always $U$-numbers. This is established in the following result, which is not a consequence of our previous theorems, but rather follows directly from $\mathbf{Dio}$.
\begin{proposition}\label{Lemma: Dio infinite implies U number}  We keep the notation as in Theorem~\ref{Theorem: dichotomy}. If $\mathbf{Dio}(\mathbf{a})=\infty$, then  either $\xi=\sum_{i\geq0}a_i\beta^{-i}$ lies in $K$ or $\xi$ is a ($p$-adic) $U_d$-number for some $1\leq d\leq [K:\mathbb Q]$.
\end{proposition}
\begin{proof}Assume that $\xi\not\in K$.
In view of Theorem~\ref{Theorem: dichotomy}, $\xi$ is transcendental. We need to prove that it is a ($p$-adic) $U_d$-number for some $1\leq d\leq [K:\mathbb Q]$.  By extending $\mathcal S$, we assume that $\beta$ is an $\mathcal{S}$-integer.

 Given any $\omega>0$. Let $\tau$ be a positive real number such that $e^{\tau}|\beta|_\mathsf{v}^{-1}<1$. There is $\rho\geq1$ such that $(e^{\tau}|\beta|_\mathsf{v}^{-1})^{\rho}<(e^{\tau|\mathcal{S}|}\mathrm{H}(\beta))^{-\omega}$. Let $(r_n)_{n\geq0},(s_n)_{n\geq0}$, and  $(t_n)_{n\geq0}$ be sequences with respect to $\rho$ in the data of $\mathbf{Dio}(\mathbf{a})$. We have 
 \[\xi\beta^{s_n}-\xi\beta^{r_n}=\sum_{i=0}^{s_n}a_i\beta^{s_n-i}-\sum_{i=0}^{r_n}a_i\beta^{r_n-i}+\sum_{i=t_n-s_n+1}^{\infty}(a_{i+s_n}-a_{i+r_n})\beta^{-i}.\]
Set $\alpha_n=\frac{\sum_{i=0}^{s_n}a_i\beta^{s_n-i}-\sum_{i=0}^{r_n}a_i\beta^{r_n-i}}{\beta^{s_n}-\beta^{r_n}}\in K$, then
\[|\xi-\alpha_n|_\mathsf{v}\ll (e^{\tau}|\beta|^{-1})^{t_n}.\]
We note that  $\mathrm{H}(\alpha_n)\ll s_n^{|\mathcal{S}|}e^{\tau|\mathcal{S}|s_n} \mathrm{H}(\beta)^{s_n}$ and $t_n\geq \rho s_n$. Therefore
\begin{equation}\label{eq: condition for U number}
    |\xi-\alpha_n|_\mathsf{v}\ll(e^{\tau}|\beta|_\mathsf{v}^{-1})^{\rho s_n}\ll (e^{\tau|\mathcal{S}|}\mathrm{H}(\beta))^{-\omega s_n}\ll \mathrm{H}(\alpha_n)^{-\omega}.
\end{equation}     
     Since all $\alpha_n\in K$ and $\lim\alpha_n=\xi$ is transcendental, there is an infinite subsequence $(\alpha_n)_{n\in \mathcal N}$ such that $\alpha_{n}\neq \alpha_{m}$ for all $n,m\in \mathcal N$ with $n\neq m$. It follows from~\eqref{eq: condition for U number} and Proposition~\ref{Proposition: compare two classifications} that $\xi$ is a $U_d$-number for some $d\leq [K:\mathbb Q]$ as wanted.
\end{proof} 
% \begin{remark}
%    Proposition~\ref{Lemma: Dio infinite implies U number} can be applied to lacunary sequences, Sturmian words and words generated by coding rotations on rational intervals.
% \end{remark}
\begin{remark}
 However,  when $\mathbf{Rdio}$ is infinite, there are examples of numbers that are $U$-numbers (see Theorem~\ref{Theorem: trans meas of lacunary}), as well as examples that are $S$- or $T$-numbers (see Theorems~\ref{Theorem: trans meas of lacunary} and~\ref{theorem: transcendence measure of certain numbers}).
\end{remark}

\section{Transcendence measures of lacunary numbers}\label{section: Transcendence measures of lacunary numbers}
% adapt the method used to prove Theorem~\ref{Theorem: transcendence measure of refined Dio} to
In this section, we study the transcendence measures of {lacunary numbers}, as considered in~\cite{Corvaja-Zannier-2002}, with the aim of proving Theorem~\ref{Theorem: trans meas of lacunary}. We begin by recalling the specific class of lacunary numbers considered in this paper.

\begin{definition}\label{Definition: lacunary numbers}
    Let $K$ be a number field and $w$ be a place on $K$. Let $(u_i)_{i\geq0}$ be an increasing sequence of integers such that $\liminf \frac{u_{i+1}}{u_i}>1.$ Let $\beta$ be an algebraic number with $|\beta|_\mathsf{v}>1$, and non-zero $a_0,a_1,\ldots\in K$ such that $\mathrm{h}(a_i)=o(u_i)$. We define a \emph{lacunary number} as a series of the form
$\xi=\sum_{i\geq0}{a_i}{\beta^{-u_i}}\in K_\mathsf{v}.$
\end{definition}
% Our goal now is to bound the exponents $\omega_d(\xi),d\geq1$.
  By~\cite[Corollary~5]{Corvaja-Zannier-2002}, we know that if $\xi$ is a lacunary number, then $\xi$ is transcendental.  We will follow the strategy as in the proof of Theorem~\ref{Theorem: transcendence measure of refined Dio}. However, note that the condition that $a_i$ be an $\mathcal{S}$-integer is not required in this context; hence, we need to modify the estimates used in that proof. Before proving Theorem~\ref{Theorem: trans meas of lacunary}, we establish several preparatory estimates.

 The hypothesis implies that there is $c_1>1$ such that $u_{i+1}>c_1u_i$ for $i$ sufficiently large. We note that the exponents $\omega_d$ are stable under addition by algebraic numbers; in particular, the class of ($p$-adic) $U_d$-numbers for all $d\geq1$, the class of ($p$-adic) $S$-numbers, and the class of ($p$-adic) $T$-numbers are also stable under addition by algebraic numbers. Therefore, without loss of generality, we assume that $u_{i+1}>c_1u_i$ for all $i\geq0$.

 For  $n\geq0$, we set \[\xi_n=\sum_{i=0}^n{a_i}{\beta^{-u_i}}.\] 
 Our computations will be based on the following estimates.
\begin{lemma}
    For every $n\geq0$, we have $u_0+u_1+\cdots+u_n\leq \frac{c_1u_{n}}{c_1-1}.$
\end{lemma}
\begin{proof}
    This follows from $u_{i}\leq\frac{u_n}{c_1^{n-i}}$ for all $0\leq i\leq n$. 
\end{proof}

\begin{lemma}\label{assympotic growth of H(xiN)} Given any $\tau>0$. There exists a constant $c_2>0$ depending on $\tau$ such that \[\mathrm{H}(\xi_n)\leq c_2e^{\tau u_n} \mathrm{H}(\beta)^{\frac{c_1 u_{n}}{c_1-1}}\] for all $n\geq0$.
\end{lemma}
    \begin{proof} 
 The lemma follows from \[\mathrm{H}(\xi_n)\leq (n+1)\prod_{i=0}^n\mathrm{H}(a_i\beta^{u_i})\leq(n+1)\prod_{i=0}^n\mathrm{H}(a_i)\mathrm{H}(\beta)^{u_i}\leq e^{o(u_n)}\mathrm{H}(\beta)^{\frac{c_1 u_{n}}{c_1-1}}.\]
    \end{proof}
Now we prove Theorem~\ref{Theorem: trans meas of lacunary}.
\begin{proof}[Proof of Theorem~\ref{Theorem: trans meas of lacunary}]
% Without loss of generality, we may and do assume that $u_0>0$.  

If $\limsup \frac{u_{i+1}}{u_i}=\infty$, then we argue as in the proof of Proposition~\ref{Lemma: Dio infinite implies U number} to deduce that $\xi$ is a $U_d$-number for some $d\leq[K:\mathbb Q]$.

 It remains to consider the case $\limsup \frac{u_{i+1}}{u_i}<\infty$, i.e., there exists $c_0>1$ such that $u_{i+1}<c_0u_i$ for all $i\geq0$. We set $\mathcal{S}=M_{K}^\infty\cup \{\mathsf{w}\in M_{K}^\mathrm{fin}:|\beta|_\mathsf{w}\neq1\}$ and fix an extension of $|\cdot|_\mathsf{v}$ to $\overline{\mathbb Q}$.

 Let $\rho>\frac{\log\mathrm{H}(\beta)}{\log |\beta|_\mathsf{v}}$ and $\epsilon>0$ be sufficiently small. Recall that in the proof of Proposition~\ref{Proposition: refined Dio of lacunary}, the lacunary sequence $\mathbf{a}$ defined by the sequence $(u_i)_{i\geq0}$ satisfies condition $(*)_{\rho,\epsilon}$ with $r_n=u_n$, $s_n=u_{n+1}$, and $t_n=u_{n+n_0}$ for some sufficiently large integer $n_0$, and $\delta=2(n_0-1)$. Furthermore, since $\limsup \frac{u_{i+1}}{u_i}<\infty$, the sequences $(r_n)_{n\geq0}$, $(s_n)_{n\geq0}$, and $(t_n)_{n\geq0}$ satisfy the additional asymptotic conditions of $(**)_{\rho}$.

We note that the set of mismatches is contained in\[\{1\leq i\leq t_n-s_n: i+r_n=u_j \text{ or } i+s_n=u_j \text{ for some } j\}.\]Thus, our intervals of mismatches are given by $I_{n,1}=u_{n+1}-u_n$, $I_{n,2}=u_{n+2}-u_{n+1}$, $I_{n,3}=u_{n+2}-u_{n}$, $I_{n,4}=u_{n+3}-u_{n+1}$, $I_{n,5}=u_{n+3}-u_{n}$, \ldots, $I_{n,\delta}=u_{n+n_0}-u_{n+1}$. Since these $I_{n,j}$ are not necessarily in increasing order, we first need to reorder them. Furthermore, we apply the modification from Lemma~\ref{Lemma: gap between boxes} to obtain new intervals with large distances between them. We do not need to apply the modification from Lemma~\ref{Lemma: gap between sn and ln1}, since we already have a large gap between $r_n$ and the first mismatch:\[ u_{n+1}-u_n\geq (c_1-1)u_n\geq \frac{(c_1-1)(t_n-s_n)}{c_0^{n_0}}\cdot \]For convenience, we still denote the new intervals by $I_{n,j}$ for $1\leq j \leq \delta$.

 We follow the proof of Theorem~\ref{Theorem: transcendence measure of refined Dio} by setting\[x_{n,1}=\beta^{s_n},\quad x_{n,2}=\beta^{r_n},\]\[x_{n,3}=\sum_{i=0}^{s_n}a_i\beta^{s_n-i}-\sum_{i=0}^{r_n}a_i\beta^{r_n-i}=\beta^{s_n}\xi_{n+1}-\beta^{r_n}\xi_n,\]and\[x_{n,3+j}=\sum_{i\in I_{n,j}}(a_{i+s_n}-a_{i+r_n})\beta^{-i}\]for $j=1,\ldots,\delta$. Note that for each $j$, there are at most $\delta$ non-zero coefficients in $x_{n,3+j}$. However, due to the more general condition on the coefficients, the estimates in the proof of Theorem~\ref{Theorem: transcendence measure of refined Dio} need to be modified—namely, the inequalities~\eqref{eq: bounded above H(xn1,...,xn3+delta)},~\eqref{eq: bounded above H(Pn)}, and~\eqref{eq: |xn|w leq 1}. To adjust these estimates, we use the following bounds.

 \begin{lemma}
     Given any $\tau>0$. There exists a constant $c_2>0$ depending on $\tau$ such that for all $n$, the following estimates hold:
    \begin{itemize}
        \item We have \begin{align}\label{eq:upper bound for Hxn lacunary}\mathrm{H}(x_{n,1},\ldots,x_{n,3+\delta})\leq c_2e^{\tau t_{n}}\mathrm{H}(\beta)^{\frac{c_1t_{n}}{c_1-1}};
           \end{align}
\item We have \begin{align}\label{upper bound for prod xnw lacunary}
               \prod_{\mathsf{w}\not\in \mathcal{S}} |\mathbf{x}_n|_\mathsf{w}\leq c_2e^{\tau t_n} \mathrm{H}(\beta)^{\frac{c_1t_{n}}{c_1-1}};
           \end{align}
           
        \item For all $\mathcal{J}\subseteq\{1,\ldots,\delta\}$ and all $\{b_1,b_2,b_3,b_{3+j}:j\in\mathcal{J}\}\in K^{3+|\mathcal{J}|}$, we have 
           \begin{align}\label{eq: upper bound for HPn lacunary}
            \mathrm{H}(P_n)\leq c_2e^{\tau t_{n}}\mathrm{H}(b_1,b_2,b_3,b_{3+j}:j\in\mathcal J)
           \end{align}
 for all $n$,   where $P_n$ is the polynomial considered in~\eqref{eq: upper bound for HPn}.

    \end{itemize}
 \end{lemma}
\begin{proof}
The inequalities~\eqref{eq:upper bound for Hxn lacunary} and~\eqref{eq: upper bound for HPn lacunary} follow from Lemma~\ref{upper bound for height of polynomial} and Lemma~\ref{assympotic growth of H(xiN)}. For the inequality~\eqref{upper bound for prod xnw lacunary}, we have
\begin{align*}
    \prod_{\mathsf{w}\not\in \mathcal{S}}|\mathbf{x}_n|_\mathsf{w}&\leq \prod_{\mathsf{w}\not\in \mathcal{S}}|(\xi_n,1)|_\mathsf{w}\prod_{i=1}^{n_0}\prod_{\mathsf{w}\not\in \mathcal{S}}|(a_{n+i}\beta^{-u_{n+i}},1)|_\mathsf{w}\\&= \prod_{\mathsf{w}\not\in \mathcal{S}}|(\xi_n,1)|_\mathsf{w}\prod_{i=1}^{n_0}\prod_{\mathsf{w}\not\in \mathcal{S}}|(a_{n+i},1)|_\mathsf{w}\\
    &\leq \mathrm{H}(\xi_n)\prod_{i=1}^{n_0}\mathrm{H}(a_{n+i})=\mathrm{H}(\xi_n)e^{o(t_n)}. 
\end{align*}
Here, we have used that $|\beta|_\mathsf{w}=1$ for ${\mathsf{w}\not\in \mathcal{S}}$. The desired inequality then follows from Lemma~\ref{assympotic growth of H(xiN)}.
\end{proof}

We retain the notation from the proof of Theorem~\ref{Theorem: transcendence measure of refined Dio} and proceed similarly until we reach~\eqref{condition: alpha not in K}. Namely, using the estimates~\eqref{eq:upper bound for Hxn lacunary} and~\eqref{upper bound for prod xnw lacunary}, along with a suitable modification of $\epsilon$, $\tau$, $\epsilon'$, $N_1$, and $N$, we are able to apply the Quantitative Subspace Theorem. This yields, for each subset $\mathcal{J}\subseteq\{1,\ldots,\delta\}$, an effective number (denoted by $T$, see~\eqref{eq: upper bound for T}) of linear subspaces $\mathcal{H}$ of $K^{|\mathcal J|+3}$ \[b_1X_1+b_2X_2+b_3X_3+\sum_{j\in\mathcal{J}}b_{3+j}X_{3+j}=0\]  containing all $\mathbf{x}_{\kappa+h}$, $0\leq h\leq M-1$. It remains to find an upper bound for $l$, where $l$ is the number of vectors $\mathbf{x}_{\kappa+h_i}$ lying in a given linear space $\mathcal{H}$. Using Lemma~\ref{Lemma: existence of hyperplane 2}, it suffices to assume that $\mathrm{H}(\mathcal{H})$ is bounded by an exponent of $t_{\kappa+h_1}$, see~\eqref{eq: bounded above mathcal H1}.  With the estimate~\eqref{eq: upper bound for HPn lacunary}, we can argue as in Claim~\ref{Claim: bi=0 for i>4 quantitative} to deduce that either $l$ is bounded by $N$ independent of $d$, or the coefficients $b_i$ of $\mathcal{H}$ vanish for all $i>3$. 

It remains to consider the case where $l>N$. In this case, $\mathcal{H}$ is given by the equation $b_1X_1+b_2X_2+b_3X_3=0$. At this point, we take advantage of the special form of the lacunary sequence to bound $l$. 

\begin{claim}If $N$ is chosen sufficiently large (independent of $d$), then $b_{3}=0$.\end{claim}
\begin{proof}
    Let $\nu>0$ be an integer (independent of $d$) such that $c_1-\frac{1}{c_0^\nu}>1,$
then \[u_{n+1}-u_{n-\nu}\geq \left(c_1-\frac{1}{c_0^\nu}\right)u_n>u_n.\]
Recall that
\[P_n(x)=b_1x^{u_{n+1}}+b_2x^{u_{n}}+b_3\left(x^{u_{n+1}}\sum_{i=0}^{n+1}a_ix^{-u_i}-x^{u_{n}}\sum_{i=0}^{n}a_ix^{-u_i}\right)\in K[x].\]
Then $P_{\kappa+h_i}(\beta)=0$ for all $1\leq i\leq l$. 
For every $n\geq \nu+2$, we consider the following polynomials with coefficient in $K$ \[Q_{n,1}=b_1x^{u_{n+1}}+b_3x^{u_{n+1}}\sum_{i=0}^{n-\nu-2}a_ix^{-u_i}, Q_{n,2}=b_3x^{u_{n+1}}a_{n-\nu-1}x^{-u_{n-\nu-1}},\]
and 
\[Q_{n,3}=b_2x^{u_{n}}+b_3x^{u_{n+1}}\sum_{i=n-\nu}^{n+1}a_ix^{-u_i}-b_3x^{u_n}\sum_{i=0}^na_ix^{-u_i}.\] Then $P_n=Q_{n,1}+Q_{n,2}+Q_{n,3}$. We observe that the minimal degree in $Q_{n,1}$ is $u_{n+1}-u_{n-\nu-2}$, the degree of the monomial $Q_{n,2}$ is ${u_{n+1}}-u_{n-\nu-1}$, and the maximal degree in $Q_{n,3}$ is $u_{n+1}-u_{n-\nu}$ (since $u_{n+1}-u_{n-\nu}>u_n$). Therefore, for $k\in\{1,2\}$, the difference between the minimal degree in $Q_{n,k}$ and  the maximal degree in $Q_{n,k+1}$ is always at least \[ u_{n-\nu-1}-u_{n-\nu-2}\geq \frac{c_1-1}{c_0^{\nu+2}}u_n.\]
Using~\eqref{eq: upper bound for HPn lacunary}, the upper bound~\eqref{eq: bounded above mathcal H1} of $\mathrm{H}(\mathcal{H})$ in terms of an exponent of $t_{\kappa+h_1}$, and Lemma~\ref{Lemma: a lemma of Lenstra} twice (with a suitable \emph{a priori} condition on $\tau$, namely $\frac{c_1-1}{c_0^{\nu+2}}>\tau c_0^{n_0}$) as in the proof of Claim~\ref{Claim: bi=0 for i>4 quantitative}, we deduce that if $N$ is chosen sufficiently large (still independent of $d$), we have $Q_{\kappa+h_l,2}(\beta)=0$, i.e., $b_3=0$, as desired.
\end{proof}

Thanks to the claim, it remains to consider the case \[b_1\beta^{{u_{\kappa+h_i+1}}}+b_2\beta^{{u_{\kappa+h_i}}}=0 \]
for all $1\leq i\leq l$. Thus $b_1\neq0$ and \[\beta^{{u_{\kappa+h_i+1}-u_{\kappa+h_i}}}=\frac{-b_2}{b_1}\cdot\]
Substituting $i=1$ and $i=l$ gives $u_{\kappa+h_1+1}-u_{\kappa+h_1}=u_{\kappa+h_l+1}-u_{\kappa+h_l}$. Therefore $l$ must be bounded by some constant independent of $d$ since \[u_{\kappa+h_l+1}-u_{\kappa+h_l}\geq c_1^{l-2}({c_1-1})u_{\kappa+h_1+1}.\]

In either case, we conclude that $l$ is bounded by some constant independent of $d$, and hence so is $M$ (arguing as in the proof of Theorem~\ref{Theorem: transcendence measure of refined Dio}). Using the upper bound~\eqref{eq: upper bound for T} for $T$ along with Theorem~\ref{compare two classifications} and Proposition~\ref{Proposition: compare two classifications}, we deduce that there is a constant $c>0$, independent of $d$, such that
\[ \omega_d(\xi)\leq (2d)^{c(\log\log 4d)} \]
for all $d\geq 1$. It follows that $\xi$ is a ($p$-adic) $S$- or $T$-number.
\end{proof}

\section{Proof of Theorem~\ref{theorem: transcendence measure of certain numbers}}\label{section: applications} 
In this last section, we prove Theorem~\ref{theorem: transcendence measure of certain numbers}. The approach relies on constructing a dense approximating sequence converging to $\xi$, based on the following criterion for showing that a number is not a $U$-number.

\begin{lemma}(\cite[Lemma~8.1]{Adamczewski-Cassaigne-2006-Compositio})\label{lemma: criteria not to be U number}
    Let $\xi$ be a real number. Assume that there exist positive numbers $s,\eta,\eta'$ and a sequence $(\alpha_n)_{n \geq 0}$ of algebraic numbers of degree at most $d$ such that
\begin{itemize}
    \item[(i)] $\mathrm{H}(\alpha_n)<\mathrm{H}(\alpha_{n+1})<\mathrm{H}(\alpha_n)^s$, and
    \item[(ii)] $\mathrm{H}(\alpha_n)^{-d\eta'}<|\xi-\alpha_n|<\mathrm{H}(\alpha_n)^{-d\eta}$.
\end{itemize}
Then $\xi$ is not a $U_t$-number for any integer $t<\eta$.
\end{lemma}
\begin{remark}
    We note that the statement of~\cite[Lemma~8.1]{Adamczewski-Cassaigne-2006-Compositio} use the naive height $\mathrm{H}_{\mathrm{naive}}$. Thanks to the relation $\frac{\mathrm{H}_\mathrm{naive}(\alpha)}{2^d }\leq \mathrm{H}(\alpha)^d \leq \frac{\mathrm{H}_\mathrm{naive}(\alpha)}{(d +1)^d }$
when $\alpha$ is of degree $d$ (see~\cite[Lemma~3.11]{Waldschmidt-book}), one can transfer to the statement above regarding the absolute Weil height $\mathrm{H}$.
\end{remark}
\begin{remark}
 Lemma~\ref{lemma: criteria not to be U number} is a number field extension of Baker's result~\cite{Baker-1964}.
\end{remark}
\begin{lemma}(\cite[Lemma~8.2]{Adamczewski-Cassaigne-2006-Compositio})\label{lemma: lower bound for |xi-alpha|}
Let $\beta$ be a Pisot or a Salem number of degree $d$. Let $\mathbf{a}=a_0a_1\cdots$ be an infinite word over $\{0,1,\ldots, \lfloor\beta\rfloor\}$ and set $\xi=\sum_{i\geq0}a_i\beta^{-i}$. Let $U$ and $V$ be two finite words over  $\{0,1,\ldots, \lfloor\beta\rfloor\}$. We set $b_0b_1\cdots=UV^\infty$ and $
\alpha = \sum_{i\geq0}b_i\beta^{-i}$. Assume that there exists a positive integer $j \geq |UV|$ satisfying:
\begin{itemize}
    \item[(i)] $a_i = b_i$, for $0 \leq i < j-1$.
    \item[(ii)] $a_{j-1} \neq b_{j-1}$.
\end{itemize}
Then we have $|\xi - \alpha|>\frac{1}{(|V|+1)^{d-1}\beta^{j+|V|+d-2}}.$  
\end{lemma}

\begin{theorem}\label{theorem: criteria to be S T number}
    Let $\beta$ be a Pisot or Salem number of degree $d$. Let $\mathbf{a}$ be an infinite word over $\{0,1,\ldots, \lfloor\beta\rfloor\}$ such that $2<\mathbf{Dio}(\mathbf{a})<\infty$. Assume in addition that there exists $\rho$ with $2<\rho<\mathbf{Dio}(\mathbf{a})$ such that there exist sequences $(r_n)_{n\geq0}, (s_n)_{n\geq0}, (t_n)_{n\geq0}$ defining $\mathbf{Dio}(\mathbf{a})$ with respect to $\rho$ satisfying $\limsup\frac{t_{n+1}}{t_n} < \infty$. Then the real number $\xi=\sum_{i\geq0}a_i\beta^{-i}$ is neither a $U_1$- nor a $U_2$-number.
\end{theorem}
 % and let $(r_n)_{n\geq0}$, $(s_n)_{n\geq0}$, $(t_n)_{n\geq0}$ be such sequences in the data of $\mathbf{Dio}(\mathbf{a})$ with respect to $\rho$
\begin{proof}
 It suffices to assume that $\xi$ is transcendental. We denote $\nu=\mathbf{Dio}(\mathbf{a})$, then $2<\nu<\infty$. The assumption yields that  $\limsup\frac{s_{n+1}}{s_n} < \infty$.  Arguing as in Lemma~\ref{lemma: extract subsequence of tn}, we may and do assume further that there exists $c>1$ such that $s_n<s_{n+1}<cs_n$ for all $n$. We set $j_n$ to be the smallest $j> t_n$ such that $a_{j-1}\neq a_{j-1\bmod (s_n-r_n)}$, then $\rho s_n\leq t_n< j_n\leq \nu s_n+2 $ for all $n$ sufficiently large.  The number generated by $U_nV_n^\infty$ in base $\beta$ is 
\[\alpha_n=\frac{\sum_{i=0}^{s_n}a_i\beta^{s_n-i}-\sum_{i=0}^{r_n}a_i\beta^{r_n-i}}{\beta^{s_n}-\beta^{r_n}}\in \mathbb Q(\beta).\]
\begin{claim}
  We have  $\beta^{\frac{(\rho-1)s_n}{2cd}}\ll\mathrm{H}(\alpha_n)\ll s_n^{d}\beta^{s_n/d}$ where the implied constants are independent of $n$.
\end{claim}
\begin{proof}[Proof of the claim]
Since $\beta$ is a Pisot or Salem number of degree $d$, we have $\mathrm{H}(\beta)=\beta^{1/d}>1$. Thus     \[\mathrm{H}(\alpha_n)\ll s_n^{d}\mathrm{H}(\beta)^{s_n}\ll s_n^{d}\beta^{s_n/d}.\]
Next, we have
\begin{align*}
|\alpha_n-\alpha_{n-1}|&\leq |\xi-\alpha_n|+|\xi-\alpha_{n-1}|\\&\ll \beta^{-t_{n}}+\beta^{-t_{n-1}}\\
&\ll \beta^{-\rho s_n}+ \beta^{-\rho s_{n-1}}\\
&\ll \beta^{-\rho s_{n-1}}.
\end{align*}
 Since  $|\alpha_n-\alpha_{n-1}|\geq \frac{1}{2^d\mathrm{H}(\alpha_n)^d\mathrm{H}(\alpha_{n-1})^d}$ by Corollary~\ref{A gap principle}, it follows that 
\begin{align*}
    \mathrm{H}(\alpha_n)&\geq \frac{1}{2|\alpha_n-\alpha_{n-1}|^{1/d}\mathrm{H}(\alpha_{n-1})}\\&\gg \frac{\beta^{t_{n-1}/d}}{s_{n-1}^d\beta^{s_{n-1}/d}}\\&\gg s_{n-1}^{-d}\beta^{(\rho-1)s_{n-1}/d}\\&\gg\beta^{\frac{(\rho-1)s_n}{2cd}}
\end{align*}
as claimed.
\end{proof}

Now, we apply Lemma~\ref{lemma: lower bound for |xi-alpha|} to $U_nV_n^\infty$ to deduce that \begin{align*}
    |\xi-\alpha_n|&>\frac{1}{(s_n-r_n+1)\beta^{j_n+s_n-r_n}}\\&\geq\frac{1}{(s_n-r_n+1)\beta^{2+(1+\nu)s_n-r_n}}\\&\gg \beta^{-2\nu s_n}\\& \gg \mathrm{H}(\alpha_n)^{-4\nu cd/(\rho-1)}.  
\end{align*}

Next, let $\rho'$ be any real number such that \[\nu>\rho>\rho'>2.\]
Then we have  \[|\xi-\alpha_n|\ll \beta^{-t_n}\ll \beta^{-\rho s_n}\ll \mathrm{H}(\alpha_n)^{-d\rho'}.\]
Therefore, the hypothesis (ii) in Lemma~\ref{lemma: criteria not to be U number} is fulfilled. It remains to extract a subsequence of $(\alpha_n)_{n\geq0}$ so that condition (i) in Lemma~\ref{lemma: criteria not to be U number} is satisfied. We note that \[\mathrm{H}(\alpha_{n+1})\ll s_{n+1}^d\beta^{s_{n+1}/d}\ll c^ds_n^d\beta^{cs_n/d}\ll \mathrm{H}(\alpha_n)^{2c^2/(\rho'-1)}.\]
Again, by arguing as in Lemma~\ref{lemma: extract subsequence of tn}, we can extract a subsequence $(\alpha'_n)_{n\geq0}$ such that \[\mathrm{H}(\alpha'_n)<\mathrm{H}(\alpha'_{n+1})<\mathrm{H}(\alpha'_n)^{2c^2/(\rho'-1)}.\]It then follows from Lemma~\ref{lemma: criteria not to be U number} that $\xi$ is neither a $U_1$- nor a $U_2$-number, since $\rho'>2$.
\end{proof}

\subsection{\texorpdfstring{$k$-bonacci words}{k-bonacci words}}\label{section: transcendence measure for k bonacci}
Now, we prove Theorem~\ref{theorem: transcendence measure of certain numbers} when $\mathbf{a}$ is a $k$-bonacci word over $\{0,1,\ldots,k-1\}$ with $k-1\leq \beta$.  

\begin{proof}[Proof of Theorem~\ref{theorem: transcendence measure of certain numbers} for $k$-bonacci words]
First, we would like to apply Corollary~\ref{Corollary: transcendence measure of refined Dio}. By~\cite[Theorem~12]{Kebis-Luca-Ouaknine-Scoones-Worrell-2025}, $\mathbf{a}$ is an echoing word whose data is defined as follows. Let $\rho\geq1$, then first we choose $n_0$ sufficiently large and  set  $r_n = 0$, $s_n = |\varphi^{n+n_0} (0)|$ with suitable intervals of mismatches $I_{n,j}$ where $\varphi$ is the $k$-bonacci morphism. Now, let $t_n$ be as in the proof of Proposition~\ref{Proposition: echoing word has inf refined Dio}, then $t_n \asymp s_n$ thanks to the property $d(I_{n,j}, I_{n,j+1}) \asymp s_n$. We note that $s_{n+1}\leq \zeta_ks_n$ for all $n$, where $\zeta_k$ is the $k$-bonacci constant defined as a unique positive real root of $x^k-x^{k-1}-\cdots-1$, which is also the spectral radius of the incidence matrix associated with $\varphi$. Thus, we have ${r_n}\ll{s_n-r_n}$,   $\limsup\frac{t_{n}}{s_n}<\infty$ and $\limsup\frac{t_{n+1}}{t_n}<\infty$. It follows that $\mathbf{a}$ satisfies the condition $(**)_{\rho}$ for all $\rho\geq1$. Since $\xi$ is transcendental by~\cite[Theorem~11]{Kebis-Luca-Ouaknine-Scoones-Worrell-2025}, Corollary~\ref{Corollary: transcendence measure of refined Dio} yields that $\xi$ is  either a  $U_d$-number for $1\leq d\leq [\mathbb Q(\beta):\mathbb Q]=2$, or an $S$- or $T$-number.

Next, we note that $\mathbf{Dio}(\mathbf{a}) = 1+\frac{1}{\zeta_k-1}>2$ by~\cite[Proposition~6.15]{Peltomaki-2024}. Furthermore, the proof in \emph{loc.~cit.}\footnote{Namely, in this case, $\mathbf{Dio}(\mathbf{a})$ is equal to the \emph{initial critical exponent} $\mathbf{ice}(\mathbf{a})$ of $\mathbf{a}$, which is defined as the supremum over all $\rho\geq1$ for which there exist arbitrarily long prefixes $V$ of $\mathbf{a}$ such that $V^\rho$ is a prefix of $\mathbf{a}$. Here, we can take the prefixes $V$ to be $\varphi^n(0),n\geq0$.} indicates that $\mathbf{a}$ satisfies the hypotheses of Theorem~\ref{theorem: criteria to be S T number}. Since $k-1 \leq \lfloor\beta\rfloor$, $\xi$ is neither a $U_1$- nor a $U_2$-number. We conclude that $\xi$ is either an $S$- or a $T$-number, as desired.
\end{proof}
% For example,  the Diophantine exponent of the Fibonacci word ($k=2$) is $1+\frac{1+\sqrt{5}}{2}$.
\subsection{Sturmian words}
Next, we apply our results to Sturmian words, which can be defined as follows. We use the arithmetic characterization of Sturmian words over $\{a,b\}$ recalled below.

For $(x,\theta)\in [0,1)\times([0,1]\setminus\mathbb Q)$, we define the word $\mathbf{s}_{\theta,x}=\mathbf{s}_{\theta,x,0}\mathbf{s}_{\theta,x,1}\ldots$  by \[\mathbf{s}_{\theta,x,n}=a \text{ if } \{x+n\theta\}\in[0,\theta) \text{ and } \mathbf{s}_{\theta,x,n}=b\text{ if }\{x+n\theta\}\in[\theta,1),\]
and define the word  $\mathbf{s}'_{\theta,x}=\mathbf{s}'_{\theta,x,0}\mathbf{s}'_{\theta,x,1}\ldots$ by
\[\mathbf{s}'_{\theta,x,n}=a\text{ if }\{x+n\theta\}\in (0,\theta]\text{ and }\mathbf{s}'_{\theta,x,n}=b\text{ if }\{x+n\theta\}\in(\theta,1)\cup\{0\}.\]
The words $\mathbf{s}_{\theta,x}$ and $\mathbf{s}'_{\theta,x}$ are Sturmian words. Conversely, for any Sturmian word $\mathbf{a}$, there exists a unique pair $(x,\theta)\in [0,1)\times([0,1]\setminus\mathbb Q)$ such that $\mathbf{a}=(\mathbf{s}_{\theta,x,n})_{n\geq0}$ or $\mathbf{a}=(\mathbf{s}'_{\theta,x,n})_{n\geq0}.$ The irrational number $\theta$ is called the \emph{slope} of the word, and it corresponds to the frequency of occurrence of the letter $b$ in the word. 
% The real number $x$ is called the \emph{intercept}.

\begin{proof}[Proof of Theorem~\ref{theorem: transcendence measure of certain numbers} for Sturmian words]

Again, we want to apply Corollary~\ref{Corollary: transcendence measure of refined Dio}. Assume that the Sturmian word $\mathbf{a}$ is determined by some $(x,\theta)$. We assume further that $\mathbf{a}$ is of the form $\mathbf{s}_{\theta,x}$; the other case is proved similarly.
By~\cite[Theorem~4]{Luca-Ouaknine-Worrell-2023}, we know that $\mathbf{a}$ is stuttering, so $\xi$ is transcendental by ~\cite[Theorem~5]{Luca-Ouaknine-Worrell-2023} and $\mathbf{Rdio}(\mathbf{a})=\infty$. 

We write $\lVert\cdot\rVert$ for the {distance} of a real number to its nearest integer, then $\lVert q_n\theta\rVert=|q_n\theta-p_n|<\frac{1}{q_n}$  where $\frac{p_n}{q_n}$ is the continued fraction approximant of $\theta$. The construction in~\cite[Theorem~4]{Luca-Ouaknine-Worrell-2023} then shows that the data $(u_n)_{n\geq0},(v_n)_{n\geq0},d$ defining the stuttering word $\mathbf{a}$ are as follows: for any integer $\rho\geq1$, we set $d=4\rho$, and define $(u_n)_{n\geq0}$ to be the subsequence of $(q_n)_{n\geq0}$ such that $\lVert q_n\theta\rVert=q_n\theta-p_n$. Thus either $u_n=q_{2n}$ for all $n$ or $u_n=q_{2n+1}$ for all $n$, and $u_{n+1}\ll u_n$ where the implied constant is independent of $n$, since $\theta$ has bounded partial quotients. Furthermore, the proof in {\it loc.~cit.} also indicates that each $v_n$ is chosen as \[v_n=\max\{v\geq0:|\{0\leq m\leq v:a_m\neq a_{u_n+m}\}|\leq 2d\},\]
i.e., $v_n$ is the largest integer such that the Hamming distance between $\mathbf{a}[0,v_n]$ and $\mathbf{a}[u_n,u_n+v_n]$ is at most $2d$.  
From the choice of the Sturmian words, we know that the index $m\in \{0,\ldots, v_n\}$ satisfies $a_m= a_{u_n+m}$ if and only if one of the following conditions holds:
\begin{itemize}
     \item $\{x+m\theta\}\in[0,\theta)$ and $\{x+(m+u_n)\theta\}\in[0,\theta)$, or equivalently $\{x+m\theta\}\in[0,\theta-\lVert u_n\theta\rVert)$ when $n$ is sufficiently large;
     \item $\{x+m\theta\}\in[\theta,1)$ and $\{x+(m+u_n)\theta\}\in[\theta,1)$, or equivalently $\{x+m\theta\}\in[\theta,1-\lVert u_n\theta\rVert)$ when $n$ is sufficiently large.
\end{itemize}
Therefore, for all $n$ sufficiently large, the index $m\in \{0,\ldots v_n\}$ satisfies $a_m\neq a_{u_n+m}$ if and only if one of the following conditions holds:
\begin{itemize}
     \item $\{x+m\theta\}\in[\theta-\lVert u_n\theta\rVert,\theta)$;
     \item $\{x+m\theta\}\in[1-\lVert u_n\theta\rVert,1)$.
\end{itemize}
 Since $\theta$ is irrational, the equidistribution of $\{n\theta\}$ inside $(0,1)$ yields that $|\{0\leq m\leq v_n:a_m\neq a_{u_n+m}\}|\asymp v_n\lVert u_n\theta\rVert.$ Since $|\{0\leq m\leq v_n:a_m\neq a_{u_n+m}\}|\leq 4\rho$ and $\theta$ has bounded partial quotients, we deduce that $v_n\ll 1/\lVert u_n\theta\rVert\ll u_n$ thanks to Lemma~\ref{eq: equivalences of badly approximable}. Therefore, using the proof of Proposition~\ref{Proposition: refined Dio of stuttering}, we deduce that the word $\mathbf{a}$ satisfies the condition $(**)_{\rho}$. Thus $\xi$ is either a $U_d$-number for some $d\leq 2$, or an $S$- or $T$-number.

Finally, we have $\infty>\mathbf{Dio}(\mathbf{a})>2$ by~\cite[Proposition~4]{Adamczewski-2010}, and the proof in \emph{loc.~cit.}\footnote{More precisely, we have $\mathbf{Dio}(\mathbf{a}) \geq \mathbf{ice}(\mathbf{a}) > 2$. Further, by~\cite[Corollary~3.5]{Berthe-Holton-Zamboni-2006}, the prefix $V$ determining $\mathbf{ice}(\mathbf{a})$ has length $|V| \in \{q_n, q_n - c_n q_{n-1}\}$, where $c_n$ is the $n$-th \emph{Ostrowski digit} associated with the Sturmian word $\mathbf{a}$.} also shows that our $\mathbf{a}$ satisfies the assumptions of Theorem~\ref{theorem: criteria to be S T number}. Therefore, $\xi$ is either an $S$- or a $T$-number.
\end{proof}
% $\xi$ is neither a $U_1$- nor a $U_2$-number. We conclude that
\begin{remark}
Alternatively, in Theorem~\ref{theorem: transcendence measure of certain numbers}, the transcendence of $\xi$ follows from Theorem~\ref{Theorem: dichotomy} (see Remark~\ref{remark: very restricted forms}). Because $\beta$ is a Pisot number, it also follows from a theorem of Schmidt~\cite{Schmidt-1993}, which asserts that the $\beta$-expansion of $x\in[0,1)$ is eventually periodic if and only if $x\in \mathbb{Q}(\beta)$.
\end{remark}

\begin{remark}
    Every quadratic Pisot number is of the form $\frac{p + \sqrt{p^2 - 4q}}{2}$, where $p$ and $q$ are integers satisfying $p \geq 1$, $-p \leq q \leq p - 2$, and $p^2 - 4q$ is not a perfect square. In contrast, there are no quadratic Salem numbers; the minimal degree of a Salem number is $4$.
\end{remark}

\bibliographystyle{plain} % We choose the "plain" reference style
\bibliography{references} % Entries are in the refs.bib file
\end{document}